\documentclass[11pt]{preprint}

\usepackage[margin=2.8cm]{geometry}
\usepackage[T1]{fontenc} 
\usepackage[latin1]{inputenc}
\usepackage[english]{babel}
\usepackage[babel]{csquotes}
\usepackage{cite}
\usepackage{amssymb}
\usepackage{amsmath}
\usepackage{amsthm}
\usepackage{latexsym}
\usepackage{graphicx}
\usepackage{mathrsfs}
\usepackage{bbm}
\usepackage{verbatim}
\usepackage[colorlinks=true, pdfstartview=FitV, linkcolor=colorLink, citecolor=colorCite, urlcolor=colorLink, linktocpage=true]{hyperref}
\usepackage{enumerate, enumitem}
\usepackage{tabularx}

\numberwithin{equation}{section}

\usepackage[usenames,dvipsnames]{color}

\newtheorem{thm}{Theorem}[section]
\newtheorem{cor}[thm]{Corollary}
\newtheorem{lem}[thm]{Lemma}
\newtheorem{prop}[thm]{Proposition}
\newtheorem{defin}[thm]{Definition}
\newtheorem{remark}[thm]{Remark}

\def\OO{\mathcal O}

\def\Td{\OO}
\renewcommand{\d}{{\mathrm d}} 

\usepackage{mathtools}

\DeclareMathOperator*{\essinf}{ess\,inf}

\def\enne{\mathbb{N}}

\def\erre{\mathbb{R}}

\def\P{\mathbb{P}}

\def\E{\mathop{{}\mathbb{E}}}

\def\s{\varsigma}
\def\ind{\mathbbm 1}

\def\cL{\mathscr{L}}
\def\cF{\mathscr{F}}
\def\cB{\mathscr{B}}
\def\eps{\varepsilon}
\def\cP{\mathscr{P}}

\def\OO{\mathcal{O}}

\def\xe{X^x_\varepsilon}
\def\xex{X^x_\varepsilon}
\def\ye{Y_\varepsilon}
\def\yeh{Y_\varepsilon^h}
\def\pe{P^\varepsilon}
\def\ve{v^\varepsilon}

\renewcommand{\d}{{\mathrm d}}

\def\beq{\begin{equation}}
\def\eeq{\end{equation}}

\def\to{\rightarrow}
\def\wto{\rightharpoonup}
\def\wstarto{\stackrel{*}{\rightharpoonup}}
\def\embed{\hookrightarrow}
\def\cembed{\stackrel{c}{\hookrightarrow}}

\def\norm #1{\left\|#1\right\|}

\def\sp #1#2{\left<#1,#2\right>}
\newcommand\ip\sp

\allowdisplaybreaks

\definecolor{colorLink}{RGB}{0,100,162}
\definecolor{colorCite}{RGB}{8,124,100}

\setlist[itemize]{leftmargin=5mm}
\setlist[enumerate]{leftmargin=10mm}

\begin{document}

\title{Strong Feller property, irreducibility,\\
and uniqueness of the invariant measure\\
for stochastic PDEs with degenerate multiplicative noise}

\author{Luca Scarpa$^{1}$, Margherita Zanella$^{2}$}

\institute{Politecnico di Milano, Email: \href{mailto:luca.scarpa@polimi.it}{\color{black} luca.scarpa@polimi.it} \and Politecnico di Milano, Email: \href{mailto: margherita.zanella@polimi.it}{\color{black} margherita.zanella@polimi.it}}

\maketitle

\begin{abstract}
We establish strong Feller property and irreducibility 
for the transition semigroup associated to a 
class of nonlinear stochastic partial differential equations
with multiplicative degenerate noise.
As a by-product, we prove uniqueness of the invariant 
measure under no strong-dissipativity assumptions.
The drift of the equation 
diverges exactly 
where the noise coefficient vanishes,
resulting in a competition between the dissipative 
effects and the degeneracy of the noise.
We propose a method to measure the accumulation 
of the solution towards the potential barriers,
allowing to give rigorous meaning to
the inverse of the degenerate noise coefficient.
From the mathematical perspective, 
this is one of the first contributions in the literature 
establishing strong Feller properties and irreducibility
in the multiplicative degenerate case,
and opens up novel investigation paths in the direction
of regularisation effects and ergodicity 
in the degenerate-noise framework.
From the application perspective, 
the models cover interesting scenarios
in physics, in the context of evolution of relative concentrations of mixtures, under the influence of 
thermodynamically-relevant
potentials of Flory-Huggins type.

\vspace{2mm}

\vspace{2mm}
\end{abstract}

\setcounter{tocdepth}{2}
\tableofcontents

\maketitle

\section{Introduction}
\label{sec:intro}
One of the core research lines in the theory of 
stochastic processes is 
the investigation of smoothing properties
and long-time behaviour of Markov semigroups.
These are families $P=(P_t)_{t\geq0}$
of liner operators acting on
bounded measurable functions over a given separable metric space $\mathcal X$
and play a central role in the analysis of transition dynamics 
for stochastic evolution equations.
In this direction, a key feature
is represented by the well-known strong Feller property,
meaning that 
at every positive time the operator $P_t$ maps
bounded measurable functions into bounded continuous functions.
This provides instant regularisation 
along the flow, and is a crucial 
point in establishing important qualitative 
results, such as regularisation-by-noise phenomena
and ergodicity. A further core trademark arising 
in the study of Markov transition semigroups
is irreducibility, meaning that every open 
subset of $\mathcal X$ is reached with positive probability, 
starting from any given initial point.
Together with the strong Feller property, irreducibility
allows to achieve refined
results in the long-time analysis,  
such as the equivalence of transition probabilities 
and, more importantly, 
uniqueness of invariant measures.

\subsection{State of the art and open problems}
The study of strong Feller properties 
for Markov semigroups has been well developed 
in the last decades. In the finite dimensional 
case, several results have been achieved, 
especially in the context of 
hypoellipticity 
\cite{hor, mal}.
In the infinite dimensional case, 
the importance of the 
strong Feller property is even more evident,
with long-standing knowledge of applications
to potential theory \cite{car, gross} and,
alongside irreducibility, to uniqueness of
invariant measures and ergodicity
\cite{doob, has}.
Nonetheless, 
the infinite dimensional case is 
notably more delicate. 

For stochastic 
evolution equations with additive noise,
the strong Feller property is usually obtained when the stochastic forcing acts on an infinite numbers of modes:
We refer to the 
pioneering contributions
\cite{mas1, mas2, mas3, DPZ1991, FlaMas95} and 
the monographs
\cite{dapratozab, DPZ-erg} for a general overview. More recently, strong Feller property and irreducibility 
have 
been 
dealt with in more general scenarios:
we mention \cite{MS} for perturbed evolutions, 
\cite{BR} for the Banach space setting, 
\cite{LW} for fast-diffusion equations, 
and \cite{HM} for singular SPDEs. 
While
strong Feller property and irreducibility 
have been widely analysed in addressing
uniqueness of invariant measures,
over the past decades also
different techniques have been developed to address such problem. Among them, we mention coupling and generalised coupling methods that require a non-degenerate forcing on the unstable part of the equation and use a change of measure via Girsanov Theorem and the pathwise contractive properties of the dynamics to prove ergodicity; see \cite{GHMR17, KS}. A major breakthrough in the study of uniqueness of invariant measures was the paper by Hairer and Mattingly \cite{HM1} (see also \cite{HM2}), in which the authors introduced the weaker notion of the asymptotic strong Feller property which is sufficient to prove unique ergodicity in the hypoelliptic setting, i.e. where some modes are excited by the noise and nonlinear term propagates the noise to the whole system.

By contrast, the multiplicative noise case is significantly less studied.
In this direction, all contributions in the literature proving uniqueness of the invariant measure via the strong Feller property and irreducibility require a non-degeneracy assumption on the multiplicative noise, for which we refer to the first works \cite{DPEZ,PesZab}.
This allows indeed to give proper sense to 
the inverse of the noise operator and to 
exploit a Bismut-Elworthy-Li formula \cite{EL}
to represent the derivative of the transition 
semigroup.
In terms of the Kolmorogov operator associated
to the equation, non-degeneracy
consists in requiring 
uniform ellipticity in space.
Only recently it has been shown that generalized coupling techniques can be used to treat the case of multiplicative noises that are non-degenerate along the unstable directions; see \cite{FZ23, FerZannew, DPSZ}.

Up to now, when looking at strong Feller properties and irreducibility for stochastic evolution equations with multiplicative noise, all contributions in the literature require that the noise is 
non-degenerate \cite{DPEZ,PesZab}.
No general results are actually 
available so far in the genuine case of degenerate 
multiplicative noise,
for which strong Feller property and irreducibility
remain open in general, let alone uniqueness of invariant measures.

This paper introduces a method to achieve 
strong Feller property and irreducibility 
also for a class of stochastic evolution equations
with genuine degenerate noise of multiplicative type, 
hence also to
establish uniqueness of the invariant measure
under very general assumptions.
Up to the authors knowledge, this is the 
first contribution in the literature 
that answers such questions.

\subsection{Presentation of the model}
We are interested in smoothing properties and long-time behaviour of the Markov transition
semigroup associated to
stochastic partial differential equations with degenerate 
diffusion in the form 
\beq
\label{eq0}
  \d X - \Delta X\,\d t + \Psi_\beta'(X)\, \d t= G_\alpha(X)\,\d W, \qquad X(0)=x,
\eeq
on a bounded domain $\Td$ in $\erre^d$, with 
$W$ being a cylindrical Wiener process on $H=L^2(\Td)$,
endowed with the boundary conditions
\[
  c_d X+c_n \partial_{\bf n}X = 0 \quad\text{on } \partial\OO.
\]
Here, the parameters $c_d,c_n\in\{0,1\}$
satisfy $c_d+c_n=1$ and specify 
the type of boundary condition:
the choice $(c_d,c_n)=(1,0)$ yields
Dirichlet conditions while 
$(c_d,c_n)=(0,1)$ yields
Neumann conditions.
The dependence of the multiplicative 
noise coefficient $G_\alpha$ 
on the state variable $X$ is rendered through 
the superposition operator associated to
a degenerate mobility function of the form
\[
m_\alpha:\erre\to[0,+\infty),
\qquad m_\alpha(r):=\begin{cases}
(1-r^2)^\alpha \quad&\text{if } r\in[-1,1],\\
0 \quad&\text{if } r\in\erre\setminus[-1,1],
\end{cases}
\]
where $\alpha\geq1$ is a parameter taking into account
the degree of degeneracy. 
Roughly speaking, for every point $v\in H$,
the linear operator $G_\alpha(v)$ acts along every 
direction of $H$, possibly in a coloured fashion, 
with intensity tuned through the term $m_\alpha(v)$.
More precisely, 
the multiplicative 
noise operator is defined as
\[
  G_\alpha: H\to\cL(H,H),
  \qquad
  G_\alpha(v)[h]:=\sqrt{m_\alpha(v)}\cdot
  (\operatorname{I}-\Delta)^{-\delta}h, \quad v,h\in H,
\]
where $\delta>0$ keeps track of the roughness of the noise.
Let us highlight that the noise is indeed degenerate, in the sense that 
\[
  \inf_{v\in H}\operatorname{Tr}[G_\alpha^*(v)G_\alpha(v)]=0.
\]
By contrast, the drift $\Psi_\beta'$ is the derivative of a singular potential of the form 
\[
  \Psi_\beta:(-1,1)\to[0,+\infty),
  \qquad\Psi_\beta(r):=\int_0^r\int_0^z\frac1{(1-w^2)^\beta}\,\d w\,\d z, \quad r\in(-1,1),
\]
where $\beta\geq1$ is another parameter that keeps track of the singularity of the drift close to the potential barriers $\pm1$.

The choices of the degenerate 
mobility $m_\alpha$ and of the singular drift $\Psi_\beta'$
are classical and arise naturally
in numerous fields in physics, such as
thermodynamics and phase separation models.
The variable $X$ in \eqref{eq0} typically 
represents the difference of two concentrations 
in mixtures of fluids or alloys, so
is expected to 
take values in the physically-relevant range $[-1,1]$.
This is indeed forced
by the presence of the singularities of $\Psi_\beta'$ at the potential barriers $\pm1$.
The family of singular potentials 
$(\Psi_\beta)_{\beta\geq1}$ introduced above include 
the relevant choices that have been proposed
in the pioneering works \cite{cahn-hill, cahn-hill2, AC1979} on deterministic phase-separation, such as
the well-known Flory-Huggins \cite{Flory42, Huggins41} logarithmic potential 
\[
  \Psi_1(r)=(1+r)\ln(1+r)+(1-r)\ln(1-r), \quad r\in(-1,1).
\]
The family of degenerate mobilities 
$(m_\alpha)_{\alpha\geq1}$ also cover the interesting choices 
required in thermodynamics, see \cite{ell-gar},
and the scaling $\sqrt{m_\alpha}$ appearing in $G_\alpha$
is well-established in stochastic modelling of 
phase-separation theory, see e.g.~the pioneering work \cite{cook}.
A special mention goes to particular
case $\alpha=\beta$, where one recovers the typical 
balance $m_\alpha\Psi''_\beta\equiv1$. In our framework, 
keeping track of the parameters $\alpha$ and $\beta$ separately
allows for even more generality in this sense.
From the mathematical point of view, the equation \eqref{eq0} is a
generalisation of well-studied stochastic models in phase-field theory, for which we refer 
the reader to \cite{scar-SCH, scar-SVCH, scarpa21, Bertacco21, SZ, DPSZ, DPGS}.

\subsection{Heuristic idea}
It is evident that 
the noise operator $G_\alpha$ degenerates exactly 
where the drift $\Psi_\beta'$ diverges, i.e.~in
correspondence of the potential barriers.
This means, at least from a heuristic perspective, that
the evolution described 
by equation \eqref{eq0} is driven by the competition 
between two opposite factors: 
the monotone singular drift $\Psi_\beta'$ tends to 
keep $X$ away from the potential barriers $\pm1$, 
whereas the degenerate noise $G_\alpha$
tends to attract $X$ towards $\pm1$.
The intuitive idea behind our proposed method
is that, 
despite the fact that $X$ might be arbitrarily close to $\pm1$
throughout the evolution,
if one is able to balance these two opposite 
effects then
on the long run the dynamics stabilise.

This intuition naturally calls 
for a rigorous mathematical method 
to measure the accumulation of the solution $X$ towards
the potential barriers. To this end, the idea is to introduce 
the convex subsets 
\[
  \mathcal K_\gamma:=
  \left\{v:\Td\to(-1,1): \;\int_{\Td}\Psi_\gamma(v)<+\infty\right\}, \quad\gamma\geq1,
\]
and to rely on the fact that for
$v\in\mathcal K_\gamma$, the parameter
$\gamma$ yields exactly a degree of the accumulation of
$v$ towards $\pm1$: the higher $\gamma$ is, the less accumulated $v$ is at the potential barriers.
In order to keep track of this in the dynamics of 
equation \eqref{eq0}, we first look for explicit scalings
of the parameters $\alpha,\beta,\gamma$ such that
the convex set $\mathcal K_\gamma$ is invariant for the 
transition semigroup $P$ associated to \eqref{eq0}, 
namely such that if $x\in \mathcal K_\gamma$ then 
$\P(X(t)\in\mathcal K_\gamma\;\forall\,t\geq0)=1$.
Analogously, it will be necessary to keep track of the 
Sobolev space-regularity of the solution $X$
throughout the evolution: this is done by 
prescribing the initial datum $x$ in $H^{2\xi}(\Td)$,
where $\xi\in[0,\frac12)$.
To summarise, the framework of the whole paper
involves five main parameters, $\alpha,\beta,\gamma,\delta,\sigma,\xi$, with the following interpretation:
\begin{itemize}
    \item $\alpha$: degree of degeneracy of the noise;
    \item $\beta$: degree of singularity of the drift;
    \item $\gamma$: accumulation
    at the potential barriers;
    \item $\delta$: range of the noise;
    \item $\sigma$: colour of the noise;
    \item $\xi$: space regularity.
\end{itemize}

\subsection{Main results and technical strategy}
Let us now briefly discuss the main results and 
the technical ideas behind them. The paper presents
four main results, focusing on well-posedness of equation \eqref{eq0},
strong Feller property, irreducibility, and uniqueness of the invariant measure. These are achieved 
by prescribing a specific tuning of the main parameters introduced above.

The first result concerns well-posedness of equation \eqref{eq0} under very general assumptions on the 
coefficients. More precisely, as we have anticipated above, 
the initial datum is assumed to be in $H^{2\xi}(\Td)\cap\mathcal K_\gamma$ and the solution is shown to 
belong to $H^{2\xi}(\Td)\cap\mathcal K_\gamma$
at all times. The main idea to treat the singularity 
of the drift is to rely on 
energy estimates for the potential $\Psi_\gamma$,
and on the fact that the degeneracy of the noise 
helps in controlling second-order contributions 
in It\^o formula.
This technique implicitly requires the noise to be 
sufficiently coloured, hence prescribing 
suitable conditions on the parameter $\delta$.
The well-posedness result identifies 
the metric space $\mathcal X:=\mathcal K_1$ as the natural one to frame the 
evolution of \eqref{eq0}, and 
allows to introduce the 
Markov transition semigroup
as a family $P=(P_t)_{t\geq0}$ of linear operators
acting on bounded measurable functions defined 
on $\mathcal X$. Let us stress that it still holds that
\[
  \inf_{v\in \mathcal X}\operatorname{Tr}[G_\alpha^*(v)G_\alpha(v)]=0,
\]
so even by framing the evolution in the metric space $\mathcal X$, the solution $X$ may approach the potential barriers $\pm1$
indefinitely, and
the action of the noise on $X$
is still degenerate.

The second result of the paper states the strong Feller property for the transition semigroup $P$.
The main argument consists first in studying differentiability 
with respect to the initial datum in suitable topologies:
this is done via careful stochastic maximal regularity arguments and refined energy estimates on the solution.
Secondly, it is rigorously shown that the noise operator 
is invertible in a specific domain, 
even in the degenerate framework. More precisely, 
we prove that there exists a convex set 
$\mathcal D_{\delta}\subset H$, which 
is also invariant for the transition semigroup $P$,
and an operator $G_\alpha^{-1}:\mathcal D_{\delta}
\to\cL(H^{2\delta}(\Td),H)$ such that 
$G_\alpha(x)G_{\alpha}^{-1}(x)=\operatorname{I}_{H^{2\delta}(\Td)}$
and $G_\alpha^{-1}(x)G_{\alpha}(x)=\operatorname{I}_{H}$
for every $x\in \mathcal D_\delta$.
Unlike the classical setting of non-degenerate noise, 
in this case the inverse $G_\alpha^{-1}(x)$
is not uniformly controlled with respect to $x\in\mathcal D_\delta$. Nonetheless, we prove the estimate 
\[
  \|G_\alpha^{-1}(X)\|_{\cL(H^{2\delta}(\Td), H)}
  \lesssim\|\Psi''_{\frac\alpha2}(X)\|_{H^{2\delta}(\Td)},
\]
where the right-hand side can be controlled 
provided that $\gamma$ is sufficiently large
with respect to $\delta$ and $\sigma$.
The proof
strongly relies on energy estimate on $X$, 
hence it requires the noise to be coloured: 
this will force
the space dimension to be $1$ in 
this framework.
Eventually, these considerations
actually allow to obtain uniform estimates 
on the Bismut-Elworthy-Li formula and to infer the required 
strong Feller property. 
This argument ensures 
that the semigroup maps bounded measurable functions
into bounded locally Lipschitz-continuous functions,
which are not necessarily globally Lipschitz
as it happens in the classical setting of non-degenerate 
noise with uniformly-bounded inverse.
However, this is still
enough 
to infer a strong Feller property 
with respect to the metric of $H^{2\delta}(\Td)$.

The third main result concerns irreducibility.
In this direction, since the strong Feller property 
is obtained with respect to the metric of $H^{2\delta}(\Td)$,
it is natural to investigate irreducibility of 
$P$ with respect to the same topology.
Namely, we show that for every initial datum 
$x\in\mathcal K_\gamma\cap H^{2\delta}(\Td)$
and $t>0$, there is 
a strictly positive probability that $X(t)$ belongs to 
any ball of $H^{2\delta}(\Td)$ in 
$\mathcal X$. The proof 
is inspired from an idea 
presented in \cite{PesZab} and \cite{Zhang} in the non-degenerate case, 
and consists is suitably forcing \eqref{eq0}
with an additional drift term in order to
recover irriducibility. However, while in the classical 
non-degenerate case the modified system is trivially equivalent to the original one via Girsanov arguments, here 
the situation is much more delicate. Indeed, 
since the inverse of the noise is not uniformly bounded,
we are not able to rely on the classical Girsanov 
result to prove equivalence of the laws. Nonetheless, the estimate on $G_\alpha^{-1}$
pointed out above
still allows to exploit a weaker version 
of the Girsanov theorem (see \cite{LS}
for the finite dimensional case
and \cite{Ferr} for the infinite dimensional case), which only ensures that the modified system is absolutely continuous with respect to the original one \eqref{eq0}. This is shown to be enough to infer irreducibility.

The final result of the paper is then uniqueness of the invariant measure for the Markov semigroup $P$. This is a classic corollary of the fact that the strong Feller property and irreducibility
are obtained with respect to the same metric,
and is based on showing that the semigroup is regular, in the sense that the transition probabilities are equivalent.
Let us stress that such uniqueness result is 
extremely general, as it does not rely on 
strong dissipativity assumptions,
and definitely uncommon for 
the degenerate-noise case.

\subsection{Structure of the paper}
We eventually summarise here the main structure of the work. Section~\ref{sec:main} contains
the main setting of the work, the assumptions, and the statements of the results.
In Section~\ref{sec:X} we prove the well-posedness result and collect useful estimates on
some approximated equations.
Section~\ref{sec:Y} is focused on the study of differentiability with respect to the initial datum and estimates on the derivatives.
Eventually, 
in Section~\ref{sec:str_fell} we prove the 
strong Feller property, while is Section~\ref{sec:irr} we show irreducibility and uniqueness of the invariant measure.

\section{Setting and main results}
\label{sec:main}

\subsection{Notation}
In the following, we write $A \lesssim B$ for $A,B>0$ if $A \le cB$ for some constant $c>0$
independent of $A$ and $B$. We also use the notation $A \simeq B$ when $A \lesssim B$ and $B \lesssim A$.

For a general Banach space $E$, its topological dual is denoted by $E^*$, while the duality pairing between $E^*$ and $E$ is denoted by $\ip{\cdot}{\cdot}_{E^*,E}$.
	If $E$ is a Hilbert space, then the scalar product of $E$ is denoted by $(\cdot,\cdot)_E$.
	Given two separable Hilbert spaces $E_1$ and $E_2$, the space of linear continuous, 
    Hilbert-Schmidt, and trace-class operators  
	from $E_1$ to $E_2$ are denoted, respectively, 
    by the symbols $\cL(E_1,E_2)$, $\cL^2(E_1,E_2)$, and $\cL^1(E_1,E_2)$.
	For every $s\in[1,+\infty]$, if $(A,\mathcal A, \mu)$
    is a measure space and $E$ is a Banach space, 
	the symbol $L^s(A; E)$ indicates the usual spaces of strongly measurable, Bochner-integrable functions
	from $A$ to $E$.
    If $E$ is omitted, it is understood that $E = \mathbb{R}$.
    For every $T>0$, we denote by $\mathcal C^{1,2}([0,T]\times E)$ the space of all functions $f:[0,T] \times E \rightarrow \mathbb{R}$ which are one-time differentiable w.r.t.~$t \in [0,T]$ and two-times Fr\'echet differentiable 
    w.r.t.~$x \in E$.

\subsection{Functional and probabilistic setting}
Let $d\in\{1,2,3\}$ and $\OO\subset\erre^d$ be
a bounded Lipschitz domain. 
For $p \in [1, + \infty]$ and $s\geq0$ 
we denote by $L^p(\Td)$ and 
$H^{s}(\Td)$ the classical Lebesgue and Sobolev spaces 
of real-valued functions 
on $\Td$, respectively, endowed with their natural norms 
and scalar products.

We also recall the technical lemma in \cite[Thm.~7.4]{BH2021} stated in the form that best fits our needs. 
\begin{lem}[Multiplication of Sobolev functions]
\label{multiplication}
Let $s\ge 0$, $s_1\geq s$, and $s_2\ge s$
be real numbers satisfying
\[
s_1+s_2 >s+ \frac{d}2.
\]
Then, the pointwise multiplication of functions extends uniquely to a continuous bilinear map 
\[
H^{s_1}(\Td) \times H^{s_2}(\Td) \rightarrow H^{s}(\Td),
\]
namely there exists a constant $C$, depending on 
$s, s_1, s_2, d$ such that, for every $f_1\in H^{s_1}(\Td)$ and $f_2\in H^{s_2}(\Td)$ it holds that 
$f_1f_2\in H^{s}(\Td)$ and
\[
\norm{f_1f_2}_{H^{s}(\Td)}\leq C
\norm{f_1}_{H^{s_1}(\Td)}\norm{f_2}_{H^{s_2}(\Td)}.
\]
\end{lem}

Let
    $T > 0$ arbitrary final time, the probabilistic domain of the problem is denoted by  $(\Omega,\cF,(\cF_t)_{t\in[0,T]},\P)$, i.e., a filtered probability space satisfying the usual conditions (namely, the filtration is saturated and right-continuous). For every Banach space $E$ and 
    for every $s\in[1,+\infty]$, 
    the space $L^s(\Omega; E)$ is the classical space of strongly measurable $E$-valued random variables with finite moments up to order $s$, with the usual extension for $p = +\infty$. If the space $E$ is a functional space depending on time, we use the notation $L^s_\cP(\Omega;E)$, to stress that measurability is also intended with respect to the progressive $\sigma$-algebra $\cP$.
    We recall that for all $s\in(1,+\infty)$ and for every separable and reflexive Banach space $E$, 
	the space
	\[
	L^s_w(\Omega; L^\infty(0,T; E^*)):=
	\left\{v:\Omega\to L^\infty(0,T; E^*) \text{ weakly*-measurable, }
	\norm{v}_{L^\infty(0,T; E^*)}\in L^s(\Omega)
	\right\}
	\]
	satisfies
	\cite[Thm.~8.20.3]{edwards} the identification
	\[
	L^s_w(\Omega; L^\infty(0,T; E^*))=
	\left(L^{\frac{s}{s-1}}(\Omega; L^1(0,T; E))\right)^*.
	\]
    The symbol $W$ denotes a $U$-cylindrical Wiener process,
    where $U$ is a fixed separable Hilbert space. 
    Namely, it formally holds that 
	\begin{equation} \label{eq:representation}
		W = \sum_{k=0}^{\infty} \beta_k u_k,
	\end{equation}
	where $\{\beta_k\}_{k \in \enne}$ are real independent Brownian motions and the family $\{u_k\}_{k\in\enne} \subset U$ is a fixed orthonormal system of $U$. The series \eqref{eq:representation} may not be convergent in $U$, 
    but converges in any separable Hilbert 
    space $U_1\supset U$ such that the 
    inclusion $\iota:U\to U_1$ is Hilbert-Schmidt.
    Such space $U_1$ always exists and 
    it is possible then to identify $W$ as a $Q_1$-Wiener process on $U_1$, where $Q_1:=\iota\circ\iota^*\in\cL^1(U_1,U_1)$ is trace-class on $U_1$ and 
    satisfies $Q_1^{\frac12}(U_1)=\iota(U)=U$ (for an exhaustive treatment of the subject, see \cite[Subsec.~2.5.1]{LiuRo}). In order not to burden the notation in this work, we may implicitly assume this extension by simply saying that $W$ is a cylindrical process on $U$. Stochastic integration with respect to the cylindrical process $W$ is defined in terms of
    stochastic integration with respect to $Q_1$-Wiener process, so that for every Hilbert space $E$ and for every 
    process $B\in L^2_\cP(\Omega; L^2(0,T; E))$, the stochastic integral 
	\[
	\int_0^\cdot B(s)\,\d W(s)
	\]
	is well-defined in $L^2_\cP(\Omega; C^0([0,T]; E))$. This definition is well posed and does not depend on the choice of $U_1$ or $\iota$ (cfr.~\cite[Subsec.~2.5.2]{LiuRo}).

\subsection{Operators}
From now on, given $c_d,c_n\in\{0,1\}$ with $c_d+c_n=1$,
we set
	\[
    H:= L^2(\Td), \qquad 
	V:=
    \begin{cases}
    H^1_0(\Td) \quad&\text{if } (c_d,c_n)=(1,0),\\
    H^1(\Td) \quad&\text{if } (c_d,c_n)=(0,1),
    \end{cases}
	\]
	endowed with their standard norms $\norm{\cdot}_H$,
	$\norm{\cdot}_{V}$, respectively.
	As usual, we identify the Hilbert space $H$ with its dual through
	the corresponding Riesz isomorphism, so that we have the variational structure
	$V\embed H \embed V^*$
	with dense an compact embeddings. 

\subsubsection{The linear operator and its fractional powers}
Let $A:=-\Delta$ 
be the realisation of the negative Laplace operator with homogeneous boundary conditions
of either Dirichlet type (if $c_d=1$)
or Neumann type (if $c_n=1$).
Namely, $A$ can be viewed equivalently either as an unbounded linear non-negative operator on $H$ with effective domain 
\[
\mathcal{D}(A):=\begin{cases}
H^1_0(\Td)\cap H^2(\Td) \quad&\text{if } (c_d,c_n)=(1,0),\\
\left\{v\in H^2(\Td):\;\partial_{\bf n}v=0 \text{ a.e.~on } \partial\OO\right\}
\quad&\text{if } (c_d,c_n)=(0,1),
\end{cases}
\]
or as a variational operator 
$A:V\to V^*$ as
	\[
	\ip{A\psi}{\phi}_{V^*,V}=
    \int_{\Td}\nabla\psi\cdot\nabla\phi\,,
	\qquad \psi,\phi\in V.	
    \]
We recall that  $A$ is a
self-adjoint operator with compact resolvent. 
Let also $(e_k)_{k\in\enne_+}$
be 
a complete orthonormal system
of $H$
formed by eigenfunctions of $A$ with associated eigenvalues $(\lambda_k)_{k\in\enne_+}$:
let us recall that $\lambda_k\sim k^{\frac 2d}$
as $k\to\infty$.
Furthermore, 
the operator $-A$ generates 
a strongly continuous analytic
semigroup $S$ of contractions on $H$.

For $s \in (0,1)$, the fractional operator $A^s$ 
is defined in the sense of
spectral theory as
the linear unbounded operator on $H$ with 
effective domain 
\[
\mathcal{D}(A^s):=
\left\{ u \in H \ : \ \sum_{k \in \mathbb N_+}|\lambda_k|^{2s}|(u, e_k)_H|^2\simeq\sum_{k \in \mathbb N_+}|k|^{\frac{4s}d}
|(u, e_k)_H|^2< +\infty\right\}
\]
given by 
\[
 A^su:=\sum_{k \in \mathbb N_+}\lambda_k^{s}(u, e_k)_He_k, \quad u\in\mathcal D(A^s).
\]
For brevity of notation, we set
\[
V_{2s}:=\mathcal{D}(A^{s}), \quad s \in [0,1],
\]
and note that we have 
\[
\|u\|_{V_{2s}}^2\simeq\|u\|_H^2+\|A^su\|^2_H
\simeq\|(\operatorname{I}+A)^su\|_H^2
\qquad\forall\,u\in V_{2s}.
\]
For any $s \in (0,1)$ we set 
$V_{-2s}:=(\mathcal{D}(A^s))^*$ endowed with its natural dual norm. The duality pairing between $V_{-2s}$ and $V_{2s}$ will be denoted by the symbol $\langle \cdot, \cdot\rangle_{V_{-2s},V_{2s}}$. Notice that the restriction of $S$ to $V_{2s}$ is a strongly continuous semigroup of contractions on $V_{2s}$.

We conclude this part with the following lemma, which will be useful in the sequel.
\begin{lem}
\label{reg_A}
   Let  $\sigma\in[0,1]$ and $\delta > \sigma + \frac d4$. Then, 
   $(\operatorname{I}+A)^{-\delta} \in \cL^2(H, V_{2\sigma})$.
\end{lem}
\begin{proof}
It holds that 
\begin{align*}
  \sum_{k \in \mathbb N_+}
  \|(\operatorname{I}+A)^{-\delta} e_k\|^2_{V_{2\sigma}}
  \simeq \sum_{k \in \mathbb N_+} 
  (1+\lambda_k)^{2(\sigma-\delta)}
  =\sum_{k \in \mathbb N_+} 
  (1+k^{\frac2d})^{2(\sigma-\delta)},
\end{align*}
which is finite provided that 
$4(\delta-\sigma)>d$, i.e.~$\delta > \sigma + \frac d4$. 
\end{proof}

\subsubsection{The potential and mobility operators}
For $\gamma>0$, we introduce the family of functions 
\[
m_\gamma: (-1,1) \rightarrow \mathbb{R},
\qquad
m_\gamma(r):= (1-r^2)^\gamma, \quad r \in (-1,1).
\]
and 
\[
\Psi_\gamma:(-1,1) \rightarrow \mathbb{R}, \qquad
\Psi_\gamma(r):=\int_0^r\int_0^z\frac{1}{m_\gamma(w)}\,\d w\,\d z, \quad r \in (-1,1).
\]
Let us note that, for every $\gamma>0$, 
$\Psi_\gamma\in C^2(-1,1)$ is strictly convex
and $m_\gamma \Psi_\gamma''\equiv1$.
As usual, we extend $m_\gamma$ to $0$ outside $(-1,1)$
and denote such extension with the same symbol $m_\gamma$.
This allows to extend also $\Psi_\gamma$ to a proper, convex, lower semicontinuous function $\Psi_\gamma:\mathbb R\to[0,+\infty]$,
that satisfies in particular $\Psi_\gamma(x)=+\infty$ 
for all $x\in\mathbb R\setminus[-1,1]$.

Note that for every $\gamma\geq2$ it holds that
$\Psi_\gamma(r)\to+\infty$ and $|\Psi_\gamma'(r)|\to+\infty$
as $|r|\to1^-$,
while for $\gamma\in[1,2)$ one has that 
$\Psi_\gamma$ is bounded in $(-1,1)$
and $|\Psi_\gamma'(r)|\to+\infty$ as $|r|\to1^-$,
and for $\gamma\in(0,1)$ both 
$\Psi_\gamma$ and $\Psi_\gamma'$ are bounded in $(-1,1)$.
In particular, 
for $\gamma=1$ the potential $\Psi_\gamma=\Psi_1$
coincides with the classical logarithmic one, i.e. 
\[
  \Psi_1(r)=(1+r)\ln(1+r) + (1-r)\ln(1-r), \quad r\in[-1,1],
\]
and
\[
  \Psi_1'(r)=\ln\frac{1+r}{1-r}, \quad r\in(-1,1).
\]

In order to keep track of specific 
summability properties, 
let us also preliminarily introduce, for every 
$\beta,\gamma\geq1$, the function 
\begin{equation}
    \label{N}
    N_{\beta,\gamma}:\erre\to\erre,
    \qquad
    N_{\beta,\gamma}(r):=
      r\Psi_\gamma'((\Psi_\beta')^{-1}(r)), \quad r\in\erre.
\end{equation}
Note that for every $\beta,\gamma\geq1$ the function 
$N_{\beta,\gamma}$ is convex and superlinear at infinity, namely
\[
  \lim_{|r|\to+\infty}\frac{N_{\beta,\gamma}(r)}{|r|}=+\infty.
\]
Moreover, it is not difficult to check  that 
\[
N_{\beta,\gamma}(\Psi_\beta'(r))\lesssim 
\Psi_\beta'(r)\Psi_\gamma'(r)
  \quad\forall\,r\in(-1,1),
\]
and 
\[
N_{\beta,\gamma}(r)\gtrsim
\begin{cases}
    |r|^2
      \quad&\text{if } \beta=1 \text{ and } \gamma\geq 1,\\
      |r|^{\frac{\beta+\gamma-2}{\beta-1}} 
      \quad&\text{if } \beta>1 \text{ and } \gamma>1.
\end{cases}
\]

\subsubsection{The multiplicative noise operator}
Let $B_1^\infty$ be the closed unit ball in $L^\infty(\Td)$
and let $\delta>\frac d4$.
For every $\gamma>0$, we define the operator
\[
  G_\gamma:B^\infty_1\to \cL^2(H,H)
\]
by setting
\[
  G_\gamma(v)[e_k]:=
  \sqrt{m_{\gamma}(v)}
  (1+\lambda_k)^{-\delta}e_k, \quad k\in\mathbb N_+,
  \quad v\in B^{\infty}_1.
\]
Note that the operator is well-defined 
as a consequence of Lemma~\ref{reg_A} and since for every 
$v\in B^{\infty}_1$ one has $m_{\frac\gamma2}(v)\in L^\infty(\Td)$
by the boundedness of $m_\gamma$.
Equivalently, the operator $G_\gamma$ can be expresses, in compact form,
as
\[
G_\gamma(v)[h]=m_{\frac\gamma2}(v)
(\operatorname{I}+A)^{-\delta}(h), \qquad h \in H,
\quad v\in B^{\infty}_1.
\]
We collect some properties of the noise operator
in the following lemmata.

\begin{lem} 
\label{G_reg}
Let $\gamma\geq2$ 
and $\delta > \frac d4$: then, 
there exists a positive constant $C$, depending on $\gamma, \sigma, \eta, \delta$, such that 
    \[
    \|G_\gamma(v)\|^2_{\cL^2(H, H)} \le C
     \quad\forall\,
    v\in B_1^\infty.
    \]
Moreover, if $d\in\{1,2\}$, $\sigma\in[0,\frac12)\cap (\frac d4-\frac12,\frac12)$, $\delta>\frac d4+\sigma$,
and $\gamma\geq2d$, 
then for every $\eta\in(\sigma,1)\cap(\frac d4, \frac12+\sigma)$
the restriction
    \[
    G_\gamma: B_1^\infty \cap V_{2 \eta} \rightarrow \cL^2(H,V_{2\sigma})
    \]
    is well defined, and 
    there exists a positive constant $C$, 
    depending on $\gamma, \sigma, \eta, \delta$, such that 
    \[
    \|G_\gamma(v)\|^2_{\cL^2(H, V_{2\sigma})} \le C
    \|v\|_{V_{2\rho}}^{2(d-1)}
    \|v\|^2_{V_{2\eta}} \quad\forall\,
    v\in B_1^\infty \cap V_{2 \eta},
    \]
    where $\rho=0$ if $d=1$ and
    $\rho\in(1-\eta,\frac12)$ if $d=2$.
 \end{lem}
 \begin{proof}
The first statement follows directly from 
the H\"older inequality 
and the fact that 
$\|m_{\frac\gamma2}(v)\|_{L^\infty(\Td)}\leq\|m_{\frac\gamma2}\|_{L^\infty(-1,1)}$.
Let us consider the second one:
since $(I+A)^{-\delta}\in \cL^2(H, V_{2\sigma})$ by Lemma~\ref{reg_A},
by assumption on $\eta$ there exists 
$\eta'\in(\sigma,\eta)\cap (\frac d4, \eta)$ and
by Lemma~\ref{multiplication}
we have
 \begin{align*}
 \sum_{k \in \mathbb N_+}\|m_{\frac \gamma 2}(v)(\operatorname{I}+A)^{-\delta}e_k\|^2_{V_{2\sigma}}
\lesssim 
\|m_{\frac \gamma 2}(v)\|^2_{V_{2\eta'}}
\|(\operatorname{I}+A)^{-\delta}\|^2_{\cL^2(H, V_{2\sigma})}
\lesssim \|m_{\frac \gamma 2}(v)\|^2_{V_{2\eta'}}.
 \end{align*}
If $d=1$, one can choose $\eta'\in(\sigma,\eta)\cap(\frac14,\frac12)$,
so that 
the Lipschitz-continuity of $m_{\frac\gamma2}$ ($\gamma\geq2$) 
yields directly
\[
\int_{\Td\times\Td}
\frac{|m_{\frac\gamma2}(v(x))-m_{\frac\gamma2}(v(y))|^2}
{|x-y|^{d+2\eta'}}\, \d x\, \d y
\lesssim\int_{\Td\times\Td}\frac{|v(x)-v(y)|^{2}}
{|x-y|^{d+2\eta'}}\, \d x\, \d y
\lesssim\|v\|_{V_{2\eta'}}^2\lesssim\|v\|_{V_{2\eta}},
\]
as required. If $d=2$, one has 
$\eta\in(\frac12,\frac12+\sigma)$ and
$\eta'\in(\frac 12,\eta)$:
hence, noting that $\rho\in(1-\eta,\frac12)$,
one can choose $\eta'\in
(\frac12, \eta+\rho-\frac12)\cap(\frac12, \frac12+\rho)$,
so that $\rho\geq\eta'-\frac12$
and $\rho+\eta-\frac12>\eta'-\frac12+\frac d4$, and 
Lemma~\ref{multiplication} 
and the Lipschitz-continuity of $m_{\frac\gamma2}'$ ($\gamma\geq4$)
yield
\[
\|m_{\frac \gamma 2}(v)\|^2_{V_{2\eta'}}\lesssim
\|m'_{\frac \gamma 2}(v)\nabla v\|^2_{V_{2\eta'-1}}
\lesssim\|m'_{\frac \gamma 2}(v)\|_{V_{2\rho}}^2\|\nabla v\|^2_{V_{2\eta-1}}
\lesssim \|v\|_{V_{2\rho}}^2\|v\|_{V_{2\eta}}^2.
\]
This concludes the proof.
\end{proof}

\begin{lem}
\label{G_basis}
  Let $\gamma\geq2$, $\sigma\in[0,\frac12)\cap[\frac d4-\frac12, \frac12)$, and $\delta>\frac d4+\sigma$. Then, 
  the operator $G_\gamma$ is 
  Lipschitz continuous, in the sense that 
    for every $\lambda>0$
  there exists $L_\lambda>0$ such that,
  \[
  \|G_\gamma(v_1)- G_\gamma(v_2)\|^2_{\cL^2(H,H)} 
  \le \lambda\|v_1-v_2\|_V^2+
  L_\lambda\|v_1-v_2\|^2_{H}
  \quad\forall\,v_1,v_2\in B^\infty_1\cap V.
  \]
\end{lem}
\begin{proof}
Since $\delta>\frac d4+\sigma$, one has $(\operatorname{I}+A)^{-\delta}\in \cL^2(H,V_{2(\sigma+\varsigma)})$ for some 
$\s>0$ sufficiently small.
Let now $q>2$ be arbitrary if $d\in\{1,2\}$
and $q:=\frac{6}{3-4(\sigma+\s)}>2$ if $d=3$:
the Sobolev embeddings imply that $V_{2(\sigma+\s)}\embed L^q(\OO)$.
By setting then 
$p:=\frac{2q}{q-2}$, one has $\frac1p+\frac1q=\frac12$: hence
for all $v_1,v_2 \in B_1^\infty$, 
by the H\"older inequality and the 
Lipschitz-continuity of $m_\gamma$ ($\gamma\geq2$),
\begin{align*}
    \|G_\gamma(v_1)-G_\gamma(v_2)\|^2_{\cL^2(H,H)} &=
    \sum_{k \in \mathbb N_+} 
    \|(m_{\frac \gamma2}(v_1)-
    m_{\frac \gamma2}(v_2))(\operatorname{I}+A)^{-\delta}e_k\|^2_H \\
    &\leq\sum_{k \in \mathbb N_+} 
    \|(m_{\frac \gamma2}(v_1)-
    m_{\frac \gamma2}(v_2))\|_{L^p(\Td)}^2
    \|(\operatorname{I}+A)^{-\delta}e_k\|^2_{L^q(\Td)} \\
    &\lesssim
    \|m'_{\frac \gamma2}\|^2_{L^\infty(-1,1)}
    \|(\operatorname{I}+A)^{-\delta}\|^2_{\cL^2(H,V_{2(\sigma+\varsigma)})}
    \|v_1-v_2\|_{L^p(\Td)}^2.
\end{align*}
The thesis follows since the inclusion
$V\embed L^p(\OO)$ is compact. Indeed, this is clear
for $d\in\{1,2\}$, while for $d=3$ one has
$p=\frac{3}{2(\sigma+\s)}$ and
$1-\frac32>0-3\frac{2(\sigma+\s)}{3}$
since $\sigma+\s>\sigma\geq\frac14$.
\end{proof}

\subsection{Main results}
Given $\alpha,\beta\geq 1$, we consider problem \eqref{eq0} written in the following abstract form:
\beq
\label{eq_ast}
\begin{cases}
  \d X +A  X\,\d t + \Psi_{\beta}'(X)\, \d t
  = G_{\alpha} (X)\,\d W\,,
  \\
  X(0)=x\,.
\end{cases}
\eeq
We start by giving the definition of variational solution to problem \eqref{eq_ast}. 

\begin{defin}
\label{def_sol}
Let $\alpha\geq2$, $\beta\geq1$, $\gamma\geq1$,
    $\sigma\in[0,\frac12)$,
    $d\in\{1,2,3\}$, 
    $\delta>\frac d4$, and
\begin{equation}
    \label{init}
    x \in H, 
    \qquad
    \Psi_\gamma(x) \in L^1(\Td).
\end{equation}
A variational solution to \eqref{eq_ast}
starting from $x$ is a process $X$ such that, for every \( T > 0 \),
\begin{align*}
    X &\in L^2_\cP(\Omega; C^0([0,T]; H)) \cap 
    L^2_\cP(\Omega; L^2(0,T; V)),\\
N_{\beta,\gamma}(\Psi'_\beta(X))&\in 
L^{1}_\cP(\Omega; L^1(0,T; L^1(\Td))),
\end{align*}
and it holds, for every $t\in[0,T]$, $\P$-almost surely, that
\begin{multline*}
 ( X(t),v)_H  + \int_0^t (A^{\frac 12}X(s) , A^{\frac 12}v)_H \, {\rm d}s +\int_0^t (\Psi'_\beta(X(s)),v)_H  \, {\rm d}s\\
= (x,v )_H  + \left(\int_0^t G_\alpha(X(s))\,\d W (s), v\right)_H 
\quad\forall\,v\in V\cap L^\infty(\Td).
\end{multline*}
\end{defin}
\begin{remark}
    The condition $N_{\beta,\gamma}(\Psi'_\beta(X))\in 
L^{1}_\cP(\Omega; L^1(0,T; L^1(\Td)))$
 in Definition~\ref{def_sol}
should be interpreted as 
a specific summability condition on $\Psi_\beta'(X)$, which 
in particular ensures that the corresponding term in 
the variational formulation is well-defined since 
$v\in L^\infty(\Td)$. In the case $\gamma=\beta$,
one recovers the classical 
condition $\Psi_\beta'(X)\in L^2_\cP(\Omega; L^2(0,T; H))$.
\end{remark}

The first
main result concerns well-posedeness of equation \eqref{eq_ast}.
\begin{thm}[Well-posedness]
    \label{th:wp}
    Let $\alpha\geq2$, $\beta\geq1$,
    $d\in\{1,2,3\}$, 
    $\sigma\in[0,\frac12)\cap[\frac d4-\frac12, \frac12)$,
    and $\delta>\frac d4+\sigma$.
    Assume also that:
    \begin{itemize}
  \item if $\sigma\in\left(\frac d4, \frac12\right)$, then $\gamma\in[1,+\infty)$;
  \item if $\sigma=\frac d4$, then 
  $\gamma\in[1,+\infty)$, with also
  $\alpha>2$ if $\gamma>\alpha+\beta$;
  \item if $\sigma\in[0,\frac d4)$, then 
  $\gamma\in\left[1,\frac{\alpha - 8 \sigma/d}{1- 4\sigma/d}\right]
  \cup
  \left[1,\min \left\{\alpha+ \beta, 
    \frac{\alpha - 8 \sigma/d +4\beta\sigma/d}{1- 4\sigma/d} 
    \right\}\right]$.
\end{itemize}
    Then, for every 
$x$ satisfying \eqref{init}, there exists 
a unique variational solution $X^x$ to
\eqref{eq_ast}, in the sense of Definition~\ref{def_sol}, and
it holds, for every $T>0$ and $\ell\geq2$, that
\begin{align*}
  &X^x\in L^\ell_\cP(\Omega; C([0,T]; H)\cap L^2(0,T; V)),\\
  &\Psi_\gamma(X^x(t))\in L^\ell(\Omega;L^1(\Td))\quad\forall\,t\in[0,T].
\end{align*}
Moreover, for every $T>0$ and $\ell\geq2$,
there exists a constant $C>0$, only depending on
$\alpha,\beta,\gamma,\delta,\sigma,\ell,T$,
such that for every 
$x_1,x_2$ satisfying \eqref{init}, it holds that
\[
  \norm{X^{x_1}-X^{x_2}}_{L^\ell_\cP(\Omega; C^0([0,T]; H)\cap L^2(0,T; V))}
  \leq C\norm{x_1-x_2}_H.
\]
Eventually, when $d=1$,
if $\gamma\geq\beta$ it holds for every 
$\xi\in[0,\frac12)$ that
\[
X^x\in L^\ell_\cP(\Omega; C([0,T]; V_{2\xi}))
  \quad\text{if } x\in V_{2\xi},
\]
while if $\gamma\geq2\beta$ it further holds
for every $\xi\in[\sigma,\frac12+\sigma)$ that
\[
X^x\in L^\ell_\cP(\Omega; C([0,T]; V_{2\xi})\cap 
L^2(0,T; V_{2\sigma+1}))
  \quad\text{if } x\in V_{2\xi}.
\]
\end{thm}
\begin{remark}
    Let us comment on the range of values of the parameter $\gamma$. First, note that in the case $\sigma\in(\frac d4,\frac12)$ we have no restriction on $\gamma$.
    Secondly, if $\sigma=\frac d4$ again all values of $\gamma$ are
    admissible, provided that $\alpha>2$ when $\gamma>\alpha+\beta$. Lastly, if $\sigma\in[0,\frac d4)$
    there is an actual restriction on the range of $\gamma$,
    which becomes more strict the more $\sigma$ is close to $0$:
    indeed, when $\sigma\to0^+$ the range of $\gamma$
    comes down to $\gamma\in[1,\alpha]$.
    However, we point out that when $\sigma\to(\frac d4)^-$
    then $\frac{\alpha - 8 \sigma/d +4\beta\sigma/d}{1- 4\sigma/d}
    \to+\infty$, and if also $\alpha>2$ then 
    $\frac{\alpha - 8 \sigma/d}{1- 4\sigma/d}\to+\infty$.
    This means that for every $\gamma\geq1$
    there exists $\sigma_0\in(0,\frac d4)$
    sufficiently close to $\frac d4$ such that $\gamma$
    is admissible
    with the only condition
    that $\alpha>2$ when $\gamma>\alpha+\beta$.
\end{remark}

The well-posedness result in Theorem~\ref{th:wp}
allows to introduce the transition semigroup associated to 
\eqref{eq_ast}. To this end, for every $\gamma\geq1$
we introduce the sets 
\begin{align*}
    \mathcal K_{\gamma}&:=\left\{
  v:\Td\to(-1,1) \text{ measurable:}\quad
  \Psi_\gamma(v)\in L^1(\Td)\right\},\\
  \mathcal K_{\infty}&:=\left\{
  v:\Td\to(-1,1) \text{ measurable:}\quad
  \exists\,c\in(0,1):\;\|v\|_{L^\infty(\Td)}\leq c\right\}.
\end{align*}
Let us first note that 
for every $\gamma\in[1,+\infty]$,
$\mathcal K_\gamma$ 
is a Borel subset of $H$ by
convexity and lower semicontinuity.
Moreover, let us note that since $D(\Psi_1)=[-1,1]$, 
one has that 
$\mathcal K_1=B_1^\infty$
is the closed unitary ball of $L^\infty(\Td)$.
By contrast,  
$\mathcal K_\infty$
is the open unitary ball of $L^\infty(\Td)$, namely
the set of functions that are separated from the 
potential barriers $\pm1$. 

We consider the space 
\[
  \mathcal X:=\mathcal K_1=\left\{v\in L^\infty(\Td):\;
  \|v\|_{L^\infty(\Td)}\leq1\right\}
\]
as a metric space endowed with the metric of $H$, so that 
the Borel $\sigma$-algebra
$\cB(\mathcal X)$
of $\mathcal X$
is exactly the trace of the Borel $\sigma$-algebra
$\cB(H)$ of $H$ on $\mathcal X$.
We will denote by $\mathcal B_b(\mathcal X)$
the space of functions 
$\varphi:\mathcal X\to\mathbb R$
that are bounded and 
$\cB(\mathcal X)$-measurable, and by 
$\mathcal P(\mathcal X, \cB(\mathcal X))$
the space of probability measures on 
$(\mathcal X, \cB(\mathcal X))$.

Theorem~\ref{th:wp} with the choice $\gamma=1$
ensures that the evolution of \eqref{eq_ast}
can be framed in the metric space $\mathcal X$, 
in the sense that for every 
$x\in \mathcal X$, there exists a unique
variational solution $X^x$ to \eqref{eq_ast}.
Moreover, 
for every $t\geq0$ one has that 
$X^x(t):\Omega\to \mathcal X$
is a well defined 
$\cF_t/\cB(\mathcal X)$-measurable
random variable. Consequently, it is possible to introduce 
the transition semigroup associated to \eqref{eq_ast}
as the family of operators
$P:=(P_t)_{t \ge 0}$ given by 
\begin{equation}
\label{P_t}
(P_t\varphi)(x):= \mathbb{E}[ \varphi(X^x(t))], \quad x\in \mathcal X,
\quad \varphi\in \mathcal{B}_b(\mathcal X),
\quad t>0.
\end{equation}
\begin{remark}
We point out that, due to the nonlinear nature of the problem, 
the solution of equation \eqref{eq_ast} exists only on
$\mathcal X$, hence
the transition semigroup can only make sense as a family of operators acting on $\mathcal{B}_b(\mathcal X)$,
and not on $\mathcal{B}_b(H)$ as in more classical cases.
\end{remark}
Let us note that, for every $t\geq0$,
$P_t\varphi$ is bounded for every $\varphi \in \mathcal{B}_b(\mathcal X)$.
Moreover, since we know from \cite[Cor.~23]{On2005}
that the transition function 
$(x,t)\mapsto \mathbb{P}(X^x(t)\in A)$, 
$(x,t)\in \mathcal X\times[0,+\infty)$,
is jointly measurable, we have that 
$P_t\varphi$ is also $\cB(\mathcal X)$-measurable for every $\varphi \in \mathcal{B}_b(\mathcal X)$. 
Hence $P_t$ maps 
$\mathcal{B}_b(\mathcal X)$
into itself for every $t \ge0$, namely
\[
  P_t:\mathcal{B}_b(\mathcal X)\to\mathcal{B}_b(\mathcal X),
  \quad t\geq0.
\]
Furthermore, since the unique solution of \eqref{eq_ast} is an $H$-valued continuous process, 
then it is also a Markov process, see \cite[Thm.~27]{On2005}. 
Therefore, we deduce that the family of operators 
$P$ is a Markov semigroup, 
namely $P_{t+s}=P_tP_s$ for any $s,t\ge0$.
A direct application of the continuous dependence 
of Theorem~\ref{th:wp} and of the dominated convergence theorem ensures also that $P$ is also a Feller semigroup, namely
\[
  P_t\varphi\in \mathcal C_b(\mathcal X) \quad\forall\,\varphi\in \mathcal C_b(\mathcal X), \quad\forall\,t\geq0.
\]
For every $t\geq0$, 
we define the Markov transition kernel 
$\mu_t:\mathcal X\times\cB(\mathcal X)\to[0,1]$ as
\[
  \mu_t(x,A):=\P\left(X^x(t)\in A\right),
  \quad (x,A)\in \mathcal X\times\cB(\mathcal X),
  \quad t\geq0,
\]
so that $\mu_t(x,\cdot)$ is the probability distribution 
of the random variable $X^x:\Omega\to\mathcal X$, for all $x\in \mathcal X$ and $t\geq0$. 
Let us recall that an invariant measure for $P$
is a probability measure $\mu\in\mathcal P(\mathcal X, \cB(\mathcal X))$ such that 
\[
    \int_{\mathcal X}
    P_t\varphi\,\d\mu=
    \int_{\mathcal X}
    \varphi\,\d\mu \qquad\forall\,\varphi\in\,\mathcal B_b(\mathcal X), \quad\forall\,t\geq0.
\]
We also recall that an invariant measure $\mu$ for $P$
is said to be ergodic if 
\[
  \lim_{t\to\infty}
  \frac1t\int_0^tP_s\varphi\,\d s=
  \int_{\mathcal X}\varphi\,\d \mu
  \quad\text{in } L^2(\mathcal X,\mu), \quad\forall\,\varphi\in\mathcal B_b(\mathcal X),
\]
while $\mu$ is strongly mixing if 
\[
\lim_{t\to\infty}
  P_t\varphi=
  \int_{\mathcal X}\varphi\,\d \mu
  \quad\text{in } L^2(\mathcal X,\mu), \quad\forall\,\varphi\in\mathcal B_b(\mathcal X).
\]

We are now ready to state the
strong Feller property of the transition semigroup $P$.

\begin{thm}[Strong Feller property]
    \label{th:sf}
    Let $d=1$, $\alpha\geq2$, and $\beta\geq1$.
    Then, there exist 
    $\sigma_0^1, \sigma_0^2\in (0,\frac14)$,
    only depending of $\alpha,\beta$, 
    such that 
    under any of the following settings
    \begin{itemize}
        \item $\alpha>2$, 
        $\sigma\in(\sigma_{0}^1, \frac14)$,
        $\delta\in(\frac14+\sigma, \frac12)$,
        $\xi\in(\delta-\sigma-\frac14, \frac12+\sigma)$, and $\gamma:=\alpha+\max\{\beta,2\}$,
        \item $\alpha\geq4$, 
        $\sigma\in(\sigma_{0}^2, \frac12)$,
        $\delta\in[\frac12, \frac12+\sigma)\cap
        (\frac14+\sigma,1)$,
        $\xi\in[\delta,\frac12+\sigma)$, and $\gamma:=\alpha+2\beta+2$,
    \end{itemize}
    the following holds: there exist
    $\mathfrak a\in(\frac12,1)$ and
    $\mathfrak b>0$, only depending on $\sigma,\delta$,
    such that for every $T>0$ 
    there exists $C>0$, only depending 
    on $\alpha,\beta,\delta,\sigma,\xi,T$,
    such that,
    for every $\varphi\in\mathcal B_b(\mathcal X)$ and for every $t\in(0,T)$, it holds
    for every $x_1,x_2\in \mathcal K_{\gamma}\cap V_{2\xi}$
    with $x_1-x_2\in V_{2\delta}$
    that 
    \[
    |(P_t\varphi)(x_1)-(P_t\varphi)(x_2)|
    \leq\frac{C}{t^{\mathfrak a}}
    \sup_{v\in \mathcal X}|\varphi(v)|
    \sum_{i=1}^2
    \left(1+\|x_i\|_{V_{2\xi}}^{\mathfrak b}
    +\|\Psi_{\gamma}(x_i)\|_{L^1(\Td)}^{\mathfrak b}\right)\|x_1-x_2\|_{V_{2\delta}}.
    \]
    Furthermore, there exist 
    $\tilde\sigma_0^1, \tilde\sigma_0^2\in (0,\frac14)$,
    only depending of $\alpha,\beta$, such that 
    the same conclusion holds
    with $\frac12$ instead of $\mathfrak a$
and with some $\tilde{\mathfrak b}>0$ instead of $\mathfrak b$
    in any of the following scenarios:
    \begin{itemize}
        \item $\alpha>2$, 
        $\sigma\in(\tilde\sigma_{0}^1, \frac14)$,
        $\delta\in(\frac14+\sigma, \frac12)$,
        $\xi\in[\frac12+\frac{\alpha+2}{4(\gamma-2)}, 
        \frac12+\sigma)$, $\gamma\in[\alpha+\beta,+\infty)\cap(\frac{\alpha+2}{4\sigma}+2,+\infty)$,
        \item $\alpha\geq4$, 
        $\sigma\in(\tilde\sigma_{0}^2, \frac12)$,
        $\delta\in[\frac12, \frac12+\sigma)\cap
        (\frac14+\sigma,1)$,
        $\xi\in[\max\{\delta, \frac12+\frac{\alpha+4}{4(\gamma-2)}\},\frac12+\sigma)$, \\and $\gamma\in[\alpha+2\beta+2,+\infty)\cap 
        (\max\{1,\frac1{4\sigma}\}(\alpha+4)+2,+\infty)$.
    \end{itemize}
\end{thm}

\begin{remark}
    The strong Feller property stated in Theorem~\ref{th:sf}
    can be formulated by saying that 
    the transition semigroup $P$ maps bounded measurable functions on $\mathcal X$
    into locally Lipschitz functions on $\mathcal K_{ \gamma}\cap V_{2\xi}$
 with respect to the metric
    of $V_{2\delta}$.
    The precise constants appearing 
    in Theorem~\ref{th:sf} are the following:
        \[
    \sigma_{0}^i:=
    \begin{cases}
    \frac14\max\left\{\frac\beta{\alpha+\beta-2}, \frac1{2\alpha}\right\}
    \quad&\text{if } i=1,\\
        \frac14\max\left\{
        \frac{2\beta+2}{\alpha+2\beta},
        \frac1{\alpha+2}\right\}
    \quad&\text{if } i=2,
    \end{cases}
    \qquad
    \tilde\sigma_0^i:=
    \begin{cases}
    \max\left\{\sigma_0^1, \frac{\alpha+2}{8\alpha}\right\}
    \quad&\text{if } i=1,\\
      \max\left\{\sigma_0^3, \frac18\frac{\alpha+4}{\alpha+1}\right\}
        \quad&\text{if } i=2,
    \end{cases}
    \]
    and
    \[
    \mathfrak a:=
    \begin{cases}
    \frac12+\delta,\\
      \in(\frac34,1),\\
     \delta,
    \end{cases}
    \mathfrak b:=
    \begin{cases}
      \frac12+\delta+2\sigma ,\\
      3+\mathfrak a ,\\
      3+\delta,
    \end{cases}
    \tilde{\mathfrak b}:=
    \begin{cases}
      \frac12+2\delta+2\sigma \qquad\qquad&\text{if } \delta\in(\frac14,\frac12),\\
      \in(4,\frac92) \qquad\qquad&\text{if } \delta\in[\frac12,\frac34],\\
      \frac52+2\delta 
      \qquad\qquad&\text{if } \delta\in(\frac34,1).
    \end{cases}
    \]
\end{remark}

The last two results concern irreducibility 
of the semigroup $P$ and 
uniqueness of invariant measures.
Since the strong Feller property has been established with respect to the metric of 
$V_{2\delta}$, it is natural
(and also necessary to get uniqueness of invariant measures)
to state irreducibility 
with respect to the same topology.

\begin{thm}[Irreducibility]
    \label{th:irr}
    Under the assumptions of Theorem~\ref{th:wp},
    suppose that $d=1$, 
    $\delta<\frac12+\sigma$,
    and either 
    $\gamma\geq\max\{\beta,\alpha+2\}$ if $\delta<\frac12$
    or 
    $\gamma\geq\max\{2\beta,\alpha+4\}$ if $\delta\geq\frac12$.
    Then,
    for every 
    $x,a\in\mathcal K_{\gamma}\cap V_{2\delta}$ and $t,r>0$,
    it holds that 
    \[
      \P\left(
      \norm{X^x(t)-a}_{V_{2\delta}}<r
      \right)>0.
    \]
\end{thm}

\begin{thm}[Uniqueness of the invariant measure]
    \label{th:uniq_im}
    Let $d=1$, $\beta\geq1$, and assume one of the following settings:
    \begin{itemize}
        \item $\alpha>2$, $\sigma\in(\sigma_0^1,\frac14)$, 
        and $\delta\in(\frac14+\sigma,\frac12)$,
        \item $\alpha\geq4$, $\sigma\in(\sigma_0^2,\frac12)$, 
        and $\delta\in[\frac12,\frac12+\sigma)
        \cap(\frac14+\sigma,1)$.
    \end{itemize}
    Then,
    there exists a unique probability measure $\mu$
    on $(\mathcal X, \cB(\mathcal X))$ that is invariant 
    for $P$. Moreover, $\mu$ is ergodic, strongly mixing,
    and $\mu(\mathcal K_\infty\cap V_{2\sigma+1})=1$.
    Lastly, $\mu$ is equivalent to all transition 
    probabilities $\mu_t(x,\cdot)$ for every $t>0$ and $x\in\mathcal X$, and it holds that 
    \[
    \lim_{t\to+\infty}\mu_t(x,A)=\mu(A)
    \quad\forall\,A\in\cB(\mathcal X),
    \quad\forall\,x\in\mathcal X.
    \]
\end{thm}

\subsection{Consequences on the random separation property}
The irreducibility Theorem~\ref{th:irr} has 
important consequences on the so-called random separation property, first introduced in
\cite{BOS21, orr-scar-sep}. In particular, 
it answers an open problem on the qualitative behaviour 
of the separation layers of the trajectories.

First of all, let us note that for $d=1$ if $\xi>\frac14$,
then for all $x\in V_{2\xi}\cap \mathcal K_\infty$ 
the following random separation property holds:
    \[
      \P\left(\sup_{(t,z)\in[0,T]\times\overline\Td}|X^x(t,z)|<1\right)
      =1\quad\forall\,T>0.
    \]
This can be shown via the techniques of 
\cite{BOS21, orr-scar-sep}
and means that $\P$-almost every trajectory 
is strictly separated from the potential barriers $\pm1$
in space and time. Equivalently, this can be formulated as 
\[
      \P\left(\left\{\omega\in\Omega:\;
      \exists\,\mathfrak L\in(0,1):
      \sup_{(t,z)\in[0,T]\times\overline\Td}|X^x(\omega, t,z)|\leq1-\mathfrak L\right\}\right)
      =1\quad\forall\,T>0.
\]
The random separation layer $\mathfrak L$ in general depends 
on $\omega\in\Omega$, i.e.~on the trajectory.
One may wonder whether such separation layer 
can be chosen uniformly with respect to the trajectories:
such question was left open in \cite{BOS21, orr-scar-sep}.
The irreducibility Theorem~\ref{th:irr} answers this question, 
in a negative way.
We have in fact the following corollary, that 
states in particular that $\mathfrak L$
{\em cannot} be chosen uniformly with respect to $\omega$.

\begin{cor}
\label{cor:sep}
    Under the assumptions of Theorem~\ref{th:irr}, suppose also that $\gamma>\frac2{4\delta-1}+2$.
    Then, for every $x\in \mathcal K_{\gamma}\cap V_{2\delta}$,
    for every $t>0$, and for every $\mathfrak L\in(0,1)$,
    it holds that 
    \[
    \P\left(\norm{X^x(t)}_{C^0(\overline\Td)}>1-\mathfrak L
    \right)>0.
    \]
\end{cor}
\begin{proof}
    It follows directly from Theorem~\ref{th:irr}
    and the fact that $V_{2\delta}\embed C^0(\overline\Td)$.
    For the proof of the random separation property 
we refer e.g.~to the techniques of \cite{BOS21, orr-scar-sep}.
The proof
is based 
on the regularities $X^x\in L^\infty(0,T;V_{2\xi})$,
$\Psi_\gamma(X^x)\in L^\infty(0,T; L^1(\Td))$ $\P$-almost surely,
and the fact that 
for $d=1$
it holds that 
\[
    V_{2\xi}\cap\mathcal K_\gamma=
    V_{2\xi}\cap\mathcal K_\infty \quad\forall\,
    \xi>\frac14, 
    \quad\forall\,
    \gamma>\frac{2}{4\xi-1}+2.
\]
    Indeed, for $\xi>\frac 14$ one has that 
    $V_{2\xi}\embed C^{0,\vartheta}(\overline\Td)$ for all 
    $\vartheta\in[0,2\xi-\frac 12)$. Hence,
    for $v\in V_{2\xi}\cap\mathcal K_\gamma$,
    it holds that $\|v\|_{L^\infty(\Td)}=\max_{x\in \overline\Td}|v(x)|<1$: indeed, if by contradiction there is $x_0\in\overline\Td$ such that $|v(x_0)|=1$, then one would have
    \[
    \frac1{2^{\gamma-2}}\int_{\Td}
    \frac1{|v(x)-v(x_0)|^{\gamma-2}}\,\d x\leq
    \int_{\Td}
    \frac1{|v(x)-1|^{\gamma-2}|v(x)+1|^{\gamma-2}}\,\d x=
    \int_{\Td}\Psi''_{\gamma-2}(v)
    \]
    where the right-hand side is finite since $v\in\mathcal K_\gamma$ and $\Psi''_{\gamma-2}\lesssim\Psi_\gamma$. However, since $\frac 1{\gamma-2}<2\xi-\frac 12$
    one can choose $\vartheta\in(\frac 1{\gamma-2},2\xi-\frac 12)$, so that
    \[
    \int_{\Td}
    \frac1{|v(x)-v(x_0)|^{\gamma-2}}\,\d x
    \geq\|v\|_{C^{0,\vartheta}(\overline\Td)}^{2-\gamma}
    \int_{\Td}\frac1{|x-x_0|^{\vartheta(\gamma-2)}}\,\d x
    =+\infty.
    \]
    Choosing $\xi=\delta>\frac14$ yield then 
    $\gamma>\frac{2}{4\delta-1}+2$, as required.
\end{proof}

    \begin{remark}
    An important consequence of the irreducibility Theorem~\ref{th:irr} and of its Corollary~\ref{cor:sep}
    on the non-uniformity of the random separation property 
    is that, even if the initial datum $x$ is separated from 
    $\pm1$ and even if almost every trajectory is separated 
    in space and time from $\pm1$,
    the process $X^x$ may accumulate arbitrarily towards 
    $\pm1$. This means that, even under the setting 
    of a random separation property, the action of the multiplicative noise on the process $X^x$ is still
    degenerate, i.e.~
    \[
      \essinf_{(\omega,t)\in\Omega\times[0,T]}
      \norm{G_\alpha(X^x(\omega,t))}^2_{\cL^2(H,H)}=0 \quad\forall\,T>0.
    \]
    In this sense, the result of Theorem~\ref{th:sf}
    on the strong Feller property
    is profoundly non-trivial and requires a taylored measurement of accumulation of $X^x$ towards $\pm1$.  
\end{remark}

\section{Well-posedness}
\label{sec:X}
This section is devoted to the proof of 
Theorem~\ref{th:wp} on well-posedness of equation \eqref{eq_ast}.
We first prove continuous dependence with respect to the initial datum, hence also uniqueness, and then we focus on existence.
This is based on the introduction 
of a suitable regularised equation, 
for which a priori estimates are obtained.

\subsection{Uniqueness and continuous dependence}
\label{ssec:uniq}
We focus here on the prove of the continuous dependence
and uniqueness part of Theorem~\ref{th:wp}.
Let then $x_1$ and $x_1$ be initial data satisfying \eqref{init},
and let $X_1$ and $X_2$ be any two respective 
solutions to equation \eqref{eq_ast}, in the sense of 
Definition~\ref{def_sol}.
By setting $R:=X_2-X_1$, 
$\Xi:=\Psi_\beta'(X_2)-\Psi_\beta'(X_1)$,
and $\Gamma:=G_\alpha(X_2)-G_\alpha(X_1)$,
the equation
\[
\d R + AR\,\d t +\Xi\d t=\Gamma\, \d W, \qquad
R(0)=x_2-x_1,
\]
is satisfied still in the sense of Definition~\ref{def_sol}.
Since It\^o formula is not directly applicable 
due to the low 
regularity of $\Xi$,
we rely on an approximation procedure, 
inspired by the one performed in the proof of \cite[Prop.~6.2]{mar-scar-diss}.
To this end, for every $n\in\mathbb N$, let $\mathcal R_n:=(\operatorname{I}+\frac1n A)^{-m}$, where $m>0$ is fixed so that 
$\mathcal R_n\in \cL(H,V)\cap \cL(L^1(\Td), H)$
for every $n\in\mathbb N$
(see e.g.~\cite[App.~A]{scar-div}).
We recall that $\mathcal R_n:$ is non-expansive 
on $H$, $V$, and $L^1(\Td)$, 
and that, as $n\to\infty$, 
$\mathcal R_n z\to z$ in $H$ (resp.~$V$ or $L^1(\Td)$) 
for every $x\in H$ (resp.~$x\in V$ or $x\in L^1(\Td)$).
By setting, for $k=1,2$, $X_{k,n}:=\mathcal R_nX_k$, $x_{k,n}:=\mathcal R_nx_k$,
$R_n:=\mathcal R_nR=X_{2,n}-X_{1,n}$, 
$\Xi_n:=\mathcal R_n\Xi$, 
and $\Gamma_n:=\mathcal R_n\Gamma$, one has
\[
\d R_n + AR_n\,\d t +\Xi_n\d t=\Gamma_n\, \d W, \qquad
R_n(0)=x_{2,n}-x_{1,n}.
\]
Now, the classical It\^o formula yields,
for all $t\in[0,T]$, $\P$-almost surely,
\begin{align*}
    &\frac12\norm{R_n(t)}_H^2
    +\int_0^t\norm{A^\frac12R_n(s)}_H^2\,\d s
    +\int_0^t\int_{\Td}\Xi_n(s)R_n(s)\,\d s\\
    &=\frac12\norm{x_{2,n}-x_{1,n}}_H^2
    +\int_0^t\left(R_n(s), \Gamma_n(s)\,\d W(s)\right)_H
    +\frac12\int_0^t
    \norm{\Gamma_n(s)}^{2}_{\cL^2(H,H)}\,\d s.
\end{align*}
Thanks to the convergence properties of $(\mathcal{R}_n)_n$, 
it is
immediate to see that, as $n\to\infty$,
\begin{align*}
    R_n(t)\to R(t) \quad&\text{in } H, \quad\forall\,t\in[0,T], \quad\P\text{-a.s.},\\
    R_n\to R \quad&\text{in } 
    L^2_\cP(\Omega; L^2(0,T;V)),\\
    \Xi_n\to\Xi \quad&\text{in } 
    L^1_\cP(\Omega; L^1(0,T;L^1(\Td))),\\
    \Gamma_n\to\Gamma
    \quad&\text{in } 
    L^2_\cP(\Omega; L^2(0,T;\cL^2(H,H))).
\end{align*}
Moreover, since $|X_k|\leq1$ for $k=1,2$,
one has $|\Xi_n R_n|\leq2|\Xi_n|$
almost everywhere in $\Omega\times(0,T)\times\Td$.
Since the right-hand side converges in $L^1_\cP(\Omega\times(0,T)\times\Td)$,
we infer that the left-had side is uniformly integrable, hence also that 
\[
  \Xi_n R_n\to\Xi R \quad\text{in } L^1_\cP(\Omega; L^1(0,T; L^1(\Td))).
\]
By letting $n\to\infty$ we get then,
for all $t\in[0,T]$, $\P$-almost surely,
\begin{align*}
    &\frac12\norm{R(t)}_H^2
    +\int_0^t\norm{A^\frac12R(s)}_H^2\,\d s
    +\int_0^t\int_{\Td}\Xi(s)R(s)\,\d s\\
    &=\frac12\norm{x_2-x_1}_H^2
    +\int_0^t\left(R(s), \Gamma(s)\,\d W(s)\right)_H
    +\frac12\int_0^t
    \norm{\Gamma(s)}_{\cL^2(H,H)}^2\,\d s.
\end{align*}
Now, by monotonicity we have 
\[
\int_0^t\int_{\Td}\Xi(s)R(s)\,\d s \geq0,
\]
while by Lemma~\ref{G_basis} we have, for every $\lambda>0$,
\[
  \frac12\int_0^t
    \norm{\Gamma(s)}^2_{\cL^2(H,H)}\,\d s
  \leq \lambda
  \int_0^t\norm{A^\frac12R(s)}_H^2\,\d s
  +C_\lambda \int_0^t\norm{R(s)}_H^2\,\d s.
\]
Similarly, by the Burkholder-Davis-Gundy inequality we get, by possibly renominating the constant $C_\lambda$,
\begin{align*}
    \E\sup_{\tau\in[0,t]}\left|\int_0^\tau
    \left(R(s), \Gamma(s)\,\d W(s)\right)_H
    \right|^{\frac\ell2}
    &\leq C_\ell
    \E\left(\int_0^t
    \norm{R(s)}_H^2
    \norm{\Gamma(s)}^2_{\cL^2(H,H)}\,\d s
    \right)^{\frac\ell4}\\
    &\leq\lambda\E\sup_{s\in[0,t]}
    \norm{R(s)}_H^\ell
    +C_\lambda\E\left(\int_0^t
    \norm{\Gamma(s)}_{\cL^2(H,H)}^2\,\d s\right)^{\frac\ell2},
\end{align*}
where, again by Lemma~\ref{G_basis},
\begin{align*}
    \E\left(\int_0^t
    \norm{\Gamma(s)}_{\cL^2(H,H)}^2\,\d s\right)^{\frac\ell2}
    \leq 
    \lambda\E\left(
  \int_0^t\norm{A^\frac12R(s)}_H^2\,\d s\right)^{\frac\ell2}
  +C_\lambda \E\int_0^t\norm{R(s)}_H^\ell\,\d s.
\end{align*}
The continuous dependence of Theorem~\ref{th:wp}
follows then by
raising the It\^o formula to power $\frac\ell2$,
by taking supremum in time and expectations, 
and by rearranging the terms with a chosen $\lambda$
small enough.

\subsection{The approximated equation}
Let us focus here on the proof of the existence part of Theorem~\ref{th:wp}.
For every $\varsigma>0$, recalling that $m_\varsigma$
is extended to $0$ outside $(-1,1)$,
we set the approximated mobility and potential, 
for every $\varepsilon \in (0,1)$, as
\begin{equation}
    \label{m_eps}
m_{\varsigma, \varepsilon}: \mathbb{R} \rightarrow \mathbb{R},
\qquad
m_{\varsigma, \varepsilon}(r):=(\rho_\eps*m_\varsigma)(r)
+ \varepsilon^\s, \quad r \in\mathbb{R},
\end{equation}
and 
\begin{equation}
\label{psi_eps}
\Psi_{\varsigma, \varepsilon}:\mathbb{R} \rightarrow \mathbb{R},
\qquad 
\Psi_{\varsigma, \varepsilon}(r):=
\int_0^r\int_0^z\frac{1}{m_{\varsigma,\varepsilon}(w)}\,\d w\,\d z, \quad r \in \mathbb{R},
\end{equation}
where $(\rho_\eps)_\eps$ is a classical family of mollifiers.
For example, one can take $\rho_\eps(r)=\eps^{-1}\rho(r/\eps)$, $r\in\erre$, where $\rho(r):=c_Ne^{-\frac1{1-r^2}}\ind_{(-1,1)}(r)$, $r\in\erre$, and $c_N>0$ is a normalisation constant.
Analogously, 
for every $\varsigma>0$,
we define the approximated noise operator as
\begin{equation}
\label{G_eps}
G_{\varsigma, \varepsilon}:H \rightarrow \cL^2(H,H),
\qquad
G_{\varsigma, \varepsilon}(x)[h]:=
m_{\frac \varsigma 2, \varepsilon}(x)
(\operatorname{I}+A)^{-\delta}(h), \qquad x,h \in H.
\end{equation}

We collect the main properties of the approximated families 
in the following lemmata.
\begin{lem}
    \label{lem_app}
    The following properties hold.
        \begin{itemize}
\item[(i)]
\label{mse1}
For every $\varsigma\geq1$ and $\varepsilon\in(0,1)$ 
it holds
that $m_{\varsigma,\varepsilon}\in C^{\infty}_c(\mathbb R)$ and 
\[
  \norm{m_{\varsigma, \varepsilon}}_{L^\infty\left(\mathbb R\right)}
  \leq
\norm{m_\varsigma}_{L^\infty(\mathbb R)}+1,
\qquad
\norm{m'_{\varsigma, \varepsilon}}_{L^\infty(\mathbb R)}\leq
\norm{m'_\varsigma}_{L^\infty(\mathbb R)}.
\]
\item[(ii)]
\label{mse2}
For every $\varsigma\geq1$ and $\varepsilon\in(0,1)$ 
it holds
that 
\[
  \frac{1}{m_{\varsigma, \varepsilon}(r)}\le \frac{1}{\varepsilon^\s} \quad \forall\, r \in \mathbb{R}.
\]
\item [(iii)]
\label{Psi1} 
For every $\varsigma\geq1$ and $\varepsilon\in(0,1)$
it holds that $\Psi_{\varsigma,\varepsilon}\in C^2(\mathbb R)$
is strictly convex, 
$\Psi_{\varsigma,\varepsilon}'$ is monotone increasing and $\frac{1}{\varepsilon^\varsigma}$-Lipschitz continuous,
with $\Psi_{\varsigma,\varepsilon}'(0)=0$, 
$m_{\varsigma,\varepsilon}\Psi_{\varsigma,\varepsilon}''\equiv1$,
and 
\[
  \frac1{\norm{m_\varsigma}_{L^\infty(\mathbb R)}+1}\leq\Psi_{\varsigma,\varepsilon}''(r)\leq \frac1{\varepsilon^\s}
  \quad\forall\,r\in\mathbb R.
\]
\item [(iv)]
\label{Psi2}
For every $\varsigma>1$, there exist $C_\varsigma, c_\varsigma>0$
such that, for every $\varepsilon\in(0,1)$,
it holds that 
\[
c_\varsigma\Psi_{\varsigma-1, \varepsilon}''(r)\leq 
|\Psi_{\varsigma, \varepsilon}'(r)| 
\le C_\varsigma\Psi_{\varsigma-1, \varepsilon}''(r)
\quad\forall\,r\in\mathbb R.
\]
\item [(v)]
\label{Psi3}
For every $\varsigma\geq1$, there exist $C_\varsigma, c_\varsigma>0$
such that, for every $\varepsilon\in(0,1)$,
it holds that 
\begin{alignat*}{2}
c_\varsigma\Psi_{\varsigma+1, \varepsilon}(r) &\leq 
|\Psi_{\varsigma, \varepsilon}'(r)| &&\le C_\varsigma\Psi_{\varsigma+1, \varepsilon}(r)
\quad\forall\,r\in\mathbb R,\\
c_\varsigma\Psi_{\varsigma+2, \varepsilon}(r) &\leq 
\Psi_{\varsigma, \varepsilon}''(r) &&\le C_\varsigma\Psi_{\varsigma+2, \varepsilon}(r)
\quad\forall\,r\in\mathbb R,\\
c_\varsigma\Psi''_{\varsigma+1, \varepsilon}(r) &\leq 
|\Psi_{\varsigma, \varepsilon}'''(r)| &&\le 
C_\varsigma\Psi''_{\varsigma+1, \varepsilon}(r)
\quad\forall\,r\in\mathbb R.
\end{alignat*}
\item [(viii)]
\label{Psi5}
For every $\varsigma, \zeta \ge 1$, $\varsigma\le\zeta$,
there exists $C_{\varsigma,\zeta}>0$ such that,
for every $\varepsilon\in(0,1)$, it holds that
\[
\Psi_{\varsigma,\varepsilon}(r)\leq C_{\varsigma,\zeta}
\Psi_\zeta(r) \quad\forall\,r\in\mathbb R, \qquad
\Psi_{\varsigma,\varepsilon}''(r)\leq C_{\varsigma,\zeta}
\Psi_\zeta''(r) \quad\forall\,r\in\mathbb R.
\]
\item [(ix)]
\label{Psi6}
For every $\varsigma, \zeta \geq 1$ with $\varsigma+\zeta>2$,
there is $K_{\varsigma,\zeta}>0$ such that,
for every $\varepsilon\in(0,1)$, it holds 
    \[
    \Psi_{\varsigma, \varepsilon}'(r)
    \Psi_{\zeta, \varepsilon}'(r)
    \geq K_{\varsigma,\zeta}
    \Psi_{\varsigma+\zeta-2, \varepsilon}''(r) - K_{\varsigma,\zeta}^{-1} \quad\forall\,r\in\mathbb R.
    \]
   \end{itemize}
\end{lem}
\begin{proof}
    The proof is a direct consequence of the definition
    of $m$ and $\Psi$.
\end{proof}

\begin{lem}
\label{lem_app_G}
Let $d\in\{1,2,3\}$, $\delta>\frac d4$, and $\varsigma\geq2$.
Then, the following properties hold.
\begin{itemize}
 \item [(i)]  There exist a constant $C>0$
 and a positive sequence $(L_\lambda)_{\lambda>0}$
 such that, for every $\eps\in(0,1)$, it holds
 for every $v,v_1,v_2\in H$ that
    \begin{align*}
    \|G_{\varsigma, \varepsilon}(v)\|^2_{\cL^2(H,H)} &
    \le C \|(\operatorname{I}+A)^{-\delta}\|^2_{\cL^2(H, H)}
    \end{align*}
    and
    \begin{align*}
    \|G_{\varsigma, \varepsilon}(v_1)-
    G_{\varsigma, \varepsilon}(v_2)\|^2_{\cL^2(H,H)} 
    &\le \lambda\norm{v_1-v_2}_V^2 + L_\lambda 
    \norm{v_1-v_2}^2_H \quad \forall\,\lambda>0.        
    \end{align*}
\item [(ii)] If $d=1$, for every $\sigma\in[0,\frac12)$, $\delta>\frac d4+\sigma$, 
and $\eta\in(\sigma,\frac12)\cap(\frac 14, \frac12)$
    there exists a positive constant $C$, 
    depending on $\s, \sigma, \eta, \delta$, such that 
    \[
    \|G_{\s,\eps}(v)\|^2_{\cL^2(H, V_{2\sigma})} \le C
    \|v\|^2_{V_{2\eta}} \quad\forall\,
    v\in B_1^\infty \cap V_{2 \eta}.
    \]
    \end{itemize}
\end{lem}

\begin{proof}
    The statements are direct consequences of Lemma~\ref{lem_app}
    and Lemmata \ref{reg_A}, \ref{G_reg}, \ref{G_basis}.
\end{proof}

The  approximated problem reads as
\beq
\label{eq_app}
  \d X_\varepsilon +A  X_\varepsilon\,\d t + 
  \Psi_{\beta, \varepsilon}'(X_\varepsilon)\, \d t
  = G_{\alpha, \varepsilon} (X_\varepsilon)\,\d W, 
  \qquad
  X_\varepsilon(0)=x.
\eeq
The well-posedness of \eqref{eq_app}
is established in the following result.
\begin{prop}
\label{ex_app_thm}
Under the assumptions of Theorem~\ref{th:wp},
for every $\eps\in(0,1)$ and for every $x$
satisfying \eqref{init}
there exists a unique process $X_\varepsilon^x$ such that, 
for every $T > 0$,
\begin{equation*}
    X_\varepsilon^x \in L^\ell_\cP(\Omega; C([0,T]; H)) \cap 
    L^\ell_\cP(\Omega; L^2(0,T; V)) \quad\forall\,\ell\geq1,
\end{equation*}
and, for every $t\in[0,T]$, $\P$-almost surely,
\begin{multline*}
 ( X_\eps^x(t),v)_H  + 
 \int_0^t (A^{\frac 12}X^x_\eps(s) , A^{\frac 12}v)_H \, {\rm d}s
 +\int_0^t (\Psi'_{\beta,\eps}(X^x_\eps(s)),v)_H  \, {\rm d}s\\
= (x,v )_H  + \left(\int_0^t G_{\alpha,\eps}(X^x_\eps(s))\d W (s), v\right)_H 
\quad\forall\,v\in V.
\end{multline*}
Moreover, if $d=1$,
it also holds, for every $\xi\in[\sigma,\frac12+\sigma)$
and $T>0$, that 
\begin{alignat*}{2}
    X_\varepsilon^x &\in L^\ell_\cP(\Omega; C([0,T]; V_{2\xi}))
    \quad\forall\,\ell\geq1, \quad&&\text{if } x\in V_{2\xi},\\
    X_\varepsilon^x &\in 
    L^\ell_\cP(\Omega; L^2(0,T; V_{2\sigma+1})) \quad\forall\,\ell\geq1, \quad&&\text{if } x\in V_{2\sigma}.
\end{alignat*}
\end{prop}
\begin{proof}
    Since $\Psi_{\beta, \varepsilon}'$ is monotone and Lipschitz continuous, it is easy to check that the operator $A +  \Psi_{\beta, \varepsilon}'$ satisfies, for any $\varepsilon \in (0,1)$, the classical conditions of Pardoux, Krylov, and Rozovski\u{\i} \cite{Pard, KR-spde,LiuRo}. Indeed,
    if we define the operator $A_\varepsilon :  V \to V^*$ as
\[
A_\varepsilon(v) := A v + \Psi'_{\beta, \varepsilon}(v), \quad v\in V,
\]
as a consequence of the Lipschitz-continuity of 
$\Psi_{\beta,\varepsilon}'$ 
one has that $A_\eps$ is hemicontinuous, i.e.
\[
s\mapsto \langle A_\varepsilon( v_1 + s v_2), v_3 \rangle, \quad s\in\erre,
\]
is continuous for every $v_1,v_2,v_3\in V$.
Moreover, thanks to the monotonicity of $\Psi_{\beta,\eps}'$
and to Lemma~\ref{lem_app_G}, 
$A_\eps$ is weakly monotone, in the sense that 
for every $v_1,v_2\in V$ it holds that
\begin{align*}
    &\langle A_\varepsilon( v_1) - A_\varepsilon( v_2), v_1 - v_2 \rangle 
    - \frac{1}{2} \| G_{\alpha, \varepsilon}(v_1) 
    - G_{\alpha, \varepsilon}(v_2) \|_{\cL^2(H; H)}^2
    \\
    &\geq \norm{\nabla(v_1-v_2)}_H^2
    -\lambda\norm{v_1-v_2}_V^2 - L_\lambda\norm{v_1-v_2}_H^2
    \\
    &\ge \frac12\norm{v_1-v_2}_V^2-(L_{1/2}+1)\|v_1-v_2\|^2_H.
\end{align*}
By taking $v_2=0$ and using that $\Psi_{\beta,\eps}'(0)=0$, 
it follows that $A_\eps$ is weakly coercive, 
i.e.~for all $v\in V$
\[
  \langle A_\varepsilon( v), v \rangle 
    - \frac{1}{2} \| G_{\alpha, \varepsilon}(v) \|_{\cL^2(H; H)}^2
    \geq \frac12\norm{v}_V^2-(L_{1/2}+1)\|v\|^2_H 
    - C\norm{(\operatorname{I}+A)^{-\delta}}^2_{\cL^2(H,H)}.
\]
Eventually, as a consequence of the $\frac{1}{\eps^{\varsigma}}$-Lipschitz continuity 
of $\Psi_{\beta,\eps}'$, the fact that $\Psi_{\beta,\eps}'(0)=0$,
and again Lemma~\ref{lem_app_G}, it follows that
$A_\eps$ is bounded, in the sense that for every $v_1,v_2\in V$
it holds that
\[
\langle  A_\varepsilon(v_1), v_2\rangle  
\le \|v_1\|_{V}\|v_2\|_V + \frac{1}{\varepsilon^{\varsigma}}\|v_1\|_H\|v_2\|_H.
\]
The well-posedness of \eqref{eq_app} follows then from the classical 
variational theory \cite{LiuRo} in the Hilbert triplet 
$(V,H,V^*)$.
Lastly, by the Lipschitz-continuity of $\Psi_{\beta,\eps}'$
one has that $\Psi_{\beta,\eps}'(\xex)
\in L^\ell_\cP(\Omega; L^2(0,T; V_{2\sigma}))$.
Moreover, 
since
\[
\|G_{\alpha, \varepsilon}(v)\|^2_{\cL^2(H, H)} 
    \le C \quad\text{a.e.~in } \Omega\times(0,T),
\]
by maximal regularity arguments 
on the mild formulation
we get $X_\eps^x\in L^\ell_\cP(\Omega; C([0,T]; V_{2\xi}))$
for all $\xi\in(0,\frac12)$:
for more details, see also Proposition~\ref{prop5Xeps}.
If $d=1$,
this implies by Lemma~\ref{lem_app_G}
that 
\[
    \|G_{\alpha, \varepsilon}(v)\|_{L^\ell(\Omega; L^\infty(0,T;\cL^2(H, V_{2\sigma}))} 
    \leq C,
\]
so that the conclusion follows also for 
$\xi\in[\frac12,\frac12+\sigma)$ again by 
maximal regularity.
\end{proof}

\subsection{Uniform estimates}
In this section we derive some a priori estimates on
the approximated solutions $(X^x_\eps)_\eps$, 
uniform with respect to $\varepsilon \in (0,1)$. These are collected in the following results.

\begin{prop}[First estimate]
\label{prop1Xeps}
    Under the assumptions of Theorem~\ref{th:wp},
    for every $\ell \ge 1$ and $T>0$ there exists 
    a constant $C>0$, only depending on $\alpha, \delta,\ell,T$,
    such that, for every $\eps\in(0,1)$,
\[
\|\xe \|_{L^{\ell}(\Omega; C([0,T]; H))} + \|\xe \|_{L^{\ell}(\Omega; L^2(0,T; V))}
+\|\Psi_{\beta,\eps}(\xe)\|_{L^{\frac\ell2}(\Omega; 
L^1(0,T; L^1(\Td)))} \le 
C.
\]
\end{prop}

\begin{proof}
It\^o formula for the square of the $H$-norm yields
\begin{align*}
    \frac 12 \|\xe(t)\|^2_H &+ 
    \int_0^t \| A^{\frac12}\xe(s)\|^2_H \, \d s+ 
    \int_0^t \left( \Psi_{\beta, \varepsilon}'(\xe(s)), 
    \xe (s)\right) \, \d s 
    \\
    &= \frac 12 \|x\|^2_H + \int_0^t \left( \xe (s), G_{\alpha, \varepsilon} (\xe (s)) \, \d W(s)\right)_H 
    + \frac 12 \int_0^t \|G_{\alpha, \varepsilon}(\xe(s))\|^2_{\cL^2(H,H)}\, \d s
\end{align*}
for every $t \in [0,T]$, $\mathbb{P}$-almost surely.
Since $\Psi_{\beta, \varepsilon}'$ is a monotone increasing 
function with $\Psi_{\beta, \varepsilon}'(0)=0$, 
the third term on the left-hand side is nonnegative and the Fenchel-Young inequality for a single-valued subdifferential yields
\begin{align*}
  \int_0^t \left( \Psi_{\beta, \varepsilon}'(\xe(s)), 
    \xe (s)\right) \, \d s &=
    \int_0^t \int_{\Td}\Psi_{\beta, \varepsilon}(\xe(s))\, \d s 
    +\int_0^t \int_{\Td}\Psi_{\beta,\eps}^*(
    \Psi'_{\beta, \varepsilon}(\xe(s)))\, \d s \\
    &\geq \int_0^t \int_{\Td}\Psi_{\beta, \varepsilon}(\xe(s))\, \d s.
\end{align*}
Hence, taking also Lemma~\ref{lem_app_G}
into account we get 
\begin{multline*}
\frac 12 \|\xe(t)\|^2_H
+ \int_0^t \| A^{\frac12}\xe(s)\|^2_H \, \d s
+\int_0^t \int_{\Td}\Psi_{\beta, \varepsilon}(\xe(s))\, \d s\\
\le  \frac 12 \|x\|^2_H + \int_0^t \left( \xe (s), G_{\alpha, \varepsilon} (\xe (s)) \, \d W(s)\right)  
    + \frac {CT}{2}\|(\operatorname{I}+A)^{-\delta}\|^2_{\cL^2(H,H)}.
\end{multline*}
We now raise to power $\frac\ell2$, take supremum in time, 
and then expectations. By 
using the Burkholder-Davis-Gundy and the Young inequalities,
together with Lemma~\ref{lem_app_G},
we note that 
\begin{align*}
&\E\sup_{t\in[0,T]}\left| 
\int_0^t \left( \xe (s), G_{\alpha, \varepsilon} (\xe (s)) \, \d W(s)\right)  \right|^{\frac\ell2}\\
&\qquad \lesssim_\ell \mathbb{E}\left(\sup_{s \in [0,T]} \|\xe (s)\|^2_H \int_0^T\|G_\alpha(\xe (s))\|^2_{\cL^2(H,H)}\,\d s \right)^{\frac{\ell}{4}}
\\
&\qquad\le \frac 1 4 \mathbb{E}\sup_{s \in [0,T]} \|\xe (s)\|^\ell_H + C (\|(\operatorname{I}+A)^{-\delta}\|^2_{\cL^2(H,H)} T)^{\frac{\ell}{2}}.
\end{align*}
Upon rearranging the terms, this concludes the proof.
\end{proof}

\begin{prop}[Second estimate]
\label{prop3Xeps}
Under the assumptions of Theorem~\ref{th:wp},
for every $\ell \ge 2$  and $T>0$
there exists a constant $C>0$,
only depending on $\alpha,\beta,\gamma, \delta, \sigma, \ell,T$,
such that
\begin{multline*}
\|\Psi_{\gamma, \varepsilon}(\xe)\|_{
L^{\ell}(\Omega; L^\infty(0,T; L^1(\Td))} + 
\|\Psi_{\gamma, \varepsilon}'(\xe)
\Psi_{\beta, \varepsilon}'(\xe)\|_{
L^{\ell}(\Omega; L^1(0,T; L^1(\Td))} \\
+\norm{\Psi_{\frac\gamma2}''(\xe)\nabla\xe}_{
L^{\ell}(\Omega; L^2(0,T; H))}
\le C\left(1+
\|\Psi_{\gamma}(x)\|_{L^1(\Td)}\right).
\end{multline*}
\end{prop}

\begin{proof}
By proceeding similarly to \cite[App.~A]{Bertacco21}, 
the It\^o formula yields,
for every $t \in [0,T]$, $\mathbb{P}$-almost surely, that
\begin{multline*}
\int_{\Td} \Psi_{\gamma, \varepsilon}(\xe(t)) 
+\int_0^t \int_{\Td}\Psi_{\gamma, \varepsilon}''(\xe(s)) |\nabla \xe(s)|^2\, \d s 
+ \int_0^t \int_{\Td}\Psi_{\gamma, \varepsilon}'(\xe(s))\Psi_{\beta, \varepsilon}'(\xe(s))\, \d s 
\\
= \int_{\Td}\Psi_{\gamma, \varepsilon}(x)
+ \int_0^t \left(\Psi_{\gamma, \varepsilon}'(\xe(s)), G_{\alpha, \varepsilon} (\xe(s))\, \d W(s)\right)_H 
\\
+
\frac 12 \int_0^t \sum_{k \in \mathbb{N}_+}
\int_{\Td} \Psi_{\gamma, \varepsilon}''(\xe(s))
m_{\alpha, \varepsilon}(\xe(s))|(\operatorname{I}+A)^{-\delta}e_k|^2\, \d s.
\end{multline*}
We split the proof in two parts by considering separately the 
two main 
cases $0 \leq \sigma \le \frac d 4$ and $\frac d4<\sigma< \frac 12$.
For both of them, we handle the regimes
$\gamma=1$, $\gamma\in(1,\alpha+\beta]$, and $\gamma>\alpha+\beta$.

\smallskip
\noindent\boxed{\text{Case $\frac d4<\sigma<\frac12$.}}
By the continuous Sobolev embedding $V_{2\sigma}\embed L^\infty(\Td)$, we infer that 
\begin{multline*}
\frac 12 \int_0^t \sum_{k \in \mathbb{N}_+}
\int_{\Td} \Psi_{\gamma, \varepsilon}''(\xe(s))
m_{\alpha, \varepsilon}(\xe(s))|(\operatorname{I}+A)^{-\delta}e_k|^2\, \d s\\
\le \frac 12\sum_{k \in \mathbb{N}_+}\|(\operatorname{I}+A)^{-\delta}e_k\|^2_{L^\infty(\Td)}\int_0^t\int_{\Td} 
\Psi_{\gamma, \varepsilon}''(\xe(s))
m_{\alpha, \varepsilon}(\xe(s))\, \d s \\
\leq C\|(\operatorname{I}+A)^{-\delta}\|^2_{\cL^2(H, V_{2\sigma})}
\int_0^t\int_{\Td} 
\Psi_{\gamma, \varepsilon}''(\xe(s))
m_{\alpha, \varepsilon}(\xe(s))\, \d s.
\end{multline*}
As far as the stochastic integral is concerned, 
by taking $\frac\ell2$-power, supremum in time, and expectations, 
the Burkholder-Davis-Gundy inequality yields
\begin{align*}
    &\E\sup_{t\in[0,T]}
    \left|
    \int_0^t \left(\Psi_{\gamma, \varepsilon}'(\xe(s)), G_{\alpha, \varepsilon} (\xe(s))\, \d W(s)\right)_H
    \right|^{\frac\ell2}\\
    &\qquad=
    \E\sup_{t\in[0,T]}
    \left|
    \int_0^t \left(\Psi_{\gamma, \varepsilon}'(\xe(s))
    m_{\frac\alpha2,\eps}(\xe(s)), 
    (\operatorname{I}+A)^{-\delta}\, \d W(s)\right)_H
    \right|^{\frac\ell2}\\
    &\qquad\leq C\E\left(
    \int_0^T\norm{\Psi_{\gamma, \varepsilon}'(\xe(s))
    m_{\frac\alpha2,\eps}(\xe(s))}_{L^1(\Td)}^2
    \sum_{k\in\mathbb Z^d}\norm{(\operatorname{I}+A)^{-\delta}e_k}_{L^\infty(\Td)}^2\,\d s
    \right)^{\frac\ell4}\\
    &\qquad\leq 
    C\norm{(\operatorname{I}+A)^{-\delta}}^{\frac\ell2}_{\cL^2(H,V_{2\sigma})}
    \E
    \norm{\Psi_{\gamma, \varepsilon}'(\xe)
    m_{\frac\alpha2,\eps}(\xe)}_{L^2(0,T;L^1(\Td))}^{\frac\ell2}.
\end{align*}
By taking these estimates into account in It\^o formula we get 
\begin{align}
    \nonumber
    &\E\norm{\Psi_{\gamma,\eps}(\xe)}_{L^\infty(0,T;L^1(\Td))}^{\frac\ell2}+
    \E\norm{\Psi''_{\frac\gamma2}(\xe)\nabla\xe}_{L^2(0,T; H)}^\ell+\E\norm{\Psi_{\gamma,\eps}'(\xe)
    \Psi_{\beta,\eps}'(\xe)}_{L^1(0,T; L^1(\Td))}^{\frac\ell2}\\
    \nonumber
    &\lesssim_{\delta,\sigma,\ell} \norm{\Psi_\gamma(x)}_{L^1(\Td)}^{\frac\ell2}+
    \E\norm{\Psi_{\gamma, \varepsilon}'(\xe)
    m_{\frac\alpha2,\eps}(\xe)}_{L^2(0,T; L^1(\Td))}^{\frac\ell2}\\
    \label{aux:2}
    &\qquad+
    \E\norm{\Psi_{\gamma,\eps}''(\xe)m_{\alpha,\eps}(\xe)}_{L^1(0,T; L^1(\Td))}^{\frac\ell2}.
\end{align}

\smallskip
\noindent\underline{\sc Regime $\gamma=1$.}
Since $\alpha\geq2$, it is immediate to see that $\Psi_{\gamma,\eps}'m_{\frac\alpha2,\eps}$
and $\Psi''_{\gamma,\eps}m_{\alpha,\eps}$
are bounded in $L^\infty(\mathbb R)$, 
uniformly in $\eps$, so the thesis follows.

\smallskip
\noindent\underline{\sc Regime $1<\gamma\leq\alpha+\beta$.}
On the left-hand side of \eqref{aux:2}
one has, by Lemma~\ref{lem_app}-(ix) that 
\begin{equation}
\label{aux_est}
\E\norm{\Psi_{\gamma,\eps}'(\xe)
    \Psi_{\beta,\eps}'(\xe)}_{L^1(0,T; L^1(\Td))}^{\frac\ell2}
    \geq K\E\norm{\Psi_{\gamma+\beta-2,\eps}''(\xe)
    }_{L^1(0,T; L^1(\Td))}^{\frac\ell2}-K'
\end{equation}
for some positive constants $K,K'$ independent of $\eps$.
On the right-hand side, noting that 
$\Psi_{\gamma,\eps}''m_{\alpha,\eps}$ is bounded
in $L^\infty(\mathbb R)$ uniformly
with respect to $\eps$ if $\gamma\leq\alpha$,
it holds that 
\begin{alignat*}{2}
\E\norm{\Psi_{\gamma,\eps}''(\xe)m_{\alpha,\eps}(\xe)}_{L^1(0,T; L^1(\Td))}^{\frac\ell2}
&\leq C \qquad&&\text{if } \gamma\leq\alpha,\\
\E\norm{\Psi_{\gamma,\eps}''(\xe)m_{\alpha,\eps}(\xe)}_{L^1(0,T; L^1(\Td))}^{\frac\ell2}
&\lesssim
\E\norm{\Psi_{\gamma-\alpha,\eps}''(\xe)
}_{L^1(0,T; L^1(\Td))}^{\frac\ell2}
\qquad&&\text{if } \gamma>\alpha.
\end{alignat*}
In any case, since $\gamma-\alpha<\gamma+\beta-2$
(as a consequence of the fact that $\alpha+\beta\geq3$ by assumption), by the Young inequality 
for every $\s>0$ there exists $C_\s>0$, independent of $\eps$,
such that 
\[
\E\norm{\Psi_{\gamma,\eps}''(\xe)m_{\alpha,\eps}(\xe)}_{L^1(0,T; L^1(\Td))}^{\frac\ell2}
\leq\s\E\norm{\Psi_{\gamma+\beta-2,\eps}''(\xe)
    }_{L^1(0,T; L^1(\Td))}^{\frac\ell2} + C_\s.
\]
In an analogous way, since $\gamma>1$,
by the Young inequality 
for every $\zeta>0$ there exists $C_\zeta>0$, 
independent of $\eps$,
such that 
\[
  \E\norm{\Psi_{\gamma, \varepsilon}'(\xe)
    m_{\frac\alpha2,\eps}(\xe)}_{L^2(0,T; L^1(\Td))}^{\frac\ell2}
    \lesssim \zeta\E\norm{\Psi_{\gamma-1, \varepsilon}''(\xe)
    m_{\frac\alpha2,\eps}(\xe)}_{L^2(0,T; H)}^{\ell}+C_\zeta,
\]
where, similarly as before, 
\begin{alignat*}{2}
\E\norm{\Psi_{\gamma-1,\eps}''(\xe)m_{\frac\alpha2,\eps}(\xe)}_{L^2(0,T; H)}^{\ell}
&\leq C \qquad&&\text{if } \gamma\leq1+\frac\alpha2,\\
\E\norm{\Psi_{\gamma-1,\eps}''(\xe)m_{\frac\alpha2,\eps}(\xe)}_{L^2(0,T; H)}^{\ell}
&\lesssim
\E\norm{\Psi_{2\gamma-2-\alpha,\eps}''(\xe)
}_{L^1(0,T; L^1(\Td))}^{\frac\ell2}
\qquad&&\text{if } \gamma>1+\frac\alpha2.
\end{alignat*}
Since $2\gamma-2-\alpha\leq\gamma+\beta-2$
(as a consequence of the fact that $\gamma\leq\alpha+\beta$),
we get
\[
\E\norm{\Psi_{\gamma, \varepsilon}'(\xe)
    m_{\frac\alpha2,\eps}(\xe)}_{L^2(0,T; L^1(\Td))}^{\frac\ell2}
\leq\zeta\E\norm{\Psi_{\gamma+\beta-2,\eps}''(\xe)
    }_{L^1(0,T; L^1(\Td))}^{\frac\ell2} + C_\zeta.
\]
By putting all together into \eqref{aux:2} and bearing in mind \eqref{aux_est},
the estimate follows by choosing $\zeta$ and $\varsigma$
sufficiently small.

\smallskip
\noindent\underline{\sc Regime $\gamma>\alpha+\beta$.}
The same argument of the case $1<\gamma\leq\alpha+\beta$ 
continues to work, except for the second term
on the right-hand side of \eqref{aux:2}. 
To this end, 
since now $\gamma>\alpha+\beta\geq3$ and $\ell\geq4$, 
by Lemma~\ref{lem_app}-(iv)-(v) and the Jensen inequality we have that
\begin{align*}
    &\E\norm{\Psi_{\gamma, \varepsilon}'(\xe)
    m_{\frac\alpha2,\eps}(\xe)}_{L^2(0,T; L^1(\Td))}^{\frac\ell2}\\
    &\qquad\lesssim\E\left(\int_0^T\norm{
    \Psi''_{\gamma-1,\eps}(\xe(s))m_{\frac\alpha2,\eps}(\xe(s))}_{L^1(\Td)}^2\,\d s\right)^{\frac\ell4}\\
    &\qquad\lesssim\E\int_0^T
    \norm{\Psi_{\gamma,\eps}(\xe(s))
    \frac{\Psi''_{\gamma-1,\eps}(\xe(s))
    m_{\frac\alpha2,\eps}(\xe(s))}
    {\Psi_{\gamma,\eps}(\xe(s))}}_{L^1(\Td)}^{\frac\ell2}\,\d s\\
    &\qquad\lesssim\E\int_0^T
    \norm{\Psi_{\gamma,\eps}(\xe(s))
    \frac{\Psi''_{\gamma-1-\frac\alpha2,\eps}(\xe(s))}
    {\Psi''_{\gamma-2,\eps}(\xe(s))}}_{L^1(\Td)}^{\frac\ell2}\,\d s
\end{align*}
By noting that $\gamma-1-\frac\alpha2\leq\gamma-2$
as a consequence of the fact that $\alpha\geq2$ by assumption, 
one has that $\Psi''_{\gamma-1-\frac\alpha2,\eps}/\Psi''_{\gamma-2,\eps}$
is bounded in $L^\infty(\mathbb R)$, uniformly in $\eps$.
Hence, 
we infer that 
\[
\E\norm{\Psi_{\gamma, \varepsilon}'(\xe)
    m_{\frac\alpha2,\eps}(\xe)}_{L^2(0,T; L^1(\Td))}^{\frac\ell2}
    \lesssim\int_0^T\E
\norm{\Psi_{\gamma,\eps}(\xe(s))}_{L^1(\Td)}^{\frac\ell2}\,\d s.
\]
Since $T$ is arbitrary in \eqref{aux:2}, the estimate follows by the Gronwall lemma.
This concludes the proof in the case $\frac{d}4<\sigma<\frac12$.

\smallskip
\noindent\boxed{\text{Case $0\leq\sigma\leq\frac{d}4$.}}
Recall that by the Sobolev embeddings it holds that
$V_{2\sigma}\embed L^{\frac{2}{1-4\sigma/d}}(\Td)$
if $\sigma\in[0,\frac d4)$ and 
$V_{2\sigma}\embed L^{q}(\Td)$ for every $q\in[1,+\infty)$
if $\sigma=\frac d4$. Hence, by the H\"older inequality we get 
\begin{multline*}
\frac 12 \int_0^t \sum_{k \in \mathbb{N}_+}
\int_{\Td} \Psi_{\gamma, \varepsilon}''(\xe(s))
m_{\alpha, \varepsilon}(\xe(s))|(\operatorname{I}+A)^{-\delta}e_k|^2\, \d s\\
\le \frac 12\sum_{k \in \mathbb{N}_+}\|(\operatorname{I}+A)^{-\delta}e_k\|^2_{L^{q_0}(\Td)}\int_0^t
\norm{ 
\Psi_{\gamma, \varepsilon}''(\xe(s))
m_{\alpha, \varepsilon}(\xe(s))}_{L^{q_1}(\Td)}\, \d s \\
\leq C\|(\operatorname{I}+A)^{-\delta}\|^2_{\cL^2(H, V_{2\sigma})}
\int_0^t
\norm{ 
\Psi_{\gamma, \varepsilon}''(\xe(s))
m_{\alpha, \varepsilon}(\xe(s))}_{L^{q_1}(\Td)}\, \d s,
\end{multline*}
where $q_0=
\frac{2}{1-4\sigma/d}$
and
$q_1=\frac{q_0}{q_0-2}=\frac{d}{4\sigma}$   
if $\sigma\in[0,\frac d4)$, while
$q_0\in(2,+\infty)$ is arbitrary and
$q_1=\frac{q_0}{q_0-2}\in(1,2)$
if $\sigma=\frac{d}4$.
Note that $q_1=+\infty$ if $\sigma=0$.
As far as the stochastic integral is concerned, 
by taking $\frac\ell2$-power, supremum in time, and expectations, 
the Burkholder-Davis-Gundy inequality yields
\begin{align*}
    &\E\sup_{t\in[0,T]}
    \left|
    \int_0^t \left(\Psi_{\gamma, \varepsilon}'(\xe(s)), G_{\alpha, \varepsilon} (\xe(s))\, \d W(s)\right)_H
    \right|^{\frac\ell2}\\
    &\qquad=
    \E\sup_{t\in[0,T]}
    \left|
    \int_0^t \left(\Psi_{\gamma, \varepsilon}'(\xe(s))
    m_{\frac\alpha2,\eps}(\xe(s)), 
    (\operatorname{I}+A)^{-\delta}\, \d W(s)\right)_H
    \right|^{\frac\ell2}\\
    &\qquad\leq C\E\left(
    \int_0^T\norm{\Psi_{\gamma, \varepsilon}'(\xe(s))
    m_{\frac\alpha2,\eps}(\xe(s))}_{L^{q_2}(\Td)}^2
    \sum_{k\in\mathbb Z^d}\norm{(\operatorname{I}+A)^{-\delta}e_k}_{L^{q_0}(\Td)}^2\,\d s
    \right)^{\frac\ell4}\\
    &\qquad\leq 
    C\norm{(\operatorname{I}+A)^{-\delta}}^{\frac\ell2}_{\cL^2(H,V_{2\sigma})}
    \E
    \norm{\Psi_{\gamma, \varepsilon}'(\xe)
    m_{\frac\alpha2,\eps}(\xe)}_{L^2(0,T;L^{q_2}(\Td))}^{\frac\ell2},
\end{align*}
where $q_2=\frac{q_0}{q_0-1}=\frac{2}{1+4\sigma/d}\in(1,2]$ 
if $\sigma\in[0,\frac d4)$, while
$q_2=\frac{q_0}{q_0-1}\in(1,2)$
if $\sigma=\frac d4$.
By taking these estimates into account in It\^o formula we get 
\begin{align}
    \nonumber
    &\E\norm{\Psi_{\gamma,\eps}(\xe)}_{L^\infty(0,T;L^1(\Td))}^{\frac\ell2}+
    \E\norm{\Psi''_{\frac\gamma2}(\xe)\nabla\xe}_{L^2(0,T; H)}^\ell+\E\norm{\Psi_{\gamma,\eps}'(\xe)
    \Psi_{\beta,\eps}'(\xe)}_{L^1(0,T; L^1(\Td))}^{\frac\ell2}\\
    \nonumber
    &\lesssim_{\delta,\sigma,\ell} \norm{\Psi_\gamma(x)}_{L^1(\Td)}^{\frac\ell2}+
    \E\norm{\Psi_{\gamma, \varepsilon}'(\xe)
    m_{\frac\alpha2,\eps}(\xe)}_{L^2(0,T;L^{q_2}(\Td))}^{\frac\ell2}\\
    \label{aux:3}
    &\qquad+
    \E\norm{\Psi_{\gamma,\eps}''(\xe)m_{\alpha,\eps}(\xe)}_{L^1(0,T; L^{q_1}(\Td))}^{\frac\ell2}.
\end{align}

\smallskip
\noindent\underline{\sc Regime $\gamma=1$.}
Since $\alpha\geq2$, it is immediate to see that $\Psi_{\gamma,\eps}'m_{\frac\alpha2,\eps}$
and $\Psi''_{\gamma,\eps}m_{\alpha,\eps}$
are bounded in $L^\infty(\mathbb R)$, 
uniformly in $\eps$, so the thesis follows.

\smallskip
\noindent\underline{\sc Regime $1<\gamma\leq\alpha+\beta$.}
On the left-hand side of \eqref{aux:3}
one has, by Lemma~\ref{lem_app}-(ix) that 
\[
\E\norm{\Psi_{\gamma,\eps}'(\xe)
    \Psi_{\beta,\eps}'(\xe)}_{L^1(0,T; L^1(\Td))}^{\frac\ell2}
    \geq K\E\norm{\Psi_{\gamma+\beta-2,\eps}''(\xe)
    }_{L^1(0,T; L^1(\Td))}^{\frac\ell2}-K'
\]
for some positive constants $K,K'$ independent of $\eps$.
On the right-hand side, noting that 
$\Psi_{\gamma,\eps}''m_{\alpha,\eps}$ is bounded
in $L^\infty(\mathbb R)$ uniformly
with respect to $\eps$ if $\gamma\leq\alpha$,
it holds that 
\begin{alignat*}{2}
\E\norm{\Psi_{\gamma,\eps}''(\xe)m_{\alpha,\eps}(\xe)}_{L^1(0,T; L^{q_1}(\Td))}^{\frac\ell2}
&\leq C \qquad&&\text{if } \gamma\leq\alpha,\\
\E\norm{\Psi_{\gamma,\eps}''(\xe)m_{\alpha,\eps}(\xe)}_{L^1(0,T; L^{q_1}(\Td))}^{\frac\ell2}
&\lesssim
\E\norm{\Psi_{\gamma-\alpha,\eps}''(\xe)
}_{L^1(0,T; L^{q_1}(\Td))}^{\frac\ell2}
\qquad&&\text{if } \gamma>\alpha.
\end{alignat*}
If $\sigma=0$ then $\gamma\leq\alpha$ by assumption, so 
only the former case applies. If $\sigma\in(0,\frac d4)$,
in the latter case, since $q_1>1$, by the Young inequality 
we infer that for every $\s>0$ there exists $C_\s>0$, 
independent of $\eps$, such that 
\begin{align*}
  \E\norm{\Psi_{\gamma-\alpha,\eps}''(\xe)
}_{L^1(0,T; L^{q_1}(\Td))}^{\frac\ell2}
&=\E\left(\int_0^T\norm{\Psi''_{q_1(\gamma-\alpha),\eps}(\xe(s))}_{L^1(\Td)}^{\frac1{q_1}}\,\d s
\right)^{\frac\ell2}\\
&\leq
\s\E\norm{\Psi''_{q_1(\gamma-\alpha),\eps}(\xe)}_{
L^1(0,T; L^1(\Td))}^{\frac\ell2}+C_\s,
\end{align*}
where $q_1(\gamma-\alpha)\leq\gamma+\beta-2$:
indeed, if $\sigma\in(0,\frac d4)$, this is a consequence of 
the facts that $q_1=\frac{d}{4\sigma}$ and 
that $\gamma\leq\frac{\alpha+4\beta\sigma/d-8\sigma/d}{1-4\sigma/d}$ by assumption, whereas if $\sigma=\frac d4$
it follows from the facts that $q_1$ can be taken
arbitrarily close to $1$ (i.e.~by taking $q_0$ arbitrarily large) and 
$\gamma-\alpha<\gamma+\beta-2$.
In any case, we infer that
\[
\E\norm{\Psi_{\gamma,\eps}''(\xe)m_{\alpha,\eps}(\xe)}_{L^1(0,T; L^{q_1}(\Td))}^{\frac\ell2}
\leq\s\E\norm{\Psi_{\gamma+\beta-2,\eps}''(\xe)
    }_{L^1(0,T; L^1(\Td))}^{\frac\ell2} + C_\s.
\]
In an analogous way, since $\gamma>1$ and 
$q_2<2$, one has, for every $\zeta>0$,
\[
  \E\norm{\Psi_{\gamma, \varepsilon}'(\xe)
    m_{\frac\alpha2,\eps}(\xe)}_{L^2(0,T; L^{q_2}(\Td))}^{\frac\ell2}
    \leq\zeta \E\norm{\Psi_{\gamma-1, \varepsilon}''(\xe)
    m_{\frac\alpha2,\eps}(\xe)}_{L^2(0,T; H)}^{\ell}+C_\zeta,
\]
where, similarly as before, 
\begin{alignat*}{2}
\E\norm{\Psi_{\gamma-1,\eps}''(\xe)m_{\frac\alpha2,\eps}(\xe)}_{L^2(0,T; H)}^{\ell}
&\leq C \qquad&&\text{if } \gamma\leq1+\frac\alpha2,\\
\E\norm{\Psi_{\gamma-1,\eps}''(\xe)m_{\frac\alpha2,\eps}(\xe)}_{L^2(0,T; H)}^{\ell}
&\lesssim
\E\norm{\Psi_{2\gamma-2-\alpha,\eps}''(\xe)
}_{L^1(0,T; L^1(\Td))}^{\frac\ell2}
\qquad&&\text{if } \gamma>1+\frac\alpha2.
\end{alignat*}
Since $2\gamma-2-\alpha\leq\gamma+\beta-2$
(as a consequence of the fact that $\gamma\leq\alpha+\beta$),
for every $\zeta>0$ there exists $C_\zeta>0$, 
independent of $\eps$,
such that 
\[
\E\norm{\Psi_{\gamma, \varepsilon}'(\xe)
    m_{\frac\alpha2,\eps}(\xe)}_{L^2(0,T; L^{q_2}(\Td))}^{\frac\ell2}
\leq\zeta\E\norm{\Psi_{\gamma+\beta-2,\eps}''(\xe)
    }_{L^1(0,T; L^1(\Td))}^{\frac\ell2} + C_\zeta.
\]
By putting all together into \eqref{aux:3},
the estimate follows by choosing $\zeta$ and $\varsigma$
sufficiently small.

\smallskip
\noindent\underline{\sc Regime $\gamma>\alpha+\beta$.}
The same argument of the case $1<\gamma\leq\alpha+\beta$ 
continues to work, except for the second term
on the right-hand side of \eqref{aux:3}. 
To this end, 
since now $\gamma>\alpha+\beta\geq3$ and $\ell\geq4$, 
by Lemma~\ref{lem_app}-(iv)-(v) and the Jensen inequality we have that
\begin{align*}
    &\E\norm{\Psi_{\gamma, \varepsilon}'(\xe)
    m_{\frac\alpha2,\eps}(\xe)}_{L^2(0,T; L^{q_2}(\Td))}^{\frac\ell2}\\
    &\qquad\lesssim\E\left(\int_0^T\norm{
    \Psi''_{\gamma-1,\eps}(\xe(s))m_{\frac\alpha2,\eps}(\xe(s))}_{L^{q_2}(\Td)}^2\,\d s\right)^{\frac\ell4}\\
    &\qquad\lesssim1+\E\left(\int_0^T\norm{
    \Psi''_{q_2(\gamma-1),\eps}(\xe(s))
    m_{q_2\frac\alpha2,\eps}(\xe(s))}_{L^{1}(\Td)}^2\,\d s\right)^{\frac\ell4}\\
    &\qquad\lesssim1+\E\int_0^T
    \norm{\Psi_{\gamma,\eps}(\xe(s))
    \frac{\Psi''_{q_2(\gamma-1),\eps}(\xe(s))
    m_{q_2\frac\alpha2,\eps}(\xe(s))}
    {\Psi_{\gamma,\eps}(\xe(s))}}_{L^1(\Td)}^{\frac\ell2}\,\d s\\
    &\qquad\lesssim\E\int_0^T
    \norm{\Psi_{\gamma,\eps}(\xe(s))
    \frac{\Psi''_{q_2(\gamma-1-\frac\alpha2),\eps}(\xe(s))}
    {\Psi''_{\gamma-2,\eps}(\xe(s))}}_{L^1(\Td)}^{\frac\ell2}\,\d s.
\end{align*}
Note that if $\sigma\in[0,\frac d4)$ then
$q_2(\gamma-1-\frac\alpha2)=\frac{2\gamma-2-\alpha}{1+4\sigma/d}\leq\gamma-2$
as a consequence of the fact that 
$(1-4\sigma/d)\gamma\leq\alpha-8\sigma/d$
by assumption, 
whereas if $\sigma=\frac d4$
one can choose $q_2\in(0,1)$ arbitrarily close to $1$
(again, by taking $q_0$ large enough)
such that $q_2(\gamma-1-\frac\alpha2)\leq \gamma-2$
since $\alpha>2$ by assumption.
In any case, it follows that
$\Psi''_{q_2(\gamma-1-\frac\alpha2),\eps}/\Psi''_{\gamma-2,\eps}$
is bounded in $L^\infty(\mathbb R)$, uniformly in $\eps$.
Hence, 
we infer that 
\[
\E\norm{\Psi_{\gamma, \varepsilon}'(\xe)
    m_{\frac\alpha2,\eps}(\xe)}_{L^2(0,T; L^{q_2}(\Td))}^{\frac\ell2}
    \lesssim\int_0^T\E
\norm{\Psi_{\gamma,\eps}(\xe(s))}_{L^1(\Td)}^{\frac\ell2}\,\d s.
\]
Since $T$ is arbitrary in \eqref{aux:3}, the estimate follows by the Gronwall lemma.
This concludes the proof also in the case $0\leq \sigma\leq\frac{d}4$.
\end{proof}

\begin{prop}[Third estimate]
\label{prop4Xeps}
Under the assumptions of Theorem~\ref{th:wp},
for every $\ell \ge 1$, $\rho\in(0,\frac12)$, and $T>0$
there exists a constant $C>0$,
only depending on $\alpha, \delta, \sigma, \ell,\rho,T$,
such that, for every $\eps\in(0,1)$,
\begin{align*}
\|G_{\alpha,\eps}(\xe)\|_{L^{\ell}(\Omega; 
L^\infty(0,T; \cL^2(H,H))\cap L^2(0,T; \cL^2(H,V_{2\sigma})))}
&\le  
C,
\end{align*}
and
\begin{align*}
\norm{\int_0^\cdot
G_{\alpha,\eps}(\xe(s))\,\d W(s)}_{
L^{\ell}(\Omega; W^{\rho,\ell}(0,T;H))
\cap L^2(\Omega; W^{\rho,2}(0,T; V_{2\sigma}))} &\le  
C,
\end{align*}
\end{prop}
\begin{proof}
  The first estimate is a consequence of Proposition~\ref{prop1Xeps}
  and Lemma~\ref{lem_app_G}.
  The second estimate follows then from 
the well-known result 
  in \cite[Lem.~2.1]{fland-gat}.
\end{proof}

\begin{prop}[Fourth estimate]
\label{prop5Xeps}
Under the assumptions of Theorem~\ref{th:wp},
let $d=1$.
Then, if $\gamma\geq\beta$ and 
$x\in V_{2\xi}$ for some
$\xi\in(0,\frac12)$, there exists a constant $C>0$,
only depending on 
$\alpha,\beta,\gamma, \delta, \sigma, \ell,T, \xi$,
such that
\begin{align*}
\|\xe\|_{
L^{\ell}(\Omega; C([0,T]; V_{2\xi}))} 
&\le C\left(1+\|x\|_{V_{2\xi}}+
\|\Psi_{\gamma}(x)\|_{L^1(\Td)}\right),
\end{align*}
while if $\gamma\geq2\beta$ and 
$x\in V_{2\xi}$ for some
$\xi\in[\sigma,\frac12+\sigma)$, 
there exists a constant $C>0$,
only depending on 
$\alpha,\beta,\gamma, \delta, \sigma, \ell,T, \xi$,
such that
\begin{align*}
\|\xe\|_{
L^{\ell}(\Omega; C([0,T]; V_{2\xi})
\cap L^2(0,T; V_{2\sigma+1}))} 
&\le C\left(1+\|x\|_{V_{2\sigma}}+
\|\Psi_{\gamma}(x)\|_{L^1(\Td)}\right).
\end{align*}
\end{prop}
\begin{proof}
    First of all, note that $\xex$ satisfies the mild formulation
    \[
        \xex(t)+
        \int_0^tS(t-s)\Psi_{\beta,\eps}'(\xex(s))\,\d s
        =S(t)x+\int_0^tS(t-s)G_{\alpha,\eps}(\xex(s))\,\d W(s)
    \]
    for every $t\in[0,T]$, $\P$-almost surely.
    If $\gamma\geq\beta$, one has 
    by Proposition~\ref{prop3Xeps} that 
    \[
    \norm{\Psi_{\beta,\eps}'(\xex)}_{L^\ell(\Omega; L^2(0,T; H))}\lesssim1+\norm{\Psi_\gamma(x)}_{L^1(\Td)},
    \]
    whereas if $\gamma\geq2\beta$, noting that 
    \begin{align*}
    \|\Psi'_{\beta,\eps}(X^x_\eps)\|_{V_{2\sigma}} 
    &\lesssim
    \|\Psi'_{\beta, \eps}(X^x_\eps)\|_H
    +
    \|\Psi'_{\beta, \eps}(X^x_\eps)\|_H^{1-2\sigma} 
    \|\nabla\Psi'_{\beta, \eps}(X^x_\eps)\|_H^{2\sigma} 
    \\
    &\lesssim \|\Psi'_{\beta, \eps}(X^x_\eps)\|_H
    + \|\Psi'_{\beta, \eps}(X^x_\eps)\|_H^{1-2\sigma} 
    \|\Psi''_{\beta, \eps}(X^x_\eps)\nabla X^x\|_H^{2\sigma} 
    \\
    & \lesssim\|\Psi_{2\beta, \eps}(X^x_\eps)\|_{L^1(\Td)}^{\frac12}
    + \|\Psi_{2\beta, \eps}(X^x_\eps)\|_{L^1(\Td)}^{\frac12-\sigma} \|\Psi''_{\frac{2\beta}{2}, \eps}(X^x_\eps)\nabla X^x_\eps\|^{2\sigma}_H,
\end{align*}
   by Proposition~\ref{prop3Xeps}
    one infers that 
    \[
    \norm{\Psi_{\beta,\eps}'(\xex)}_{L^\ell(\Omega; L^{\frac1\sigma}(0,T; V_{2\sigma}))}\lesssim1+\norm{\Psi_\gamma(x)}_{L^1(\Td)}.
    \]
    Together with Proposition~\ref{prop4Xeps}, 
    this allows to conclude for every $\xi\in(0,\frac12)$ 
    by the maximal regularity results
\cite[Thm.~3.3, Ex.~3.2]{VNVW} 
and \cite[Prop.~A.24]{dapratozab}.
Lastly, if $\xi\in[\frac12, \frac12+\sigma)$
in the case $\gamma\geq2\beta$,
the already proved estimate for $\xi\in(0,\frac12)$
and Lemma~\ref{lem_app_G} allow
to obtain by bootstrapping
that 
\[
    \norm{G_{\alpha,\eps}(\xex)}_{L^\ell(\Omega; L^\infty(0,T; \cL^2(H,V_{2\sigma})))}\lesssim 1+\|x\|_{V_{2\xi}}+
    \|\Psi_\gamma(x)\|_{L^1(\Td)},
    \]
and one can conclude again by maximal regularity.
\end{proof}

\subsection{Existence}
In this section we perform the passage to the limit as $\eps\searrow0$ and prove the existence part 
of Theorem~\ref{th:wp}. We employ a stochastic compactness 
argument.

First of all, note that the following compact inclusions hold
(see e.g.~\cite[Cor.~4-5]{simon}):
\begin{alignat*}{2}
  L^2(0,T; V)
  \cap W^{\rho,\ell}(0,T; V_2^*)&\cembed 
L^2(0,T; V_{1-s}) \quad&&\forall\,s>0,\\
C([0,T]; H)
  \cap W^{\rho,\ell}(0,T; V_2^*)&\cembed 
C([0,T]; V_{-s}) \quad&&\forall\,s>0, \quad\text{if } 
\rho\ell>1.
\end{alignat*}
Now, let us fix $\rho\in(0,\frac12)$ and $\ell\geq2$ such that $\rho\ell>1$. Then, as a consequence of 
Propositions~\ref{prop1Xeps}
and \ref{prop4Xeps}, it follows that the sequence of 
probability distributions of 
\[
  \left(\xe,
  \int_0^\cdot G_{\alpha,\eps}(\xe(s))\,\d W(s), 
  W\right)_\eps
\]
is tight on the product space 
\[
  \left[L^2(0,T; V_{1-s})\cap 
  C^0([0,T]; V_{-s})\right]\times
  \left[L^2(0,T; H)\cap C([0,T]; V_{-s})\right]
  \times C([0,T]; \tilde H),
\]
for every $s>0$, where $\tilde H$ is a Hilbert-Schmidt 
extension of $H$.
Furthermore, by Proposition~\ref{prop3Xeps}
and the fact that $N_{\gamma,\beta}$ is superlinear,
it follows that $(\Psi_\beta'(\xe))_\eps$
is uniformly integrable on $\Omega\times(0,T)\times\Td$,
hence also relatively weakly compact in $L^1_\cP(\Omega; L^1(0,T; L^1(\Td)))$ by the Dunford-Pettis theorem.

By the classical Prokhorov and Skorokhod theorems
(see \cite[Thm.~2.7]{ike-wata})
and their weaker version by Jakubowski-Skorokhod (see e.~g.~\cite[Thm.~2.7.1]{breit-feir-hof}),
recalling all the estimates 
in Propositions~\ref{prop1Xeps}--\ref{prop3Xeps},
and \ref{prop4Xeps}, we infer that 
there exists a probability space $(\hat\Omega,\hat{\mathcal F}, \hat\P)$ and measurable 
maps $\Lambda_\eps:(\hat\Omega,\hat{\mathcal F})
\to(\Omega,\mathcal F)$ such that 
$\Lambda_\eps{\#}\hat \P=\P$ for every $\eps\in(0,1)$,
and, as $\eps\searrow0$,
\begin{align*}
  \hat X^x_\eps:=\xe\circ\Lambda_\eps\to \hat X^x
  \qquad&\text{in } 
  L^2(0,T; V_{1-s})\cap 
  C([0,T]; V_{-s})
  \quad\hat\P\text{-a.s.},\\
  \hat X^x_\eps\wstarto \hat X^x
  \qquad&\text{in } L^\ell(\hat\Omega; L^\infty(0,T; H)\cap L^2(0,T; V)),\\
  \Psi_{\beta,\eps}'(\hat X^x_\eps)\wto\hat\xi
  \qquad&\text{in } L^{1}(\hat \Omega; L^1(0,T; L^1(\Td))),\\
  \hat I_{\eps}:=\left(\int_0^\cdot G_{\alpha,\eps}
  (\xe(s))\,\d W(s)\right)\circ\Lambda_\eps \to \hat I
  \qquad&\text{in } L^\ell(\hat \Omega; L^2(0,T; H)
  \cap C([0,T]; V_{-s}))\\
  \hat W_\eps:=W\circ\Lambda_\eps \to \hat W 
  \qquad&\text{in } L^\ell(\hat\Omega; C([0,T]; \tilde H)),
\end{align*}
where
\begin{align*}
  &\hat X^x \in L^\ell(\hat \Omega; C([0,T]; V_{-s})\cap 
  L^\infty(0,T; V)\cap L^2(0,T; V)),\\
  &\hat \xi \in L^{1}(\hat \Omega; L^1(0,T; L^1(\Td))),\\
  &\hat I \in L^\ell(\hat \Omega; L^2(0,T; H)
  \cap C([0,T]; V_{-s})),\\
  &\hat W \in L^\ell(\hat \Omega; C([0,T]; \tilde H)),
\end{align*}
for every $s>0$ and $\ell\geq2$. Moreover, thanks to Lemma~\ref{lem_app_G}
and Proposition~\ref{prop1Xeps}, for every $\lambda>0$ we have 
\begin{align*}
  &\norm{G_{\eps,\alpha}(\hat X^x_\eps)-G_\alpha(\hat X^x)}_{
  L^\ell(\hat\Omega; L^2(0,T; \cL^2(H,H)))}\\
  &\leq
  \norm{G_{\eps,\alpha}(\hat X^x_\eps)-G_{\eps,\alpha}(\hat X^x)}_{
  L^\ell(\hat\Omega; L^2(0,T; \cL^2(H,H)))}+
  \norm{G_{\eps,\alpha}(\hat X^x)-G_\alpha(\hat X^x)}_{
  L^\ell(\hat\Omega; L^2(0,T; \cL^2(H,H)))}\\
  &\leq\lambda\norm{\hat X^x_\eps-\hat X^x}_{
  L^\ell(\hat\Omega; L^2(0,T; V))}
  +C_\lambda\norm{\hat X^x_\eps-\hat X^x}_{
  L^\ell(\hat\Omega; L^2(0,T; H))}
  +\eps^{\frac\alpha2}\norm{(\operatorname{I}+A)^{-\delta}}_{\cL^2(H,H)}\\
  &\leq C(\lambda+\eps) + C_\lambda\norm{\hat X^x_\eps-\hat X^x}_{
  L^\ell(\hat\Omega; L^2(0,T; H))},
\end{align*}
so that the arbitrariness of $\lambda$ and the convergences above yield
\[
  G_{\alpha,\eps}(\hat X^x_\eps)
  \to G_\alpha(\hat X^x) \qquad\text{in }
  L^\ell(\hat\Omega; L^2(0,T; \cL^2(H,H)))
  \quad\forall\,\ell\geq2.
\]

Now, let us show that $\hat\xi=\Psi_\beta'(\hat X^x)$
almost everywhere in $\hat\Omega\times(0,T)\times\Td$.
To this end, denoting by $R_{\beta,\eps}:=(\operatorname{I}+\eps\Psi_\beta')^{-1}:\erre\to\erre$ the resolvent operator, one has that 
$\hat X^x_\eps-R_{\beta,\eps}(\hat X^x_\eps)=\eps\Psi_{\beta,\eps}'(\hat X^x_\eps)$.
Hence, recalling the convergences above,
by possibly arguing on subsequence
we may assume that $\hat X^x_\eps\to\hat X^x$
and $R_{\beta,\eps}(\hat X^x_\eps)\to \hat X^x$
almost everywhere in $\hat\Omega\times(0,T)\times\Td$,
ensuring in turn that $|\hat X^x|\leq1$ almost everywhere in
$\hat\Omega\times(0,T)\times\Td$.
Since $\Psi_{\beta,\eps}'=\Psi_\beta'\circ R_{\beta,\eps}$,
by continuity of $\Psi_\beta'$ one has also that 
\[
\Psi_{\beta,\eps}'(\hat X^x_\eps)\to
\hat\xi':=\Psi_\beta'(\hat X^x)\ind_{\{-1<\hat X^x<1\}}
+\infty\ind_{\{\hat X^x=1\}} - \infty\ind_{\{\hat X^x=-1\}}
\]
almost everywhere in $\hat\Omega\times(0,T)\times\Td$.
By the Fatou lemma and the fact that 
$(\Psi_{\beta,\eps}'(\hat X^x_\eps))_\eps$
is bounded in $L^1(\hat\Omega\times(0,T)\times\Td)$,
it follows that $\hat\xi'\in L^1(\hat\Omega\times(0,T)\times\Td)$,
hence also that $\hat\xi'=\Psi_\beta'(\hat X^x)$,
i.e.~that $|\hat X^x|<1$ almost everywhere 
in $\hat\Omega\times(0,T)\times\Td$ and 
$\Psi_{\beta,\eps}'(\hat X^x_\eps)\to\Psi_\beta'(\hat X)$
almost everywhere in $\hat\Omega\times(0,T)\times\Td$.
Since we know that $(\Psi_{\beta,\eps}'(\hat X^x_\eps))_\eps$
is uniformly integrable on $\hat\Omega\times(0,T)\times\Td$, 
this implies that $\Psi_{\beta,\eps}'(\hat X^x_\eps)
\to\Psi_\beta'(\hat X)$ in $L^1(\Omega\times(0,T)\times\Td)$
by the Vitali dominated convergence theorem.
This clearly ensures that $\hat\xi=\Psi'_\beta(\hat X^x)$
almost everywhere in $\hat\Omega\times(0,T)\times\Td$, as required.

The argument to pass to the limit as $\eps\searrow0$ on $\hat\Omega$
is fairly classical, so we only 
highlight the main points, and we refer to \cite{scarpa21}
for details. By defining the filtration
\[
  \hat{\mathcal F}_{\eps,t}:=\sigma\{\hat X^x_\eps(s), 
  \hat I_\eps(s), \hat W_\eps(s):s\in[0,t]\}\,, 
  \qquad t\geq0,
\]
and by exploiting \cite[\S~8.4, Thm.~8.2]{dapratozab},
we infer that $\hat I_\eps$ is the $H$-valued martingale given by
\[
  \hat I_{\eps}(t)=\int_0^t G_{\alpha,\eps}(\hat X^x_\eps(s))\,\d \hat W_\eps(s)\qquad\forall\,t\geq0.
\]
Moreover, by setting the filtration
\[
  \hat{\mathcal F}_{t}:=\sigma\{\hat X^x(s), 
  \hat I(s), \hat W(s):s\in[0,t]\}\,, 
  \qquad t\geq0,
\]
by using the strong convergences $\hat X^x_\eps\to\hat X^x$
and $G_{\alpha,\eps}(\hat X^x_\eps)\to G_{\alpha}(\hat X^x)$
proved above together again with 
\cite[\S~8.4, Thm.~8.2]{dapratozab}, it is possible to prove that 
\[
  \hat I(t)=\int_0^t G_{\alpha}(\hat X^x(s))\,\d \hat W(s)\qquad\forall\,t\geq0.
\]
The allows to let $\eps\searrow0$ and deduce that 
$\hat X^x$ is a solution to \eqref{eq_ast}
with respect to the stochastic basis 
$(\hat \Omega, \hat{\mathcal F}, (\hat{\mathcal F}_t)_{t\geq0},\hat \P, \hat W)$, i.e.~a probabilistically-weak solution to \eqref{eq_ast}.

In order to conclude, we employ a classical 
argument \`a la Yamada-Watanabe. More precisely, 
since we already know that pathwise uniquenss holds
for the problem \eqref{eq_ast}, as proved in Section~\ref{ssec:uniq}, we can employ the well-know 
result by Gy\"ongy--Krylov \cite[Lem.~1.1]{gyo-kry}
to infer that the strong convergences above hold also 
on the original probability space, namely
\begin{align*}
  \xe\to X^x
  \qquad&\text{in } 
  L^2(0,T; V_{1-s})\cap 
  C([0,T]; V_{-s})
  \quad\P\text{-a.s.},\\
  \xe\wstarto X^x
  \qquad&\text{in } L^\ell(\Omega; L^\infty(0,T; H)\cap L^2(0,T; V)),\\
  \Psi_{\beta,\eps}'(\xe)\wto\xi
  \qquad&\text{in } L^{1}(\Omega; L^1(0,T; L^1(\Td))),
\end{align*}
as well as
\begin{align*}
  \int_0^\cdot G_{\alpha,\eps}
  (\xe(s))\,\d W(s) \to \int_0^\cdot G_{\alpha}
  (X^x(s))\,\d W(s)
  \qquad&\text{in } L^\ell(\Omega; L^2(0,T; H)
  \cap C([0,T]; V_{-s})),
\end{align*}
where
\begin{align*}
  & X^x \in L^\ell(\Omega; C([0,T]; V_{-s})\cap 
  L^\infty(0,T; H)\cap L^2(0,T; V)),\\
  & \xi \in L^{1}(\Omega; L^1(0,T; L^1(\Td))),
\end{align*}
for every $s>0$ and $\ell\geq2$. Proceeding as above, 
it is immediate to show that $\xi=\Psi'_\beta(X^x)$
almost everywhere in $\Omega\times(0,T)\times\Td$, and 
to conclude that $X^x$ is (the unique) solution to \eqref{eq_ast},
in the sense of Definition~\ref{def_sol},
on the original probability space. 
If also $d=1$ and
$x\in V_{2\xi}$, the additional regularities of 
$X^x$ follow from 
Proposition~\ref{prop5Xeps}
and maximal regularity.
This concludes
the proof of Theorem~\ref{th:wp}.

\section{Differentiability with respect to the initial datum}
\label{sec:Y}
Throughout the section, we work under the assumption of Theorem~\ref{th:wp} and
use the same notation and setting
of Section~\ref{sec:X}. Moreover, 
$T>0$ and $\ell\geq2$ are fixed arbitrarily.

For every $\eps\in(0,1)$, the well-posedness result 
in Proposition~\ref{ex_app_thm} ensures that the 
approximated solution map
\[
  S_\eps:H\to L^\ell_\cP(\Omega; C([0,T];H)\cap L^2(0,T; V)),
  \qquad x\mapsto \xe, \quad x\in H,
\]
is well-defined and Lipschitz-continuous.
This section is devoted to the study of differentiability 
of the map $S_\eps$, and the proof of uniform estimates on its derivatives. We will not focus on the passage to 
the limit as $\eps\to0$, as the differentiability 
with respect to the initial datum for the limiting problem \eqref{eq_ast} will not be necessary to our purposes.

\subsection{Fr\'echet differentiability for the approximated problem}

\begin{prop}
\label{prop:diff_psi}
Under the assumptions of Theorem~\ref{th:wp},
for every $\eps\in(0,1)$
the operator $\Psi_{\beta,\eps}':H\to H$ is G\^ateaux-differentiable along directions of $H$, 
Fr\'echet-differentiable along directions of $L^4(\Td)$, 
and $D\Psi'_{\beta,\eps}:H\to\cL(H,H)$
is given by
\[
  D\Psi'_{\beta,\eps}(z)[h]=\Psi''_{\beta, \eps}(z)h , 
  \quad z,h\in H.
\]
Furthermore, $D\Psi_{\beta,\eps}'\in C^{0,1}_b(L^4(\Td);\cL(L^4(\Td), H))\cap 
C^{0,1}_b(H;\cL(L^{2+\eta}(\Td), V^*))$ for every $\eta>0$,
and  it holds that 
\[
(D\Psi_{\beta,\eps}'(z)h,h)_H=(\Psi_{\beta,\eps}''(z)h,h)
\geq0\quad\forall\,z,h\in H,\quad \forall\,\eps\in(0,1).
\]
\end{prop}
\begin{proof}
    The result is a consequence of
    the definitions of G\^ateaux- and Fr\'echet-differentiability, 
    the dominated convergence theorem, 
    the facts that $\Psi_{\beta,\eps}'\in C^1(\erre)$ and
    $\Psi_{\beta,\eps}''\in C^{0,1}_b(\erre)$, and
    the monotonicity of $\Psi_{\beta,\eps}'$.
\end{proof}

\begin{prop}
\label{prop:diff_G}
Under the assumptions of Theorem~\ref{th:wp}, let 
\[
  \begin{cases}
  q>2, 
  \quad
  \s_q>\frac{2q}{q-2}
  \qquad&\text{if } \sigma\in\left[\frac d4,\frac12\right),\\
  q>\frac d{2\sigma},
  \quad
  \s_q>\frac{2q}{4\sigma q/d-2}
  \qquad&\text{if } \sigma\in\left(0,\frac d4\right),\\
 q=\s_q=\infty\qquad&\text{if } \sigma=0.
 \end{cases}
\]
Then, for every $\eps\in(0,1)$
the operator $G_{\alpha,\eps}:H\to\cL^2(H,H)$
is Fr\'echet-differentiable along directions of $L^q(\Td)$, 
and $DG_{\alpha,\eps}:H\to\cL(L^q(\Td),\cL^2(H,H))$ is given by 
\[
  (DG_{\alpha,\eps}(z)[h])[u]=m'_{\frac\alpha2, \eps}(z)h 
  (\operatorname{I}+A)^{-\delta}u, \quad z,u\in H, 
  \quad h\in L^q(\Td).
\]
Furthermore, $DG_{\alpha,\eps}\in C^{0,1}_b(L^{\s_q}(\Td);\cL(L^q(\Td),\cL^2(H,H)))$ 
and 
there exists a constant $C>0$,
depending only on $\alpha,\delta,\sigma,d,q$, such that 
\[
\sup_{z\in H}\norm{DG_{\alpha,\eps}(z)}_{\cL(L^q(\Td),\cL^2(H,H))}
\leq C\quad\forall\,\eps\in(0,1).
\]
\end{prop}
\begin{proof}
    First of all, note that the operator 
    $DG_{\alpha,\eps}$ is actually well-defined
    thanks to the fact that $(\operatorname{I}+A)^{-\delta}\in\cL^2(H,V_{2\sigma})$
    and $V_{2\sigma}\embed L^\infty(\Td)$ 
    if $\sigma\in(\frac d4, \frac12)$,
    $V_{2\sigma}\embed L^\s(\Td)$ for all $\s\in[1,+\infty)$ 
    if $\sigma=\frac d4$, and 
    $V_{2\sigma}\embed L^{\frac2{1-4\sigma/d}}(\Td)$ 
    if $\sigma\in[0,\frac d4)$. Indeed, 
    by definition of $q$ and the H\"older inequality,
    for every $z\in H$ and $h\in L^q(\Td)$
    one has 
    \begin{align*}
    \|DG_{\alpha, \varepsilon}(z)[h]\|^2_{\cL^2(H,H)}&=\sum_{k\in\mathbb N_+}\norm{
    m_{\frac\alpha2,\eps}'(z)h(\operatorname{I}+A)^{-\delta}e_k}_H^2\\
    &\lesssim \|m_{\frac\alpha2}'\|^{2}_{L^{\infty}(\erre)}
    \norm{h}_{L^q(\Td)}^2
    \norm{(\operatorname{I}+A)^{-\delta}}^{2}_{\cL^2(H,V_{2\sigma})},
    \end{align*}
    which also readily proves the last estimate.
    Furthermore, by assumption on $q$ and $\varsigma_q$,
    we have $\frac1q+\frac1{\s_q}<\frac12$
    if $\sigma\in[\frac d4, \frac12)$, 
    $\frac1q+\frac1{\s_q}<2\sigma/d$
    if $\sigma\in(0,\frac d4)$, and 
    $\s_q=\infty$ if $\sigma=0$.
    Hence, 
    given $z\in H$ and $h\in L^q(\mathbb{T}^d)$, 
    for every $t\in(-1,1)$ by the H\"older inequality 
    it holds that
    \begin{align*}
    &\left\Vert\frac{G_{\alpha, \varepsilon}(z+th)-
    G_{\alpha, \varepsilon}(z)}{t} - DG_{\alpha, \varepsilon}(z)[h] \right\Vert^2_{\cL^2(H,H)}
    \\
    &= \left\Vert\frac{m_{\frac{\alpha}{2}, \varepsilon}(z+th)
    (\operatorname{I}+A)^{-\delta}-
    m_{\frac{\alpha}{2}, \varepsilon}(z)
    (\operatorname{I}+A)^{-\delta}}{t} - m_{\frac{\alpha}{2}, \varepsilon}'(x)h
    (\operatorname{I}+A)^{-\delta} \right\Vert^2_{\cL^2(H,H)}
    \\
    &=\sum_{k \in \mathbb{N}_+}\left\Vert \int_0^1 h\left( m_{\frac{\alpha}{2}, \varepsilon}'(z+rth)
    (\operatorname{I}+A)^{-\delta}e_k -m_{\frac{\alpha}{2}, \varepsilon}'(z)
    (\operatorname{I}+A)^{-\delta}e_k\right)\, \d r \right\Vert_H^2
    \\
    &\le \sum_{k \in \mathbb{N}_+} \int_0^1 \left\Vert h
    (\operatorname{I}+A)^{-\delta}e_k
    \left( m_{\frac{\alpha}{2}, \varepsilon}'(z+rth)
     -m_{\frac{\alpha}{2}, \varepsilon}'(z)
    \right)\right\Vert_H^2\, \d r\\
    &\lesssim \|h\|_{L^q(\Td)}^2
    \|(\operatorname{I}+A)^{-\delta}\|_{\cL^2(H,V_{2\sigma})}^2
     \int_0^1 \left\Vert m_{\frac{\alpha}{2}, \varepsilon}'(z+rth)
     -m_{\frac{\alpha}{2}, \varepsilon}'(z)\right\Vert_{L^{\s_q}(\Td)}^2\, \d r,
\end{align*}
so that 
\begin{align*}
  &\sup_{\|h\|_{L^q(\Td)}\leq1}
    \left\Vert\frac{G_{\alpha, \varepsilon}(z+th)-
    G_{\alpha, \varepsilon}(z)}{t} - DG_{\alpha, \varepsilon}(z)[h] \right\Vert^2_{\cL^2(H,H)}\\
    &\lesssim\|(\operatorname{I}+A)^{-\delta}\|_{\cL^2(H,V_{2\sigma})}^2
     \int_0^1 \left\Vert m_{\frac{\alpha}{2}, \varepsilon}'(z+rth)
     -m_{\frac{\alpha}{2}, \varepsilon}'(z)\right\Vert_{L^{\s_q}(\Td)}^2\, \d r.
\end{align*}
Now, if $\sigma=0$, then $\s_q=+\infty$ and
$h\in L^\infty(\Td)$,
and the right-hand side converges to $0$ as $t \rightarrow 0$
since $m_{\frac\alpha2, \eps}'$ is Lipschitz-continuous.
If $\sigma>0$, then $\s_q<+\infty$, and
the right-hand side converges to $0$
by the dominated convergence theorem
since $m_{\frac\alpha2,\eps}'$ is continuous and bounded. 
This implies the Fr\'echet-differentiability 
(see \cite[Lem.~2.1]{Mar_Sca_Frechet}).
The $C^{0,1}$-regularity of $DG_{\alpha,\eps}$
follows by analogous computations and the Lipschitz-continuity 
of $m_{\frac\alpha2,\eps}'$.
\end{proof}

\begin{remark}
    As a consequence of Propositions~\ref{prop:diff_psi}--\ref{prop:diff_G}, 
    one has that 
    $\Psi_{\beta,\varepsilon}':H\to H$
    is always Fr\'echet-differentiable along directions of $V$,
    while $G_{\alpha,\eps}:H\to\cL^2(H,H)$ is
    Fr\'echet-differentiable along directions of $V\subseteq L^q(\mathbb{T}^d)$
    if $d=1$ for $\sigma\in[0,\frac12)$,
    if $d=2$ for $\sigma\in(0,\frac12)$,
    and if  $d=3$ for $\sigma\in(\frac14,\frac12)$,
    namelly under the assumptions of Theorem~\ref{th:wp}.
\end{remark}

We are ready to state the differentiability result on
the approximated solution map. In particular, we show that 
the derivative of $S_\eps$ along a direction $h\in H$ can be characterised as the solution to the equation 
\begin{equation}\label{eq:Y}
  \d Y_\eps^h + A Y_\eps^h\,\d t+
  \Psi_{\beta,\eps}''(X_\eps^x)Y_\eps^h\,\d t
  =DG_{\alpha,\eps}(X_\eps^x)[Y_\eps^h]\,\d W, \qquad
  Y_\eps^h(0)=h.
\end{equation}

\begin{prop}
\label{prop:diff_S}
Under the assumptions of Theorem~\ref{th:wp},
let $\eps\in(0,1)$ and $h\in H$.
Then, there exists a unique process $Y_\eps^h$ such that, 
for every $T>0$,
\[
Y_\varepsilon^h \in L^\ell_\cP
(\Omega; C([0,T];H) \cap L^2(0,T;V))\quad\forall\,\ell\geq1,
\]
and it holds, for all $t\in[0,T]$, $\P$-almost surely, that 
\begin{multline*}
    (Y_\eps^h(t),v)_H
    +\int_0^t(A^{\frac12}Y_\eps^h(s), A^{\frac12}v)_H\,\d s
    +\int_0^t(\Psi''_{\beta,\eps}(X_\eps^x(s))Y_\eps^h, v)_H\,\d s\\
    =(h,v)_H + \left(\int_0^t
    DG_{\alpha,\eps}(X_\eps^x(s))[Y_\eps^h(s)]\d W(s), v\right)_H
    \quad\forall\,v\in V.
\end{multline*}
Moreover, $Y_\eps^h$ satisfies \eqref{eq:Y} also in a mild sense, 
namely, for all $t\in[0,T]$, $\P$-almost surely,
\[
    Y_\eps^h(t)
    +\int_0^tS(t-s)
    \Psi''_{\beta,\eps}(X_\eps^x(s))Y_\eps^h\,\d s
    =S(t)h + \int_0^tS(t-s)
    DG_{\alpha,\eps}(X_\eps^x(s))[Y_\eps^h(s)]\d W(s).
\]
Furthermore, the approximated solution map 
$S_\eps$
is Fr\'echet-differentiable and
\[
  DS_\eps(x)[h]=Y_\eps^h \quad\forall\,x,h\in H.
\]
\end{prop}

\begin{proof}
The existence and uniqueness of a solution $Y_\eps^h$
to \eqref{eq:Y} follows from the classical theory of SPDEs.
Indeed, the drift operator 
$y\mapsto\Psi_{\beta,\eps}''(X_\eps^x)y$, $y\in H$,
is Lipschitz-continuous on $H$, uniformly on
$\Omega\times[0,T]$, since $\Psi_{\beta,\eps}''$ is bounded.
Moreover, arguing as in the proof of 
Proposition~\ref{prop:diff_G},
for the noise covariance operator it holds that
$\|DG_{\alpha,\eps}(X_\eps^x)[y]\|_{\cL^2(H,H)}
\leq C_q\|y
\|_{L^q(\Td)}$. 
Since the inclusion $V\embed L^q(\Td)$ is compact
(recall that here $\sigma>0$ if $d=2$
and $\sigma>\frac14$ if $d=3$),
we deduce that for every $\zeta>0$, there exists $C_\zeta>0$
such that $\|DG_{\alpha,\eps}(X_\eps^x)[y]\|_{\cL^2(H,H)}
\leq \zeta\|y\|_V+
C_\zeta\|y\|_H$ for every $y\in V$.
These considerations ensure indeed that \eqref{eq:Y}
admits a unique solution $Y_\eps^h$. The fact that $Y_\eps^h$
is also a mild solution is a classical result, since 
both variational and mild solutions can be obtained as 
limits of sequences of analytically-strong solutions.
It is also immediate to check that 
\begin{equation}
\label{aux_est_2}
  \norm{Y_\eps^h}_{L^\ell_\cP(\Omega; 
  C([0,T]; H)\cap L^2(0,T; V))}\lesssim_{\ell}\norm{h}_H
  \quad\forall\,h\in H.
\end{equation}
We only need to prove that $DS_\eps(x)[h]=Y_\eps^h$.
To this end, note that for $\eta\in(-1,1)$ we have 
\begin{align*}
  &\d\left(\frac{X_\eps^{x+\eta h}-X_\eps^x}{\eta}-Y_\eps^h\right)
  +A\left(\frac{X_\eps^{x+\eta h}-X_\eps^x}{\eta}-Y_\eps^h\right)\,\d t\\
  &\qquad+\left(\frac{\Psi_{\beta,\eps}'(X_\eps^{x+\eta h})
  -\Psi_{\beta,\eps}'(X_\eps^x)}{\eta}-\Psi_{\beta,\eps}''(X_\eps^x)Y_\eps^h\right)\,\d t\\
  &=\left(\frac{G_{\alpha,\eps}(X_\eps^{x+\eta h})
  -G_{\alpha,\eps}(X_\eps^x)}{\eta}-DG_{\alpha,\eps}(X_\eps^x)[Y_\eps^h]\right)\,\d W,
  \qquad
  \frac{X_\eps^{x+\eta h}(0)-X_\eps^x(0)}{\eta}-Y_\eps^h(0)=0.
\end{align*}
Standard computations based on the It\^o formula
and the Burkholder-Davis-Gundy and Young inequalities yield
for every $t\in[0,T]$,
\begin{align*}
  &\left\|\frac{X_\eps^{x+\eta h}-X_\eps^x}{\eta}-Y_\eps^h\right\|_{L^\ell(\Omega; C([0,t];H))}^\ell 
  +\norm{\frac{X_\eps^{x+\eta h}
  -X_\eps^x}{\eta}-Y_\eps^h
  }_{L^\ell(\Omega; L^2(0,t; V))}^\ell \\
  &\lesssim_\ell 
  \left\|\frac{X_\eps^{x+\eta h}-X_\eps^x}{\eta}-Y_\eps^h\right\|^\ell_{L^\ell(\Omega; L^2(0,t; H))}\\
  &\qquad+
  \left\|\frac{\Psi_{\beta,\eps}'(X_\eps^{x+\eta h})
  -\Psi_{\beta,\eps}'(X_\eps^x)}{\eta}-\Psi_{\beta,\eps}''(X_\eps^x)Y_\eps^h
  \right\|_{L^\ell(\Omega; L^2(0,t; H))}^\ell\\
  &\qquad+
  \norm{\frac{G_{\alpha,\eps}(X_\eps^{x+\eta h})
  -G_{\alpha,\eps}(X_\eps^x)}{\eta}
  -DG_{\alpha,\eps}(X_\eps^x)[Y_\eps^h]}_{
  L^\ell(\Omega; L^2(0,t; \cL^2(H,H)))}^\ell.
\end{align*}
Now, since $|\Psi_{\beta,\eps}''|\leq\frac{1}{\eps^\beta}$, we have 
\begin{align*}
&\left\|\frac{\Psi_{\beta,\eps}'(X_\eps^{x+\eta h})
  -\Psi_{\beta,\eps}'(X_\eps^x)}{\eta}-\Psi_{\beta,\eps}''(X_\eps^x)Y_\eps^h
  \right\|_{L^\ell(\Omega; L^2(0,t; H))}^\ell\\
&\lesssim_T
\left\|
\int_0^1
\left(\Psi_{\beta,\eps}''(X_\eps^{x}+
r(X_\eps^{x+\eta h}-X_\eps^x))
\frac{X_\eps^{x+\eta h}-X_\eps^x}{\eta}
-
\Psi_{\beta,\eps}''(X_\eps^x)
Y_\eps^h\right)\,\d r
  \right\|_{L^\ell(\Omega; L^2(0,t; H))}^\ell\\
&\lesssim_{\eps,\ell, \beta}
\left\|
\frac{X_\eps^{x+\eta h}-X_\eps^x}{\eta}
-
Y_\eps^h
  \right\|_{L^\ell(\Omega; L^2(0,t; H))}^\ell\\
  &\qquad+\int_0^1
\left\|
\left(
\Psi_{\beta,\eps}''(X_\eps^{x}+
r(X_\eps^{x+\eta h}-X_\eps^x))-
\Psi_{\beta,\eps}''(X_\eps^x)
\right)
Y_\eps^h
  \right\|_{L^\ell(\Omega; L^2(0,t; H))}^\ell\,\d r.
\end{align*}
Thanks 
to Lemma~\ref{lem_app_G}, 
analogously computations yield,
for every $\lambda>0$, that
\begin{align*}
    &\norm{\frac{G_{\alpha,\eps}(X_\eps^{x+\eta h})
  -G_{\alpha,\eps}(X_\eps^x)}{\eta}
  -DG_{\alpha,\eps}(X_\eps^x)[Y_\eps^h]}_{
  L^\ell(\Omega; L^2(0,t; \cL^2(H,H)))}^\ell\\
  &\lesssim_{T,\eps,\ell}
  \lambda\left\|
\frac{X_\eps^{x+\eta h}-X_\eps^x}{\eta}
-Y_\eps^h
  \right\|_{L^\ell(\Omega; L^2(0,t; V))}^\ell
  +C_\lambda\left\|
\frac{X_\eps^{x+\eta h}-X_\eps^x}{\eta}
-
Y_\eps^h
  \right\|_{L^\ell(\Omega; L^2(0,t; H))}^\ell\\
  &\qquad+
  \int_0^1
\left\|
\left(
DG_{\alpha,\eps}(X_\eps^{x}+
r(X_\eps^{x+\eta h}-X_\eps^x))-
DG_{\alpha,\eps}(X_\eps^x)
\right)
[Y_\eps^h]
  \right\|_{L^\ell(\Omega; L^2(0,t; \cL^2(H,H)))}^\ell\,\d r.
\end{align*}
By choosing $\lambda$ sufficiently small and 
by using the Gronwall lemma, we obtain then
\begin{align*}
  &\left\|\frac{X_\eps^{x+\eta h}-X_\eps^x}{\eta}-Y_\eps^h\right\|_{L^\ell(\Omega; C([0,T];H))}^\ell 
  +\norm{\frac{X_\eps^{x+\eta h}
  -X_\eps^x}{\eta}-Y_\eps^h
  }_{L^\ell(\Omega; L^2(0,T; V))}^\ell \\
  &\lesssim_{T,\eps,\ell, \beta} 
  \int_0^1
\left\|
\left(
\Psi_{\beta,\eps}''(X_\eps^{x}+
r(X_\eps^{x+\eta h}-X_\eps^x))-
\Psi_{\beta,\eps}''(X_\eps^x)
\right)
Y_\eps^h\
  \right\|_{L^\ell(\Omega; L^2(0,T; H))}^\ell\,\d r\\
  &\qquad+
  \int_0^1
\left\|
\left(
DG_{\alpha,\eps}(X_\eps^{x}+
r(X_\eps^{x+\eta h}-X_\eps^x))-
DG_{\alpha,\eps}(X_\eps^x)
\right)
[Y_\eps^h]
  \right\|_{L^\ell(\Omega; L^2(0,T; \cL^2(H,H)))}^\ell\,\d r.
\end{align*}
The first term on the right-hand side satisfies, 
thanks to the H\"older inequality, 
\begin{align*}
    &\int_0^1
\left\|
\left(
\Psi_{\beta,\eps}''(X_\eps^{x}+
r(X_\eps^{x+\eta h}-X_\eps^x))-
\Psi_{\beta,\eps}''(X_\eps^x)
\right)
Y_\eps^h
  \right\|_{L^\ell(\Omega; L^2(0,T; H))}^\ell\,\d r\\
  &\leq
  \norm{Y_\eps^h}_{L^{2\ell}(\Omega; L^4(0,T; L^4(\Td)))}^\ell\\
  &\qquad\times\int_0^1
\left\|
\Psi_{\beta,\eps}''(X_\eps^{x}+
r(X_\eps^{x+\eta h}-X_\eps^x))-
\Psi_{\beta,\eps}''(X_\eps^x)
  \right\|_{L^{2\ell}(\Omega; L^4(0,T; L^4(\Td)))}^\ell\,\d r.
\end{align*}
Since $V\embed L^\s(\Td)$ for every $\s\geq2$, 
by interpolation one has 
$C([0,T]; H)\cap L^2(0,T; V)\embed L^4(0,T; L^4(\Td))$, 
so that, thanks to \eqref{aux_est_2},
\begin{align*}
    &\int_0^1
\left\|
\left(
\Psi_{\beta,\eps}''(X_\eps^{x}+
r(X_\eps^{x+\eta h}-X_\eps^x))-
\Psi_{\beta,\eps}''(X_\eps^x)
\right)
Y_\eps^h
  \right\|_{L^\ell(\Omega; L^2(0,T; H))}^\ell\,\d r\\
  &\lesssim
  \norm{h}_H^\ell
  \int_0^1
\left\|
\Psi_{\beta,\eps}''(X_\eps^{x}+
r(X_\eps^{x+\eta h}-X_\eps^x))-
\Psi_{\beta,\eps}''(X_\eps^x)
  \right\|_{L^{2\ell}(\Omega; L^4(0,T; L^4(\Td)))}^\ell\,\d r,
\end{align*}
where the last term 
converges to $0$
as $\eta\to0$ by the dominated convergence theorem,
since $\Psi_{\beta,\eps}''$ is 
continuous and bounded.
For the second term on the right-hand side, 
by using the notation
of the proof of Proposition~\ref{prop:diff_G}, 
thanks to the H\"older inequality 
we have, for every $\rho\in(1,+\infty)$,
\begin{align*}
    &\int_0^1
\left\|
\left(
DG_{\alpha,\eps}(X_\eps^{x}+
r(X_\eps^{x+\eta h}-X_\eps^x))-
DG_{\alpha,\eps}(X_\eps^x)
\right)
[Y_\eps^h]
  \right\|_{L^\ell(\Omega; L^2(0,T; \cL^2(H,H)))}^\ell\,\d r\\
  &\leq
  \norm{(\operatorname{I}+A)^{-\delta}}_{\cL^2(H,V_{2\sigma})}^\ell\\
  &\quad\times\E\int_0^1
  \norm{Y_\eps^h}_{L^{2\rho}(0,T;L^q(\Td))}^\ell
  \norm{m'_{\frac\alpha2,\eps}
  (X_\eps^{x}+
r(X_\eps^{x+\eta h}-X_\eps^x))-
m'_{\frac\alpha2,\eps}
  (X_\eps^{x})
}_{L^{\frac{2\rho}{\rho-1}}(0,T;L^{\s_q}(\Td))}^\ell
  \,\d r\\
  &\lesssim\norm{Y_\eps^h}_{L^{2\ell}(\Omega;
  L^{2\rho}(0,T; L^q(\Td)))}^\ell
  \norm{(\operatorname{I}+A)^{-\delta}}_{
  \cL^2(H,V_{2\sigma})}^\ell\\
  &\quad\times\int_0^1\norm{m'_{\frac\alpha2,\eps}
  (X_\eps^{x}+
r(X_\eps^{x+\eta h}-X_\eps^x))-
m'_{\frac\alpha2,\eps}
  (X_\eps^{x})
}_{L^{2\ell}(\Omega; L^{\frac{2\rho}{\rho-1}}(0,T; L^{\s_q}(\Td))}^\ell\,\d r.
\end{align*}
By definition of $q$ and assumption on $\sigma$, 
one has $V\cembed L^{q}(\Td)$ compactly.
Hence, by interpolation there is $\bar\rho=\bar\rho(q)>1$ such that 
$C([0,T]; H)\cap L^2(0,T; V)\embed 
L^{2\bar\rho}(0,T;L^q(\Td))$. It follows that, thanks to \eqref{aux_est_2},
\begin{align*}
    &\int_0^1
\left\|
\left(
DG_{\alpha,\eps}(X_\eps^{x}+
r(X_\eps^{x+\eta h}-X_\eps^x))-
DG_{\alpha,\eps}(X_\eps^x)
\right)
[Y_\eps^h]
  \right\|_{L^\ell(\Omega; L^2(0,T; \cL^2(H,H)))}^\ell\,\d r\\
  &\lesssim
  \norm{h}_H^\ell
  \norm{(\operatorname{I}+A)^{-\delta}}_{
  \cL^2(H,V_{2\sigma})}^\ell\\
  &\qquad\times\int_0^1\norm{m'_{\frac\alpha2,\eps}
  (X_\eps^{x}+
r(X_\eps^{x+\eta h}-X_\eps^x))-
m'_{\frac\alpha2,\eps}
  (X_\eps^{x})
}_{L^{2\ell}(\Omega; L^{\frac{2\bar\rho}{\bar\rho-1}}(0,T; L^{\s_q}(\Td))}^\ell\,\d r,
\end{align*}
where the last term converges to $0$ as $\eta\to0$
by the dominated convergence theorem, since $m_{\frac\alpha2}'$
is continuous and bounded. 
Putting everything together yields then 
\[
  \lim_{\eta\to0}\sup_{\|h\|_H\leq 1}
  \left\|\frac{X_\eps^{x+\eta h}-X_\eps^x}{\eta}-Y_\eps^h\right\|_{L^\ell(\Omega; C([0,T];H)\cap 
L^2(0,T; V))}^\ell=0,
\]
and this concludes the proof.
\end{proof}

\subsection{Uniform estimates}
In this section we derive some apriori estimates
on the family $(Y_\eps^h)_\eps$
of solutions to problem \eqref{eq:Y}, uniform with respect to the parameter $\varepsilon \in (0,1)$.

\begin{prop}[First estimate]
\label{prop_stime_Y_1}
 Under the assumptions of Theorem~\ref{th:wp},
let $h\in H$.
 Then, for every $T>0$ and $\ell\geq2$,
 there exists a constant $C>0$,
 only depending on $T, \ell, \delta, \sigma,d$ such that,
 for every $\eps\in(0,1)$,
\[
\norm{\yeh}_{L^{\ell}_\cP(\Omega; C([0,T]; H) \cap  L^2(0,T; V))} + \norm{\Psi_{\frac{\beta}{2}, \varepsilon}''(\xe)\yeh}_{L^{\ell}_\cP(\Omega; L^2(0,T; H))} \le C\|h\|_H.
\]
\end{prop}

\begin{proof}
It\^o formula yields, 
for every $t \in [0,T]$, $\mathbb{P}$-almost surely, that
\begin{align*}
    &\frac 12\|\yeh(t)\|^2_H + 
    \int_0^t \| A^{\frac12}\yeh(s)\|^2_H \, \d s+ \int_0^t \int_{\Td}\Psi_{\beta, \varepsilon}''(\xe (s))|\yeh(s)|^2 \, \d s
    \\
    &= \frac 12 \|h\|^2_H 
   + \int_0^t \left(\yeh (s),
   DG_{\alpha,\eps}(\xex(s))[\yeh(s)] \, \d W(s)\right)_H \\
 &\qquad+ \frac 12 \int_0^t \norm{DG_{\alpha,\eps}(\xex(s))[\yeh(s)]}_{\cL^2(H,H)}^2\, \d s.
\end{align*}
Classical arguments based on the Burkholder-Davis-Gundy 
and Young inequalities imply that 
\begin{align*}
&\E\norm{\yeh}_{C([0,t]; H) \cap  L^2(0,t; V)}^\ell + \E\norm{\Psi_{\frac{\beta}{2}, \varepsilon}''(\xe)\yeh}_{L^2(0,t; H)}^\ell\\
&\lesssim_{T,\ell} \|h\|^{\ell}_H
+\E\norm{DG_{\alpha,\eps}(\xex)[\yeh]}^\ell_{L^2(0,t; \cL^2(H,H))}.
\end{align*}
Now,
we exploit again the fact that
$\|DG_{\alpha,\eps}(X_\eps^x)[Y_\eps^h]\|_{\cL^2(H,H)}
\leq C\|Y_\eps^h\|_{L^q(\Td)}$ for a constant $C>0$
independent of $\eps$ 
(see e.g.~the proof of 
Proposition~\ref{prop:diff_G}).
Since the inclusion $V\embed L^q(\Td)$ is compact,
we deduce that for every $\zeta>0$, there exists $C_\zeta>0$,
independent of $\eps$, such that 
\[
\E\norm{DG_{\alpha,\eps}(\xex)[\yeh]}_{L^2(0,t; \cL^2(H,H))}^\ell
\leq \zeta\E\norm{\yeh}_{L^2(0,t; V)}^\ell
+C_\zeta\E\norm{\yeh}_{L^2(0,t; H)}^\ell.
\]
The thesis follows then by choosing $\zeta$ small enough,
rearranging the terms, and using the Gronwall lemma.
\end{proof}

\begin{prop}[Second estimate]
\label{prop_stime_Y_2}
Under the assumptions of Theorem~\ref{th:wp},
let $h\in L^p(\Td)$, where
\begin{equation}
\label{eq:p}
p\in(2,+\infty) \quad\text{if } \sigma\in\left[\frac d4, \frac12\right), \qquad
p=\frac{2}{1-4\sigma/d} \quad\text{if } 
\sigma\in\left[0,\frac d4\right).
\end{equation}
Then, for every 
$T>0$ and $\ell\geq2$,
there exists a constant $C>0$,
only depending on $T,\ell,\delta,\sigma,d, p$, such that,
for every $\eps\in(0,1)$,
\[
\norm{\yeh}_{L^{\ell}_\cP(\Omega; L^\infty(0,T; L^p(\Td)))} + \norm{\Psi_{\frac{\beta}{p}, \varepsilon}''(\xe)\yeh}_{L^{\ell}_\cP(\Omega; L^p(0,T; L^p(\Td)))} \le C\|h\|_{L^p(\Td)}.
\]
\end{prop}

\begin{proof}
We use It\^o formula for the function 
$x\mapsto\|x\|^p_{L^p(\Td)}$, obtaining
for every $t \in [0,T]$, $\mathbb{P}$-almost surely,
\begin{multline}
\label{y_est_lp}
    \|\yeh(t)\|^p_{L^p(\Td)}+
    p(p-1) \int_0^t\int_{\Td} |\yeh(s)|^{p-2}|\nabla\yeh(s)|^2\, \d s + p \int_0^t\int_{\Td}\Psi_{\beta,\eps}''(\xe(s))|\yeh(s)|^p\, \d s
    \\
    =\|h\|^p_{L^p(\Td)} + 
    p \int_0^t \left( |\yeh(s)|^{p-2}\yeh(s), m_{\frac{\alpha}{2},\eps}'(\xe(s))\yeh(s)
    (\operatorname{I}+A)^{-\delta}\, \d W(s)\right)_H
    \\
    + \frac{p(p-1)}{2} \int_0^t  \sum_{k \in \mathbb N_+}\int_{\Td}|\yeh(s)|^p|m_{\frac{\alpha}{2},\eps}'(\xe(s))|^2|(\operatorname{I}+A)^{-\delta}e_k|^2 \, \d s.
\end{multline}
If $\sigma=0$, then $p=2$, and the thesis has been already proven in Proposition~\ref{prop_stime_Y_1}. Hence, 
consider the two cases $0 < \sigma \leq \frac d 4$ and $\frac d4 <\sigma \le \frac12$.

\smallskip
\noindent\boxed{\text{Case $\frac d4 <\sigma \le \frac12$.}}
Since $V_{2\sigma}\embed L^\infty(\Td)$, 
by Lemma~\ref{reg_A}, Lemma~\ref{lem_app}-(i),
and \eqref{mse1} we have
\begin{multline*}
\frac{p(p-1)}{2}
\int_0^t  \sum_{k \in \mathbb N_+}\int_{\Td}|\yeh(s)|^p|m_{\frac{\alpha}{2},\eps}'(\xe(s)|^2|(\operatorname{I}+A)^{-\delta}e_k|^2 \, \d s
\\
\lesssim
\|m_{\frac\alpha2}'\|_{L^\infty(\mathbb R)}^2
\sum_{k \in \mathbb N_+} \|(\operatorname{I}+A)^{-\delta}e_k\|^2_{V_{2\sigma}}
\int_0^t  \|\yeh(s)\|^p_{L^p(\Td)} \, \d s
\lesssim_{\alpha, p,\sigma,\delta,d} \int_0^t  \|\yeh(s)\|^p_{L^p(\Td)} \, \d s.
\end{multline*}
Moreover, 
by the Burkholder-Davis-Gundy, H\"older, and Young inequalities and arguing as above, for all $\ell\geq 2p$ 
we estimate the stochastic integral as 
\begin{align*}
&\E\sup_{r\in[0,t]}
\left| \int_0^r \left( |\yeh(s)|^{p-2}\yeh(s), m_{\frac{\alpha}{2}}'(\xe(s))\yeh(s)(\operatorname{I}+A)^{-\delta}\, \d W(s)\right)_H\right\|^{\frac\ell{p}}
\\
&\lesssim_{\ell} \mathbb{E} \left( \int_0^t\sum_{k\in\mathbb N_+}
\left( |\yeh(s)|^{p-2}\yeh(s), 
m_{\frac{\alpha}{2}}'(\xe(s))\yeh(s)
(\operatorname{I}+A)^{-\delta}e_k\right)_H^2
\,\d s\right)^{\frac{\ell}{2p}}
\\
&= \mathbb{E} \left( \int_0^t\sum_{k \in \mathbb N_+}
\left| \int_{\Td}|\yeh(s)|^p m_{\frac{\alpha}{2}}'(\xe(s))(\operatorname{I}+A)^{-\delta}e_k\right|^2\,\d s\right)^{\frac{\ell}{2p}}
\\
&\lesssim_{\delta,\sigma,d}
\|m_{\frac\alpha2}'\|_{L^\infty(\mathbb R)}^{\frac\ell{p}}
\norm{(\operatorname{I}+A)^{-\delta}}_{\cL^2(H,V_{2\sigma})}^{\frac\ell{p}}
\E\left( \int_0^t \|\yeh(s)\|^{2p}_{L^p(\Td)}\, \d s \right)^{\frac{\ell}{2p}}\\
&\lesssim_{\alpha,\delta,\sigma,d,p}
\E\int_0^t\|\yeh(s)\|^{\ell}_{L^p(\Td)}\, \d s.
\end{align*}
The thesis follows then by raising \eqref{y_est_lp}
to power $\frac\ell{p}$, taking supremum in time and expectations, 
and applying the Gronwall lemma.

\smallskip
\noindent\boxed{\text{Case $0 < \sigma \le \frac d4$.}} 
Recall that
$V_{2\sigma}\embed L^{p}(\Td)$, where 
$p=\frac{2}{1-4\sigma/d}>2$ if $0<\sigma<\frac d4$
and $p\in(2,\infty)$ is arbitrary if $\sigma=\frac d4$.
Since $V\embed L^{\frac{p^2}{p-2}}(\Td)\embed L^p(\Td)$,
by the H\"older inequality, interpolation, and the Young inequality
we have for every $\zeta>0$ that
\begin{align*}
&\frac{p(p-1)}{2} \int_0^t  \sum_{k \in \mathbb N_+}\int_{\Td}|\yeh(s)|^p|m_{\frac{\alpha}{2}}'(\xe(s)|^2|(\operatorname{I}+A)^{-\delta}e_k|^2 \, \d s
\\
&\le \frac{p(p-1)}{2} \|m_{\frac{\alpha}{2}}'\|_{L^\infty(\erre)}^2 \sum_{k \in \mathbb N_+} \|(\operatorname{I}+A)^{-\delta}e_k\|^2_{L^{p}(\Td)}\int_0^t \|\yeh(s)\|^p_{L^{\frac{p^2}{p-2}}(\Td)} \, \d s\\
&\lesssim_{\alpha,\delta,\sigma,d,p}
\int_0^t  \|\yeh(s)\|^p_{L^{\frac{p^2}{p-2}}(\Td)} \, \d s
\leq \int_0^t  \|\yeh(s)\|^2_V 
\|\yeh(s)\|^{p-2}_{L^p(\Td)} \, \d s\\
&\leq
\zeta\sup_{s\in[0,t]}\|\yeh(s)\|^{p}_{L^p(\Td)} \, \d s
+C_{\zeta,p}\|\yeh\|_{L^2(0,T; V)}^p.
\end{align*}
Arguing as above, 
since $V\embed L^{\frac{p^2}{p-1}}(\Td)\embed L^p(\Td)$,
for all $\ell \ge 2p$ we have
\begin{align*}
&\E\sup_{r\in[0,t]}
\left| \int_0^r \left( |\yeh(s)|^{p-2}\yeh(s), m_{\frac{\alpha}{2},\eps}'(\xe(s))\yeh(s)(\operatorname{I}+A)^{-\delta}\, \d W(s)\right)_H\right|^{\frac\ell{p}}
\\
&\lesssim_{\ell,p} \mathbb{E} 
\left( \int_0^t  \sum_{k \in \mathbb N_+}\left|
\int_{\Td}|\yeh(s)|^p m_{\frac{\alpha}{2},\eps}'(\xe(s))
(\operatorname{I}+A)^{-\delta}e_k\right|^2\,\d s
\right)^{\frac{\ell}{2p}}
\\
&\lesssim_{\alpha,\delta,\sigma,d,p}
 \mathbb{E} \left( \int_0^t 
 \|\yeh(s)\|^{2p}_{L^{\frac{p^2}{p-1}}(\Td)}\, \d s \right)^{\frac{\ell}{2p}}\leq
 \mathbb{E} \left( \int_0^t 
 \|\yeh(s)\|_V^{2}
 \|\yeh(s)\|_{L^p(\Td)}^{2p-2}\,\d s
 \right)^{\frac{\ell}{2p}}\\
 &\leq\E\left[\sup_{s\in[0,t]}
 \|\yeh(s)\|^{\frac{\ell (p-1)}{p}}_{L^p(\Td)}
 \|\yeh\|_{L^2(0,T; V)}^{\frac\ell p}\right]
\le \zeta  \mathbb{E} 
\sup_{s \in [0,T]}\|\yeh(s)\|^{p}_{L^p(\Td)} + C_{\zeta,p} 
\mathbb{E}  \|\yeh\|^p_{L^2(0,T; V)}.
\end{align*}
The thesis follows then by raising \eqref{y_est_lp}
to power $\frac\ell{p}$, taking supremum in time and expectations, 
and rearranging the terms by choosing $\zeta$ sufficiently small.
\end{proof}

\begin{prop}[Third estimate]
\label{prop_stime_Y_3}
Under the assumptions of Theorem~\ref{th:wp},
let $\sigma>0$, 
let $p$ satisfy \eqref{eq:p},
and assume that:
\begin{itemize}
\item if $\sigma\in[\frac d4, \frac12)$, then 
$\gamma>\beta+2$ and $p\geq2\frac{\gamma-2}{\gamma-\beta-2}$;
\item if $\sigma\in(0,\frac d4)$, then
$\alpha\geq\beta\frac{1-4\sigma/d}{4\sigma/d}+2$ and $\gamma\geq\frac\beta{4\sigma/d}+2$.
\end{itemize}
Then, for every $\ell\in(2,+\infty)$,
$T>0$, and 
$\vartheta \in (0,\frac12)\cap
\left(0, \sigma+\frac 12 -\frac d4\right)$,
there exists a constant $C>0$,
only depending on $T,\ell,\vartheta,\alpha,\beta,\delta,\sigma,d, p$, such that,
for every $h\in L^p(\Td)\cap V_{1-2\vartheta}$ it holds
\[
\yeh \in L^\ell_\cP(\Omega; C([0,T];V_{1-2\vartheta})) \quad\forall\,\eps\in(0,1),
\]
and, for every $\eps\in(0,1)$,
\begin{equation}
    \label{est_mr}
\|\yeh\|_{L^\ell_\cP(\Omega; C([0,T];V_{1-2\vartheta}))} 
\le C\left[\left( 1+ 
\|\Psi_{\gamma}(x)\|_{L^1(\Td)}^{\frac12-\frac1p}\right)
\|h\|_{L^p(\Td)}+\|h\|_{V_{1-2\vartheta}}\right].
\end{equation}
\end{prop}
\begin{proof}
The proof of \eqref{est_mr}
is based on maximal regularity results 
on the mild formulation of equation \eqref{eq:Y}.
To this end, we show 
uniform estimates on the deterministic and stochastic 
convolution terms. Note that it is not restrictive
to show the statement only for $\ell>\frac1\vartheta$.\\
\noindent\underline{\sc Step 1.}
First, we show that
\begin{equation}\label{aux:4}
\|\Psi_{\beta,\varepsilon}''(\xe)\yeh\|_{L^\ell_\cP(\Omega; L^2(0,T;H))} \le C\left( 1+
\|\Psi_{\gamma}(x)\|_{L^1(\Td)}^{\frac12-\frac1p}\right)
\|h\|_{L^p(\Td)} .
\end{equation}
To this end, thanks to the H\"older inequality, we have
\begin{align*}
&\|\Psi_{\beta,\varepsilon}''(\xe)\yeh\|_{L^\ell(\Omega; L^2(0,T;H))}
=\norm{\Psi_{\frac{\beta}{p},\varepsilon}''(\xe)\ye\cdot
\Psi_{\frac{\beta(p-1)}{p},\varepsilon}''(\xe)}_{L^\ell(\Omega; L^2(0,T;H))}\\
&\leq 
\norm{\Psi_{\frac{\beta}{p},\varepsilon}''(\xe)\ye}_{
L^{2\ell}(\Omega; L^p(0,T; L^p(\Td)))}
\norm{\Psi_{\frac{\beta(p-1)}{p},\varepsilon}''(\xe)}_{L^{2\ell}(\Omega; L^{\frac{2p}{p-2}}(0,T; L^{\frac{2p}{p-2}}(\Td)))}.
\end{align*}
On the one hand, 
from Proposition~\ref{prop_stime_Y_2} we infer that
\[
\norm{\Psi_{\frac{\beta}{p},\varepsilon}''(\xe)\ye}_{
L^{2\ell}(\Omega; L^p(0,T; L^p(\Td)))}\lesssim
\|h\|_{L^p(\Td)}.
\]
On the other hand, we note that 
\[
\norm{\Psi_{\frac{\beta(p-1)}{p},\varepsilon}''(\xe)}_{L^{2\ell}(\Omega; L^{\frac{2p}{p-2}}(0,T); L^{\frac{2p}{p-2}}(\Td)))}
=\norm{\Psi_{\frac{2\beta(p-1)}{p-2},\varepsilon}''(\xe)}_{
L^{\frac{\ell(p-2)}p}(\Omega; L^1(0,T; L^1(\Td)))}^{\frac{p-2}{2p}},
\]
where by assumption on $\gamma$
and by \eqref{eq:p} one always has that
$\frac{2\beta(p-1)}{p-2} \leq \gamma+\beta-2$:
hence, 
from Lemma\ref{lem_app}-(ix) and Proposition~\ref{prop3Xeps} we obtain
\begin{align*}
\norm{\Psi_{\frac{\beta(p-1)}{p},\varepsilon}''(\xe)}_{L^{2\ell}(\Omega; L^{\frac{2p}{p-2}}(0,T); L^{\frac{2p}{p-2}}(\Td)))}
&\lesssim
1+
\|\Psi_\gamma(x)\|_{L^1(\Td)}^{\frac{p-2}{2p}},
\end{align*}
and estimate \eqref{aux:4} follows.

\noindent\underline{\sc Step 2.}
Secondly, we show that there exists a family 
$(C_\lambda)_\lambda$
independent of $\eps$ such that, for every $\lambda>0$,
it holds that 
\begin{equation}
    \label{aux:5}
  \norm{DG_{\alpha,\eps}(\xex)[\yeh]}_{L^\ell_\cP(\Omega;L^\ell(0,T; \cL^2(H,H)))}
  \leq \lambda
  \|\yeh\|_{L^\ell_\cP(\Omega; C([0,T];V_{1-2\vartheta}))} 
  +C_\lambda \|h\|_H.
\end{equation}
To this end, we note that in the notation of Proposition~\ref{prop:diff_G},
one has that 
\[
  \norm{DG_{\alpha,\eps}(\xex)[\yeh]}_{\cL^2(H,H)}
  \lesssim_q\norm{\yeh}_{L^q(\Td)}
  \quad\text{a.e.~in } \Omega\times(0,T),
\]
and \eqref{aux:5} follows from the fact that 
$q$ in Proposition~\ref{prop:diff_G} can be chosen 
such that 
the inclusion $V_{1-2\vartheta}\embed L^q(\Td)$ is compact
and from Proposition~\ref{prop_stime_Y_1}.
Indeed, if
$\sigma\in[\frac d4, \frac12)$ then $q>2$ is arbitrary, 
and $V_{1-2\vartheta}\embed L^q(\Td)$ is compact provided that 
$q$ is close enough to $2$ since $1-2\vartheta>0$.
Otherwise, if $\sigma\in(0,\frac d4)$
by assumption on $\vartheta$ one has 
$1-2\vartheta-\frac d2>0-2\sigma$, so that 
by the Sobolev embeddings it holds 
$V_{1-2\vartheta}\embed L^{d/{2\sigma}}(\Td)$ compactly:
since in this case
$q>\frac d{2\sigma}$
is arbitrary, $V_{1-2\vartheta}\embed L^q(\Td)$ is compact 
provided that 
$q$ is close enough to $\frac d{2\sigma}$.

\noindent\underline{\sc Step 3.}
We are now ready to 
exploit maximal regularity arguments on the mild
formulation of equation \eqref{eq:Y}
(recall Proposition~\ref{prop:diff_S}), namely
\begin{multline*}
\yeh = S(\cdot)h  
+ \int_0^\cdot S(\cdot-s)
\Psi''_{\beta, \varepsilon}(X_\eps^x(s))\yeh(s)  \, {\rm d}s
+ \int_0^\cdot S(\cdot-s) DG_{\alpha,\eps}(\xe(s))[\yeh(s)]
\,\d W (s).
\end{multline*}
By taking the $L^\ell(\Omega; C([0,T];V_{1-2\vartheta}))$-norm on both sides we get 
\begin{align*}
 \|\yeh \|_{L^\ell(\Omega; C([0,T];V_{1-2\vartheta}))}   
 &\le  \|S(\cdot)h\|_{L^\ell(\Omega; C([0,T];V_{1-2\vartheta}))}    
 \\
 &+  \norm{\int_0^\cdot S(\cdot-s)\Psi''_{\beta, \varepsilon}(X_\eps^x(s))\yeh(s)  \, {\rm d}s }_{L^\ell(\Omega; 
 C([0,T];V_{1-2\vartheta}))}    
 \\
 &+  \norm{ \int_0^\cdot S(\cdot-s) DG_{\alpha,\eps}(\xe(s))[\yeh(s)]\,\d W (s) }_{L^\ell(\Omega; C([0,T];V_{1-2\vartheta}))}.
\end{align*}
Let us estimate the three terms on the right-hand side separately.
First of all, since the restriction of $S$ to $V_{1-2\vartheta}$ is a strongly continuous semigroup of contractions on 
$V_{1-2\vartheta}$, it holds that
\[
\|S(\cdot)h\|_{L^\ell(\Omega; C([0,T];V_{1-2\vartheta}))} \le \|h\|_{V_{1-2\vartheta}}.
\]
Secondly, from the maximal regularity results
\cite[Thm.~3.3, Ex.~3.2]{VNVW} 
and \cite[Prop.~A.24]{dapratozab}, 
as well as the estimate \eqref{aux:4}, we infer that
\begin{align*}
&\norm{\int_0^\cdot S(\cdot-s)\Psi''_{\beta, \varepsilon}(X_\eps^x(s))\yeh(s)  \, {\rm d}s }_{L^\ell(\Omega; 
 C([0,T];V_{1-2\vartheta}))}\\
 &\qquad\lesssim
    \left\Vert\int_0^\cdot S(\cdot-s)\Psi''_{\beta, \varepsilon}(X(s))\yeh(s)  \, {\rm d}s \right\Vert_{L^\ell(\Omega; L^2(0,T;V_2) \cap C([0,T];V))}   
    \\
    &\qquad\lesssim \|\Psi_{\beta,\varepsilon}''(\xe)\yeh\|_{L^\ell(\Omega; L^2(0,T;H))}\lesssim
    \left( 1+
\|\Psi_{\gamma}(x)\|_{L^1(\Td)}^{\frac12-\frac1p}\right)
\|h\|_{L^p(\Td)}.
\end{align*}
Lastly, by the stochastic maximal regularity results
\cite[Thm~3.5]{VNVW} and by
\eqref{aux:5},
we infer that
\begin{align*}
    &\norm{ \int_0^\cdot S(\cdot-s) DG_{\alpha,\eps}(\xe(s))[\yeh(s)]\,\d W (s) }_{L^\ell(\Omega; C([0,T];V_{1-2\vartheta}))}\\
    &\qquad\lesssim
    \norm{ \int_0^\cdot S(\cdot-s) DG_{\alpha,\eps}(\xe(s))[\yeh(s)]\,\d W (s) }_{L^\ell(\Omega;
    C^{\vartheta-\frac 1 \ell}([0,T];V_{1-2\vartheta}))}\\
    &\qquad\lesssim
    \norm{DG_{\alpha,\eps}(\xe)[\yeh]}_{
    L^\ell(\Omega; L^\ell(0,T;\cL^2(H,H)))}\\
    &\qquad\leq
    \lambda
  \|\yeh\|_{L^\ell_\cP(\Omega; C([0,T];V_{1-2\vartheta}))} 
  +C_\lambda \|h\|_H
\end{align*}
so that the estimate \eqref{est_mr} follows
by choosing $\lambda$ small enough.
\end{proof}

\begin{prop}[Fourth estimate]
\label{prop_stime_Y_4}
Under the assumptions of Theorem~\ref{th:wp},
let $d=1$, $\sigma>0$, 
$x\in V_{2\xi}$ for some
$\xi\in(\sigma, 1)\cap(\frac 14, \frac12+\sigma)$, let $\alpha\geq4$, let $p$ satisfy \eqref{eq:p},
and assume that:
\begin{itemize}
\item if $\sigma\in[\frac 14, \frac12)$, then 
$\gamma>2\beta+4$ and $p\geq2\frac{\gamma-2}{\gamma-\beta-2}$;
\item if $\sigma\in(0,\frac 14)$, then
$\alpha\geq2(\beta+1)\frac{1-4\sigma}{4\sigma}+2$
and $\gamma\geq
\frac{2(\beta +1)}{4\sigma}+2$.
\end{itemize}
Then,
for every $\ell\in(2,+\infty)$,
$T>0$, and
$\vartheta\in(0,\sigma]$,
there exists a constant $C>0$,
only depending on $T,\ell,\vartheta,\alpha,\beta,\delta,\sigma,d, p$, such that,
for every $h\in V_{2\sigma+1-2\vartheta}$ it holds
\[
\yeh \in L^\ell_\cP(\Omega; C([0,T];V_{2\sigma+1-2\vartheta})
\cap L^2(0,T; V_{2\sigma+1})) \quad\forall\,\eps\in(0,1),
\]
and, for every $\eps\in(0,1)$,
\begin{multline}
    \label{est2_mr}
\|\yeh\|_{L^\ell_\cP(\Omega; C([0,T];V_{2\sigma+1-2\vartheta})\cap
L^2(0,T; V_{2\sigma+1}))} 
\le C\left( 1+ \|x\|_{V_{2\xi}}^{2-\frac1p}+
\|\Psi_{\gamma}(x)\|_{L^1(\Td)}^{2-\frac1p}\right)
\|h\|_{V_{2\sigma+1-2\vartheta}}.
\end{multline}
\end{prop}
\begin{proof}
The proof of \eqref{est2_mr} follows as 
a consequence of \eqref{est_mr} by employing a 
a bootstrap argument. To this end, 
we show that \eqref{est_mr} allows to refine
the estimates on the deterministic and stochastic 
convolutions performed in the proof of Proposition~\ref{prop_stime_Y_3}.
 Again, we only focus on values of
$\ell>\frac1\vartheta$, as this is not restrictive.
\\
\noindent\underline{\sc Step 1.}
First,
we show that estimate \eqref{aux:4} can be refined 
by using \eqref{est_mr} as
\begin{equation}
\|\Psi_{\beta,\varepsilon}''(\xe)\yeh\|_{L^\ell_\cP(\Omega; L^2(0,T;V_{2\sigma}))}
\le
C\left(1+
  \|x\|_{V_{2\sigma}}\right)
  \left(1+
  \|\Psi_{\gamma}(x)\|_{L^1(\Td)}^{1-\frac1p}\right)
\|h\|_{V_{2\sigma+1-2\vartheta}}.
\label{aux:7}
\end{equation}
Indeed, by Lemma~\ref{multiplication} we have
for every $\zeta,\rho\in[\sigma,\frac12)$
with $\zeta+\rho>\sigma+\frac d4$
that
\[
  \norm{\Psi_{\beta,\eps}''(\xex)\yeh}_{V_{2\sigma}}
  \lesssim\norm{\Psi_{\beta,\eps}''(\xex)}_{V_{2\zeta}}
  \norm{\yeh}_{V_{2\rho}}.
\]
By choosing $\rho$ sufficiently close to $\frac12$, 
one has $V_{2\rho}\embed L^p(\Td)$ and 
by Proposition~\ref{prop_stime_Y_3} it holds, 
for every $l\geq2$,
\[
  \norm{\yeh}_{L^{l}(\Omega; L^\infty(0,T; V_{2\rho}))}
  \lesssim_l \left(1+
  \|\Psi_{\gamma}(x)\|_{L^1(\Td)}^{\frac12-\frac1p}\right)
\|h\|_{V_{2\rho}}.
\]
Moreover, by interpolation and the H\"older inequality we have, 
for every $\kappa> 2$,
\begin{align*}
    \|\Psi''_{\beta,\varepsilon}(\xe)\|_{V_{2\zeta}} 
    &\lesssim
    \|\Psi''_{\beta,\varepsilon}(\xe)\|_H
    +
    \|\Psi''_{\beta,\varepsilon}(\xe)\|_H^{1-2\zeta} \|\nabla\Psi''_{\beta,\varepsilon}(\xe)\|_H^{2\zeta} 
    \\
    &= \|\Psi''_{\beta,\varepsilon}(\xe)\|_H
    + \|\Psi''_{\beta,\varepsilon}(\xe)\|_H^{1-2\zeta} \|\Psi'''_{\beta,\varepsilon}(\xe)\nabla X_\eps^x\|_H^{2\zeta} 
    \\
    & \le\|\Psi''_{\beta,\varepsilon}(\xe)\|_H
    + \|\Psi''_{\beta,\varepsilon}(\xe)\|_H^{1-2\zeta} \|\Psi'''_{\beta,\varepsilon}(\xe)\|^{2\zeta}_{L^\kappa(\Td)}
    \|\xe\|^{2\zeta}_{W^{1,\frac{2\kappa}{\kappa-2}}(\Td)}.
\end{align*}
Thanks to Lemma~\ref{lem_app}-(v) 
we have
\begin{align*}
\|\Psi'''_{\beta,\varepsilon}(\xe)\|^{2\zeta}_{L^\kappa(\Td)} 
&\lesssim \|\Psi''_{\beta +1,\varepsilon}(\xe)\|_{
L^\kappa(\Td)}^{2 \zeta} 
= \|\Psi''_{\kappa(\beta +1),\varepsilon}(\xe)\|_{L^1(\Td)}^{\frac{2 \zeta}{\kappa}}\\ 
&\lesssim 
\|\Psi_{\kappa\left(\beta +1\right)+2,\varepsilon}(\xe)\|_{L^1(\Td)}^{\frac{2 \zeta}{\kappa}},\\
    \|\Psi''_{\beta,\varepsilon}(\xe)\|_H^{1-2\zeta}
    &=  \|\Psi''_{2\beta,\varepsilon}(\xe)\|_{L^1(\Td)}^{\frac12-\zeta}
    \lesssim \|\Psi_{2\beta+2,\varepsilon}(\xe)\|_{L^1(\Td)}^{\frac12-\zeta},\\
    \|\Psi''_{\beta,\varepsilon}(\xe)\|_H
    &\lesssim \|\Psi_{2\beta+2,\varepsilon}(\xe)\|^{\frac12}_{L^1(\Td)},
\end{align*}
so that by the continuous embedding 
$V_{1+\frac{d}{\kappa}}\embed 
W^{1, \frac{2\kappa}{\kappa-2}}(\Td)$ we infer then 
\begin{align*}
  &\norm{\Psi_{\beta,\eps}''(\xex)}_{
  L^2(0,T;V_{2\zeta})}
  \lesssim
  \|\Psi_{2\beta+2,\varepsilon}(\xe)\|^{\frac12}_{
  L^\infty(0,T;L^1(\Td))}\\
  &\qquad+\|\Psi_{2\beta+2,\varepsilon}(\xe)\|_{L^\infty(0,T;L^1(\Td))}^{\frac12-\zeta}
  \|\Psi_{\kappa\left(\beta +1\right)+2,\varepsilon}(\xe)\|_{L^\infty(0,T;L^1(\Td))}^{\frac{2 \zeta}{\kappa}}
  \|\xe\|_{L^2(0,T;V_{1+\frac d\kappa})}^{2\zeta}.
\end{align*}
The last term on the right-hand side is
uniformly bounded in $L^l(\Omega)$ for every $l\geq2$
by Proposition~\ref{prop5Xeps} since $x\in V_{2\sigma}$, 
provided that
$1+\frac d\kappa\leq1+2\sigma$, i.e.~for
$\kappa\geq\frac d{2\sigma}$.
The terms involving $\Psi$ on the right-hand side are
uniformly bounded in $L^l(\Omega)$ for every $l\geq2$
by Proposition~\ref{prop3Xeps} 
provided that 
$2\beta+2\leq\gamma$ and
$\kappa(\beta +1)+2\leq\gamma$. 
The first condition is true by assumption, 
while the second one is satisfied 
for $\kappa=\frac{\gamma-2}{\beta+1}$
since $\frac{\gamma-2}{\beta+1}>2$ and $
\frac{\gamma-2}{\beta+1}\geq\frac d{2\sigma}$
by assumption.
By putting everything together, the H\"older inequality 
and the choice $l=2\ell$ yield
\begin{align*}
  &\norm{\Psi''_{\beta,\varepsilon}(\xe)\yeh}_{
  L^\ell(\Omega;L^2(0,T;V_{2\sigma}))}
  \lesssim
  \norm{\Psi''_{\beta,\varepsilon}(\xe)}_{
  L^{2\ell}(\Omega;L^2(0,T;V_{2\zeta}))}
  \norm{\yeh}_{
  L^{2\ell}(\Omega;L^\infty(0,T;V_{2\rho}))}\\
  &\lesssim
  \left(1+
  \|\Psi_{\gamma}(x)\|_{L^1(\Td)}^{\frac12}\right)
  \left(1+
  \|\Psi_{\gamma}(x)\|_{L^1(\Td)}^{\frac12-\frac1p}\right)
\|h\|_{V_{2\rho}}\\
&\qquad+
\left(1+
  \|\Psi_{\gamma}(x)\|_{L^1(\Td)}^{\frac12-\zeta+\frac{2\zeta}\kappa}\right)
  \left(1+
  \|x\|_{V_{2\sigma}}^{2\zeta}\right)
  \left(1+
  \|\Psi_{\gamma}(x)\|_{L^1(\Td)}^{\frac12-\frac1p}\right)
\|h\|_{V_{2\rho}},
\end{align*}
and \eqref{aux:7} follows from the fact that 
$\kappa>2$ and $\zeta,\rho<\frac12$.

\noindent\underline{\sc Step 2.}
Analogously, we show here that \eqref{est_mr}
allows to refine also \eqref{aux:5}, in the sense that 
there exists a family $(C_\lambda)_\lambda$ independent of $\eps$ such that, for every $\lambda>0$
\begin{multline}
  \norm{DG_{\alpha,\eps}(\xex)[\yeh]}_{L^\ell_\cP(\Omega;L^\ell(0,T; \cL^2(H,V_{2\sigma})))}
  \leq 
  \lambda
\norm{\yeh}_{L^{2\ell}(\Omega; C([0,T]; V_{1+2\sigma-2\vartheta}))}\\
+C_\lambda 
\left(1+\|\Psi_\gamma(x)\|_{L^1(\Td)}
  +\|x\|_{V_{2\xi}}\right)
  \left(1+
  \|\Psi_{\gamma}(x)\|_{L^1(\Td)}^{\frac12-\frac1p}\right)
\|h\|_{V_{2\sigma+1-2\vartheta}}.
  \label{aux:8}
\end{multline}
Indeed, since $\xi>\sigma$ and $\xi>\frac d4$,
for every $\zeta\in[\sigma, \xi)\cap (\frac d4, \xi)$
by Lemma~\ref{multiplication}
we have 
\begin{align*}
  \norm{DG_{\alpha,\eps}(\xex)[\yeh]}_{\cL^2(H,V_{2\sigma})}^2
  &=\sum_{k\in\mathbb N_+}\norm{m_{\frac\alpha2, \eps}'(\xex)\yeh(\operatorname{I}+A)^{-\delta}e_k}_{V_{2\sigma}}^{2}\\
  &\lesssim\norm{(\operatorname{I}+A)^{-\delta}}^2_{\cL^2(H,V_{2\sigma})}
  \norm{m_{\frac\alpha2, \eps}'(\xex)}_{V_{2\zeta}}^2\norm{\yeh}_{V_{2\zeta}}^2.
\end{align*}
Since $d=1$, one can choose $\zeta\in[\sigma, \xi)\cap (\frac14, \frac12)$: 
hence, the fact that $(m_{\frac\alpha2, \eps}')_\eps$ and
$(m_{\frac\alpha2, \eps}'')_\eps$ are uniformly bounded in 
$L^\infty(\mathbb R)$ since $\alpha\geq4$,
the H\"older inequality,
and 
Propositions~\ref{prop5Xeps} and \ref{prop_stime_Y_3} yield
\begin{multline*}
  \norm{DG_{\alpha,\eps}(\xex)[\yeh]}_{L^\ell_\cP(\Omega;
  L^\ell(0,T; \cL^2(H,V_{2\sigma})))}\\
  \lesssim
  \norm{\xex}_{L^{2\ell}(\Omega; L^\infty(0,T; V_{2\zeta}))}
  \norm{\yeh}_{L^{2\ell}(\Omega; L^\infty(0,T; V_{2\zeta}))}\\
  \lesssim \left(1+\|\Psi_\gamma(x)\|_{L^1(\Td)}
  +\|x\|_{V_{2\zeta}}\right)
  \left(1+
  \|\Psi_{\gamma}(x)\|_{L^1(\Td)}^{\frac12-\frac1p}\right)
\|h\|_{V_{2\zeta}}
\end{multline*}
and \eqref{aux:8} follows since $\zeta<\xi$.

\noindent\underline{\sc Step 3.}
We are now ready to conclude the bootstrap argument by
exploiting again maximal regularity arguments on the mild
formulation of equation \eqref{eq:Y}.
To this end, the proof is exactly analogous to the 
one performed in Proposition~\ref{prop_stime_Y_3}, by 
relying on the improved estimates \eqref{aux:7}--\eqref{aux:8}
instead of \eqref{aux:4}--\eqref{aux:5}.
In particular, since the restriction of the $S$ to $V_{2\sigma+1-2\vartheta}$ is a strongly continuous semigroup of contractions on 
$V_{2\sigma+1-2\vartheta}$, it holds that
\[
\|S(\cdot)h\|_{L^\ell(\Omega; C([0,T];V_{2\sigma+1-2\vartheta})\cap L^2(0,T; V_{2\sigma+1}))} \lesssim \|h\|_{V_{2\sigma+1-2\vartheta}}.
\]
Moreover, 
from the maximal regularity results
\cite[Thm.~3.3, Ex.~3.2]{VNVW} 
and \cite[Prop.~A.24]{dapratozab}, we infer that
\begin{align*}
&\norm{\int_0^\cdot S(\cdot-s)\Psi''_{\beta, \varepsilon}(X_\eps^x(s))\yeh(s)  \, {\rm d}s }_{L^\ell(\Omega; 
 C([0,T];V_{2\sigma+1-2\vartheta})\cap L^2(0,T; V_{2\sigma+1}))}\\
 &\qquad\lesssim
    \left\Vert\int_0^\cdot S(\cdot-s)\Psi''_{\beta, \varepsilon}(X(s))\yeh(s)  \, {\rm d}s \right\Vert_{L^\ell(\Omega; L^2(0,T;V_{2\sigma+2}) 
    \cap C([0,T];V_{2\sigma+1}))}\\
    &\qquad\lesssim \|\Psi_{\beta,\varepsilon}''(\xe)\yeh\|_{L^\ell(\Omega; L^2(0,T;V_{2\sigma}))}.
\end{align*}
Lastly, by the stochastic maximal regularity results
\cite[Thm~3.5]{VNVW}
we infer that
\begin{align*}
    &\norm{ \int_0^\cdot S(\cdot-s) DG_{\alpha,\eps}(\xe(s))[\yeh(s)]\,\d W (s) }_{L^\ell(\Omega; C([0,T];V_{2\sigma+1-2\vartheta})\cap L^2(0,T; V_{2\sigma+1}))}\\
    &\qquad\lesssim
    \norm{ \int_0^\cdot S(\cdot-s) DG_{\alpha,\eps}(\xe(s))[\yeh(s)]\,\d W (s) }_{L^\ell(\Omega;
    C^{\vartheta-\frac 1 \ell}([0,T];V_{2\sigma+1-2\vartheta})\cap L^2(0,T; V_{2\sigma+1}))}\\
    &\qquad\lesssim
    \norm{DG_{\alpha,\eps}(\xe)[\yeh]}_{
    L^\ell(\Omega; L^\ell(0,T;\cL^2(H,V_{2\sigma})))}.
\end{align*}
The estimate \eqref{est2_mr} follows then from \eqref{aux:7}--\eqref{aux:8} by taking $\lambda$ small enough.
\end{proof}

\begin{remark}
\label{rem_parameters}
Let us point out that
the assumptions of Propositions~\ref{prop_stime_Y_3}--\ref{prop_stime_Y_4}
are compatible with the ones of Theorem~\ref{th:wp}
and of Proposition~\ref{prop_stime_Y_2}.
Indeed, if $\sigma\in[\frac d4,\frac12)$, then
Theorem~\ref{th:wp} allows for 
any values of $\gamma\in[1,+\infty)$
(provided that $\alpha>2$ if $\gamma>\alpha+\beta$ and $\sigma=\frac d4$)
while Proposition~\ref{prop_stime_Y_2}
allows for any value of $p\in(2,+\infty)$: hence, 
the requirements of Propositions~\ref{prop_stime_Y_3}--\ref{prop_stime_Y_4} that 
$\gamma>\beta+2$, $\gamma>2\beta+4$ 
and $p\geq2\frac{\gamma-2}{\gamma-\beta-2}$
cause no issues.
If $\sigma\in(0,\frac d4)$, 
then it is possible to choose some 
$\gamma\geq\frac\beta{4\sigma/d}+2$ in Proposition~\ref{prop_stime_Y_3} which also satisfies the 
assumptions of Theorem~\ref{th:wp} provided 
that either $\frac\beta{4\sigma/d}+2\leq\frac{\alpha-8\sigma/d}{1-4\sigma/d}$ or 
$\frac\beta{4\sigma/d}+2\leq\min\{\alpha+\beta, \frac{\alpha-8\sigma/d+4\beta\sigma/d}{1-4\sigma/d}\}$.
Elementary computations show that $\frac\beta{4\sigma/d}+2\leq\frac{\alpha-8\sigma/d}{1-4\sigma/d}$,
if and only if $\frac\beta{4\sigma/d}+2\leq\alpha+\beta$, 
if and only if $\alpha+\beta\leq\frac{\alpha-8\sigma/d}{1-4\sigma/d}$,
if and only if $\alpha\geq\beta\frac{1-4\sigma/d}{4\sigma/d}+2$, which is exactly to the other 
requirement 
of Proposition~\ref{prop_stime_Y_3}:
in such a case, note that the choice $\gamma=\alpha+\beta$
is always feasible.
Analogously, for Proposition~\ref{prop_stime_Y_4},
one can check that 
$\frac{2(\beta+1)}{4\sigma/d}+2
\leq \frac{\alpha-8\sigma/d}{1-4\sigma/d}$,
if and only if $\frac{2(\beta+1)}{4\sigma/d}+2\leq\alpha+2\beta+2$, 
if and only if $\alpha+2\beta+2\leq\frac{\alpha-8\sigma/d}{1-4\sigma/d}$,
if and only if $\alpha\geq2(\beta+1)\frac{1-4\sigma/d}{4\sigma/d}+2$, which is indeed required in 
Proposition~\ref{prop_stime_Y_4}: in such a case, 
note that the choice $\gamma=\alpha+2\beta+2$
is always feasible.
\end{remark}

\begin{remark}
    \label{rem_max_reg}
    Let us highlight the following points,
    which will be of interest 
    later on:
    \begin{itemize}
    \item for every $\alpha\geq2$, $\beta\geq1$,
    and $\sigma\in[\frac d4,\frac12)$
    (with $\alpha>2$ if $\sigma=\frac d4$),
    all
    the assumptions of Proposition~\ref{prop_stime_Y_3}
    and Theorem~\ref{th:wp}
    are met for all $\gamma>\beta+2$ and
    $\delta>\sigma+\frac d4$;
    \item 
    for every $\alpha\geq4$, $\beta\geq1$, and 
    $\sigma\in[\frac d4,\frac12)$
    (with $\alpha>2$ if $\sigma=\frac d4$),
    all
    the assumptions of Proposition~\ref{prop_stime_Y_4}
    and Theorem~\ref{th:wp}
    are met for all $\gamma>2\beta+4$ and
    $\delta>\sigma+\frac d4$;
        \item for every $\alpha>2$ and $\beta\geq1$,
    there exists $\sigma_0'\in(0, \frac d4)$
    such that all
    the assumptions of Proposition~\ref{prop_stime_Y_3}
    and Theorem~\ref{th:wp}
    are met for $\gamma=\alpha+\beta$,
    for all $\sigma\in(\sigma_0',\frac d4)$
    and $\delta>\sigma+\frac d4$;
    \item for every $\alpha\geq4$ and $\beta\geq1$,
    there exists $\sigma_0''\in(0, \frac d4)$
    such that all
    the assumptions of Proposition~\ref{prop_stime_Y_4}
    and Theorem~\ref{th:wp}
    are met for 
    $\gamma=\alpha+2\beta+2$, for
    all $\sigma\in(\sigma_0'',\frac d4)$
    and $\delta>\sigma+\frac d4$.
    \end{itemize}
    Indeed, for $\sigma\in[\frac d4, \frac12)$ both statements are trivial, so let us focus on the range $\sigma<\frac d4$.
    In this case, the requirements of Proposition~\ref{prop_stime_Y_3}
    and Theorem~\ref{th:wp} come down to 
    \[
    \alpha\geq\beta
    \frac{1-4\sigma/d}{4\sigma/d}+2,\qquad
    \frac{\beta}{4\sigma/d}+2\leq
    \gamma\leq
    \frac{\alpha - 8 \sigma/d}{1- 4\sigma/d},
    \]
    and direct computations show that these are satisfied 
    for $\sigma\in(\sigma_0', \frac d4)$ by choosing 
    \[
    \sigma_0':=
    \frac d4\max\left\{ 
    \frac{\beta}{\alpha+\beta-2},
    \frac{\beta}{\gamma-2}, 
    \frac{\gamma-\alpha}{\gamma-2}
    \right\}\in\left(0, \frac d4\right).
    \]
    In particular, we note that the choice $\gamma=\alpha+\beta$
    gives $\sigma_0'=\frac d4
    \frac{\beta}{\alpha+\beta-2}\in (0,\frac d4)$.
    Analogously,
    the requirements of Proposition~\ref{prop_stime_Y_4}
    and Theorem~\ref{th:wp} reduce to
    \[
    \alpha\geq2(\beta+1)
    \frac{1-4\sigma/d}{4\sigma/d}+2,\qquad
    \frac{2(\beta+1)}{4\sigma/d}+2\leq
    \gamma\leq
    \frac{\alpha - 8 \sigma/d}{1- 4\sigma/d},
    \]
    which are met for all $\sigma\in(\sigma_0'',\frac d4)$ by choosing
    \[
    \sigma_0'':=
    \frac d4\max\left\{ 
    \frac{2\beta+2}{\alpha+2\beta},
    \frac{2\beta+2}{\gamma-2}, 
    \frac{\gamma-\alpha}{\gamma-2}
    \right\}\in\left(0, \frac d4\right).
    \]
    Here, the choice $\gamma=\alpha+2\beta+2$ gives
    $\sigma_0''=\frac d4 
    \frac{2\beta+2}{\alpha+2\beta}\in (0,\frac d4)$.    \end{remark}

    \begin{remark}
    \label{rem_max_reg2}
    In the next section, it will be necessary to introduce 
    a further constraint on $\delta$, in the form $\delta<\frac12+\sigma$.
    Since $\delta>\frac d4+\sigma$
    by assumption of Theorem~\ref{th:wp},
    such requirement will force the space dimension to be $d=1$,
    and will identify $d=2$ as a critical case.
    Let us highlight a direct consequence of 
    Propositions~\ref{prop_stime_Y_3}--\ref{prop_stime_Y_4}
    in the case $d=1$ and $\delta\in(\frac14+\sigma, \frac12+\sigma)$. We have the following two regimes.
    \begin{itemize}
        \item If $\sigma\in(\sigma_0',\frac14)$ and $\delta\in(\frac14+\sigma, \frac12)$,
        then for every $\ell>2$ and $T>0$,
        there exists $C>0$,
        only depending on $T,\ell,\delta,\sigma,d$, such that,
        for every $h\in V_{2\delta}$ and 
        for every $\eps\in(0,1)$,
        \[
         \norm{Y_\eps^h}_{L^\ell_\cP(\Omega; 
         C([0,T]; V_{2\delta})
         \cap L^2(0,T; V_{1}))} \leq C
          \left( 1
          +\|\Psi_{\gamma}(x)\|_{L^1(\Td)}^{2\sigma}\right)
        \|h\|_{V_{2\delta}}.
         \]
        \item If $\sigma\in(\sigma_0'',\frac14)$ and $\delta\in[\frac12,\frac12+\sigma)$,
        then for every $\ell>2$ and $T>0$,
        there exists $C>0$,
        only depending on $T,\ell,\delta,\sigma,d$, such that,
        for every $x\in V_{2\xi}$ and $h\in V_{2\delta}$ and 
        for every $\eps\in(0,1)$,
        \[
         \norm{Y_\eps^h}_{L^\ell_\cP(\Omega; 
         C([0,T]; V_{2\delta})
         \cap L^2(0,T; V_{2\sigma+1}))} \leq C
          \left( 1
          +\|\Psi_{\gamma}(x)\|_{L^1(\Td)}^{\frac32+2\sigma}
          +\|x\|_{V_{2\xi}}^{\frac32+2\sigma}\right)
        \|h\|_{V_{2\delta}}.
         \]
         \item If $\sigma\in[\frac14,\frac12)$ and $\delta\in(\frac14+\sigma,\frac12+\sigma)$,
        then for every $\ell>2$ and $T>0$,
        there exists $C>0$,
        only depending on $T,\ell,\delta,\sigma,d$, such that,
        for every $x\in V_{2\xi}$ and $h\in V_{2\delta}$ and 
        for every $\eps\in(0,1)$,
        \[
         \norm{Y_\eps^h}_{L^\ell_\cP(\Omega; 
         C([0,T]; V_{2\delta})
         \cap L^2(0,T; V_{2\sigma+1}))} \leq C
          \left( 1
          +\|\Psi_{\gamma}(x)\|_{L^1(\Td)}^{\frac32+\frac\beta{2(\gamma-2)}}
          +\|x\|_{V_{2\xi}}^{
          \frac32+\frac\beta{2(\gamma-2)}}\right)
        \|h\|_{V_{2\delta}}.
         \]
    \end{itemize}
    Indeed, one can choose either
    $\vartheta=\frac12-\delta\in(0, 
    \frac12)$ in \eqref{est_mr} when $\delta<\frac12$
    or $\vartheta=\frac12+\sigma-\delta\in(0, 
    \frac12]$ in \eqref{est2_mr} when $\delta\geq\frac12$: 
    the results
    follow from Propositions~\ref{prop_stime_Y_3}--\ref{prop_stime_Y_4} by the fact that $V_{2\delta}\embed L^p(\Td)$.
\end{remark}

\section{Strong Feller property}
\label{sec:str_fell}
This section contains the proof of Theorem~\ref{th:sf}
on the strong Feller properties for equation \eqref{eq_ast}.
The proof is structured in several steps.
First, we give rigorous meaning to the inverse
of the noise coefficient appearing in \eqref{eq_ast}.
Secondly, we exploit the Bismut-Elworthy-Li formula
at the $\eps$-approximate level
(where the noise is non-degenerate)
in order to 
suitably represent the derivative of the transition semigroup
associated to \eqref{eq_app}.
Lastly, we 
prove uniform estimate on the derivatives of the transition semigroup
associated to \eqref{eq_app}, by relying on 
the results of Section~\ref{sec:Y}. This allows in turn
to prove the required strong Feller property for 
the limiting transition semigroup associated to \eqref{eq_ast}.
Throughout the section, we work under the assumptions
of Theorem~\ref{th:wp} and in the 
setting of Sections~\ref{sec:X}--\ref{sec:Y}.

\subsection{The inverse of the degenerate noise coefficient}
\label{sec:inv}
We rigorously introduce here the inverse of the noise 
operator $G_{\alpha}:B^\infty_1\to\cL^2(H,H)$. 
Since $G_\alpha$ is degenerate,
this requires some careful considerations.

Let us first argue heuristically
in order to understand what is the natural candidate for the inverse.
Let $x\in B^\infty_1$ and $u,v\in H$ be such that 
$G_\alpha(x)[u]=v$: by definition of $G_\alpha$ this means that
\[
  m_{\frac\alpha2}(x)\cdot(\operatorname{I}+A)^{-\delta}u = v
  \quad\text{a.e.~in } \Td.
\]
Now, the degeneracy of $m_{\frac\alpha2}$ at $\pm1$
causes issues in inverting $G_\alpha$ when $x$ takes values
in $\pm1$. However, when $x$ belongs also to some 
$\mathcal K_\s$ for some $\s>1$, then actually 
$|x|<1$ almost everywhere, hence the relation can be written 
equivalently as
\[
(\operatorname{I}+A)^{-\delta}u=
\frac1{m_{\frac\alpha2}(x)}v=\Psi''_{\frac\alpha2}(x)v,
\]
implying a fortiori that
$\Psi''_{\frac\alpha2}(x)v\in D(A^{\delta})=V_{2\delta}$ 
and that 
\[
u=(\operatorname{I}+A)^{\delta}(\Psi''_{\frac\alpha2}(x)v).
\]
Notice also that if $v\in V_{2\delta}$,
then the condition $\Psi''_{\frac\alpha2}(x)v\in V_{2\delta}$
is implied by $\Psi_{\frac\alpha2}''(x)\in V_{2\delta}$
by Lemma~\ref{multiplication} since $\delta>\frac d4$.
These considerations suggest that for every $x$
such that $\Psi_{\frac\alpha2}''(x)\in V_{2\delta}$, 
the linear operator $G_{\alpha}(x):H\to V_{2\delta}$
is invertible with inverse in
$\cL(V_{2\delta}, H)$.
With a slight abuse of notation, 
we will denote the inverse of the linear operator $G_\alpha(x)$ as $G_\alpha^{-1}(x)$, 
instead of $G_\alpha(x)^{-1}$. In such a way, 
the operator $G_\alpha^{-1}$ will have to be considered 
as the nonlinear mapping associating every 
$x$ to
the inverse of the linear operator $G_{\alpha}(x)$
(and not as the inverse of the map
$G_{\alpha}:B^\infty_1\to\cL^2(H,H)$).
Let us stress that this notational choice is classical 
in the literature, see e.g.~\cite{dapratozab}.

Let us proceed now rigorously. We define
the set
\[
  \mathcal D_{\delta}:=\left\{
  v:\Td\to(-1,1) \text{ measurable:}\quad
  v\in V_{2\delta},\;
  \Psi_{\frac\alpha2}''(v)\in V_{2\delta}\right\}.
\]
Let us first note that actually $\mathcal D_{\delta}$
does not depend on $\alpha$: indeed, since $\delta>\frac d4$, 
for every $x\in\mathcal D_{\delta}$ it holds that 
$\Psi''_{\frac\alpha2}(x)\in V_{2\delta}\embed L^\infty(\Td)$, 
hence $x$ is separated from $\pm1$, 
i.e.~$\|x\|_{L^\infty(\Td)}<1$. This implies the characterisation
\[
  \mathcal D_\delta=
  V_{2\delta}\cap\mathcal K_\infty.
\]
Bearing this in mind, we define the operator
\[
  G_\alpha^{-1}:\mathcal D_{\delta}
  \to\cL(V_{2\delta}, H)
\]
as
\[
  G_\alpha^{-1}(x)[v]:=
  (\operatorname{I}+A)^{\delta}
  \left(\Psi_{\frac\s2}''(x)v\right),
  \quad v\in V_{2\delta}, \quad x\in \mathcal D_{\delta}.
\]
Note first of all that $G_\alpha$ is actually well-defined.
Indeed, if $x\in \mathcal D_{\delta}$, then one has that 
$\Psi_{\frac\s2}''(x)\in V_{2\delta}$ and $v\in V_{2\delta}$:
hence, since $\delta>\frac d4$,
by Lemma~\ref{multiplication}
it holds also that $\Psi_{\frac\s2}''(x)v\in V_{2\delta}$,
and the definition makes sense.
Furthermore, the fact that 
$G_\alpha^{-1}(x):v\mapsto G_\alpha^{-1}(x)[v]$ is linear is obvious.
As far as continuity is concerned, note that, again by Lemma~\ref{multiplication}, one has
\begin{align*}
\sup_{\|v\|_{V_{2\delta}}\leq1}
\|G_\alpha^{-1}(x)[v]\|_{H}
&=\sup_{\|v\|_{V_{2\delta}}\leq1}
\norm{(\operatorname{I}+A)^{\delta}
  \left(\Psi_{\frac\alpha2}''(x)v\right)}_H\\
&\lesssim
\sup_{\|v\|_{V_{2\delta}}\leq1}
\norm{\Psi_{\frac\alpha2}''(x)v}_{V_{2\delta}}
\lesssim
\norm{\Psi''_{\frac\alpha2}(x)}_{V_{2\delta}}.
\end{align*}
This shows that $G_\alpha^{-1}$ is well-defined
with values in $\cL(V_{2\delta}, H)$ and that 
\begin{equation}\label{est:inv_G}
  \norm{G_\alpha^{-1}(x)}_{\cL(V_{2\delta}, H)}
  \lesssim \norm{\Psi''_{\frac\alpha2}(x)}_{V_{2\delta}}
  \quad\forall\,x\in \mathcal D_{\delta}.
\end{equation}
Moreover, it is immediate to verify that, 
for every $x\in \mathcal D_{\delta}$,
the linear operator $G_{\alpha}^{-1}(x)$ is the inverse of $G_\alpha(x)$,
namely that $G_\alpha(x)G_{\alpha}^{-1}(x)=\operatorname{I}_{V_{2\delta}}$
and $G_\alpha^{-1}(x)G_{\alpha}(x)=\operatorname{I}_{H}$.
Indeed, one has that 
\begin{align*}
G_\alpha(x)G_\alpha^{-1}(x)[v]&=
G_\alpha(x)\left[(\operatorname{I}+A)^\delta
\left(\Psi_{\frac\alpha2}''(x)v \right) \right]
= 
m_{\frac\alpha2}(x)(\operatorname{I}+A)^{-\delta}(\operatorname{I}+A)^\delta\left(\Psi_{\frac\alpha2}''(x)v\right)\\
&=
m_{\frac\alpha2}(x)\Psi_{\frac\alpha2}''(x)v=v \qquad\forall\,v\in V_{2\delta}
\end{align*}
and 
\begin{align*}
G_\alpha^{-1}(x)G_\alpha(x)[u]&=
G_\alpha^{-1}(x)\left[m_{\frac\alpha2}(x)
(\operatorname{I}+A)^{-\delta}u
\right]
= 
(\operatorname{I}+A)^{\delta}
\left(\Psi_{\frac\alpha2}''(x)
m_{\frac\alpha2}(x)
(\operatorname{I}+A)^{-\delta}u
\right)\\
&=(\operatorname{I}+A)^{\delta}
(\operatorname{I}+A)^{-\delta}u=u
\qquad\forall\,u\in H.
\end{align*}

By arguing analogously, for every $\eps\in(0,1)$ 
we introduce the inverse of the approximated operator $G_{\alpha,\eps}$
as
\[
  G_{\alpha,\eps}^{-1}:V_{2\delta}\to \cL(V_{2\delta}, H),
  \qquad
  G_{\alpha,\eps}^{-1}(x)[v]:=
  (\operatorname{I}+A)^{\delta}
  \left(\Psi_{\frac\alpha2,\eps}''(x)v\right),
  \quad v\in V_{2\delta}, \quad x\in V_{2\delta}.
\]
Note that  since $\Psi_{\frac\alpha2,\eps}''$ 
is smooth and bounded,
the fact that $x\in V_{2\delta}$ implies directly 
that $\Psi_{\frac\alpha2,\eps}''(x)\in V_{2\delta}$ too,
hence also that $\Psi_{\frac\alpha2,\eps}''(x)v\in V_{2\delta}$
by Lemma~\ref{multiplication}.
Computations similar to the ones performed above
show  that $G_{\alpha,\eps}^{-1}$ is well-defined
with values in $\cL(V_{2\delta}, H)$ and
there exists a constant $C>0$, independent of $\eps$, 
such that 
\begin{equation}
    \label{est:inv_eps}
  \norm{G_{\alpha,\eps}^{-1}(x)}_{\cL(V_{2\delta}, H)}
  \leq C\norm{\Psi''_{\frac\alpha2,\eps}(x)}_{V_{2\delta}}
  \quad\forall\,x\in V_{2\delta}.
\end{equation}
Moreover, it is immediate to verify that, 
for every $x\in V_{2\delta}$,
the linear operator $G_{\alpha,\eps}^{-1}(x)$ is the inverse of $G_{\alpha,\eps}(x)$,
namely that $G_{\alpha,\eps}(x)G_{\alpha,\eps}^{-1}(x)=\operatorname{I}_{V_{2\delta}}$
and $G_{\alpha,\eps}^{-1}(x)G_{\alpha,\eps}(x)=\operatorname{I}_{H}$.

We prove now 
uniform estimates on the family of processes 
$(G_{\alpha,\eps}^{-1}(X_\eps^x))_\eps$.

\begin{prop}
\label{prop:Xe_inverse}
Under the assumptions of Theorem~\ref{th:wp},
let $d=1$, let 
$\delta\in(\frac14+\sigma, \frac12+\sigma)$,
let $x\in V_{2\xi}$ for some $\xi\in[0,\frac12+\sigma)$,
and
assume one of the following settings:
\begin{itemize}
    \item $\delta\in(\frac14,\frac12)$, 
    $\gamma\geq\alpha+2$, $\xi\geq0$,
    $r:=\frac1\delta$, and
    $\mathfrak p:=\frac12+\delta$;
    \item $\delta\in[\frac12,\frac34]$, 
    $\gamma\geq\alpha+4$,
    $\xi\geq\delta$,
    $r\in(2,4)$, and
    $\mathfrak p:=\frac32+\frac1r$;
    \item  $\delta\in(\frac34,1)$, 
    $\gamma\geq\alpha+4$, 
    $\xi\geq\delta$,
    $r:=\frac2{2\delta-1}$, and
    $\mathfrak p:=1+\delta$.
\end{itemize}
Then,
for every $\ell\geq2$ and $T>0$, there exists a constant
$C>0$, 
only depending on $\alpha,\beta,\gamma,\delta,\sigma,\ell, T$,
such that, for every $\eps\in(0,1)$, it holds that
\begin{align*}
\|G_{\alpha, \varepsilon}^{-1}(\xe)\|_{
L^\ell(\Omega; L^r(0,T;\cL(V_{2\delta},H)))}
\le C
\left(1+\|\Psi_{\gamma}(x)\|^{\mathfrak p}_{L^1(\Td)}+
  \|x\|_{V_{2\xi}}^{\mathfrak p}\right).
\end{align*}
Furthermore, the same conclusion holds also 
under any of the following stronger scenarios:
\begin{itemize}
    \item $\delta\in(\frac14,\frac12)$, 
    $\gamma>\max\{1,\frac1{4\sigma}\}(\alpha+2)+2$,
    $\xi\geq\frac12+\frac{\alpha+2}{4(\gamma-2)}$,
    $r:=\infty$, and $\mathfrak p:=\frac12+2\delta$;
    \item $\delta\in[\frac12,\frac34]$, 
    $\gamma>\max\{1,\frac1{4\sigma}\}(\alpha+4)+2$,
    $\xi\geq\max\{\delta,\frac12+\frac{\alpha+4}{4(\gamma-2)}\}$,
    $r:=\infty$, and $\mathfrak p\in(2,\frac52)$;
    \item $\delta\in(\frac34,1)$, 
    $\gamma>\max\{1,\frac1{4\sigma}\}(\alpha+4)+2$,
    $\xi\geq\max\{\delta,\frac12+\frac{\alpha+4}{4(\gamma-2)}\}$,
    $r:=\infty$, and $\mathfrak p:=\frac12+2\delta$.
\end{itemize}
\end{prop}

\begin{proof}
We divide the proof in the cases 
$\delta<\frac12$ and $\delta\geq\frac12$.

\noindent\underline{\sc Case $\delta<1/2$.}
By interpolation and Lemma~\ref{lem_app}-(v) we have
\begin{align*}
    \|G_{\alpha, \varepsilon}^{-1}(\xe)\|_{\cL(V_{2\delta},H)}
    &\lesssim \|\Psi''_{\frac \alpha2,\varepsilon}(\xe)\|_{V_{2\delta}} \\
    &\lesssim
    \|\Psi''_{\frac \alpha2,\varepsilon}(\xe)\|_H
    +
    \|\Psi''_{\frac \alpha2,\varepsilon}(\xe)\|_H^{1-2\delta} \|\nabla\Psi''_{\frac \alpha2,\varepsilon}(\xe)\|_H^{2\delta} 
    \\
    &= \|\Psi''_{\frac \alpha2,\varepsilon}(\xe)\|_H
    + \|\Psi''_{\frac \alpha2,\varepsilon}(\xe)\|_H^{1-2\delta} \|\Psi'''_{\frac \alpha2,\varepsilon}(\xe)\nabla X_\eps^x\|_H^{2\delta} 
    \\
    & \lesssim\|\Psi_{\alpha+2,\varepsilon}(\xe)\|_{L^1(\Td)}^{\frac12}
    + \|\Psi_{\alpha+2,\varepsilon}(\xe)\|_{L^1(\Td)}^{\frac12-\delta} \|\Psi''_{\frac \alpha2+1,\varepsilon}(\xe)\nabla\xex\|^{2\delta}_H.
\end{align*}
Since $\gamma\geq\alpha+2$ by assumption,
by the H\"older inequality in time we get
\begin{multline*}
  \|G_{\alpha, \varepsilon}^{-1}(\xe)\|_{
  L^{\frac1\delta}(0,T;\cL(V_{2\delta},H))}
  \lesssim
  \|\Psi_{\gamma,\varepsilon}(\xe)\|_{L^\infty(0,T;L^1(\Td))}^{\frac12}\\
  +\|\Psi_{\gamma,\varepsilon}(\xe)\|_{
  L^\infty(0,T;L^1(\Td))}^{\frac12-\delta}
  \|\Psi''_{\frac \gamma2,\varepsilon}(\xe)\nabla\xex\|^{2\delta}_{L^2(0,T; H)}
\end{multline*}
and Proposition~\ref{prop3Xeps} yields
the required estimate for $x\in\mathcal K_\gamma$.
If also $x\in V_{1+\frac{\alpha+2}{2(\gamma-2)}}$,
by proceeding as in the proof of Proposition~\ref{prop_stime_Y_4} instead,
one has for every $\kappa>2$ that  
\begin{align*}
  \|G_{\alpha, \varepsilon}^{-1}(\xe)\|_{\cL(V_{2\delta},H)}
  &\lesssim
  \|\Psi_{\alpha+2,\varepsilon}(\xe)\|^{\frac12}_{L^1(\Td)}\\
  &+\|\Psi_{\alpha+2,\varepsilon}(\xe)\|_{L^1(\Td)}^{\frac12-\delta}
  \|\Psi_{\kappa(\frac\alpha2 +1)+2,\varepsilon}(\xe)\|_{L^1(\Td)}^{\frac{2 \delta}{\kappa}}
  \|\xe\|_{V_{1+\frac d\kappa}}^{2\delta},
\end{align*}
where the terms on the right-hand side are 
bounded in $L^l(\Omega; L^\infty(0,T))$ for all $l\geq2$
by Propositions~\ref{prop3Xeps} and \ref{prop5Xeps}
provided that 
$\alpha+2\leq\gamma$,
$\kappa(\frac\alpha2 +1)+2\leq\gamma$, 
and $1+\frac1\kappa< 1+2\sigma$:
the first condition is true by assumption,
while 
the second and third ones require $\frac1{2\sigma}<\kappa\leq2\frac{\gamma-2}{\alpha+2}$.
Hence, one can choose 
$\kappa=2\frac{\gamma-2}{\alpha+2}$
since $2\frac{\gamma-2}{\alpha+2}>2$ and $
2\frac{\gamma-2}{\alpha+2}>\frac 1{2\sigma}$
by assumption.
The required inequality follows then by combining the terms.

\noindent\underline{\sc Case $\delta\geq1/2$.}
For every 
$\zeta\in(\frac 14
,\frac12]\cap[\delta-\frac12,\frac12]$, Lemma~\ref{multiplication} and Lemma~\ref{lem_app}-(v) yield
\begin{align*}
    \|G_{\alpha, \varepsilon}^{-1}(\xe)\|_{\cL(V_{2\delta},H)}
    &\lesssim\|\Psi''_{\frac \alpha2,\varepsilon}(\xe)\|_{V_{2\delta}}
    \lesssim 
    \|\Psi''_{\frac \alpha2,\varepsilon}(\xe)\|_H+
    \|\Psi'''_{\frac\alpha2, \eps}(\xe)
    \nabla\xe\|_{V_{2(\delta-\frac12)}}\\
    &\lesssim
    \|\Psi''_{\frac \alpha2,\varepsilon}(\xe)\|_H+
    \|\Psi'''_{\frac\alpha2, \eps}(\xe)\|_{V_{2\zeta}}
    \|\nabla\xe\|_{V_{2\delta-1}}\\
    &\lesssim\|\Psi_{\alpha+2,\varepsilon}(\xe)\|_{L^1(\Td)}^{\frac12}+
    \|\Psi'''_{\frac\alpha2, \eps}(\xe)\|_{V_{2\zeta}}
    \|\xe\|_{V_{2\delta}}.
\end{align*}
By proceeding as in the previous step,
by interpolation and Lemma~\ref{lem_app}-(v) we have
\begin{align*}
    \|\Psi'''_{\frac\alpha2, \eps}(\xe)\|_{V_{2\zeta}}
    &\lesssim \|\Psi''_{\frac{\alpha+2}2,\varepsilon}(\xe)\|_{V_{2\zeta}} \\
    & \lesssim\|\Psi_{\alpha+4,\varepsilon}(\xe)\|_{L^1(\Td)}^{\frac12}
    + \|\Psi_{\alpha+4,\varepsilon}(\xe)\|_{L^1(\Td)}^{\frac12-\zeta} 
    \|\Psi''_{\frac{\alpha+4}2,\varepsilon}(\xe)\nabla\xex\|^{2\zeta}_H.
\end{align*}
Since $\gamma\geq\alpha+4$ by assumption,
by the H\"older inequality in time we get
\begin{multline*}
  \|G_{\alpha, \varepsilon}^{-1}(\xe)\|_{
  L^{\frac1\zeta}(0,T;\cL(V_{2\delta},H))}
  \lesssim
  \|\Psi_{\gamma,\varepsilon}(\xe)\|_{L^\infty(0,T;L^1(\Td))}^{\frac12}
  \left(1+
  \|\xe\|_{L^\infty(0,T;V_{2\delta})}\right)\\
  +\|\Psi_{\gamma,\varepsilon}(\xe)\|_{
  L^\infty(0,T;L^1(\Td))}^{\frac12-\zeta}
  \|\Psi''_{\frac \gamma2,\varepsilon}(\xe)\nabla\xex\|^{2\zeta}_{L^2(0,T; H)}
  \|\xe\|_{L^\infty(0,T;V_{2\delta})}
\end{multline*}
and Proposition~\ref{prop3Xeps} yields
the required estimate,
by noting that one can take $\zeta=\delta-\frac12$
if $\delta\in(\frac34, 1)$ and any $\zeta\in(\frac14, \frac12)$
if $\delta\in[\frac12,\frac34]$.
If also $x\in V_{1+\frac{\alpha+4}{2(\gamma-2)}}$,
by proceeding as in the proof of the previous step
one has for every $\kappa>2$ that  
\begin{align*}
  \|G_{\alpha, \varepsilon}^{-1}(\xe)\|_{\cL(V_{2\delta},H)}
  &\lesssim
  \|\Psi_{\gamma,\varepsilon}(\xe)\|^{\frac12}_{L^1(\Td)}
  \left(1+\|\xe\|_{V_{2\delta}}\right)\\
  &+\|\Psi_{\alpha+4,\varepsilon}(\xe)\|_{L^1(\Td)}^{\frac12-\zeta}
  \|\Psi_{\kappa\frac{\alpha+4}2+2,\varepsilon}(\xe)\|_{L^1(\Td)}^{\frac{2 \zeta}{\kappa}}
  \|\xe\|_{V_{1+\frac 1\kappa}}^{2\zeta}
  \|\xe\|_{V_{2\delta}},
\end{align*}
where the terms on the right-hand side are 
bounded in $L^l(\Omega; L^\infty(0,T))$ for all $l\geq2$
by Propositions~\ref{prop3Xeps} and \ref{prop5Xeps}
provided that 
$\alpha+4\leq\gamma$,
$\kappa\frac{\alpha+4}2 +2\leq\gamma$, 
and $1+\frac1\kappa< 1+2\sigma$:
the first condition is true by assumption,
while 
the second and third ones require $\frac1{2\sigma}<\kappa\leq2\frac{\gamma-2}{\alpha+4}$.
Hence, one can choose 
$\kappa=2\frac{\gamma-2}{\alpha+4}$
since $2\frac{\gamma-2}{\alpha+4}>2$ and $
2\frac{\gamma-2}{\alpha+4}>\frac 1{2\sigma}$
by assumption.
The required inequality follows then by combining the terms.
\end{proof}

\begin{remark}
    \label{rem:G_inv}
    Let us comment on the range of values 
    for $\sigma$ and $\delta$ appearing in Proposition~\ref{prop:Xe_inverse}.
    First of all, let us highlight that the restriction to 
    dimension $d=1$ arises since Proposition~\ref{prop:Xe_inverse} strongly relies on the 
    maximal regularity $C([0,T]; V_{2\delta})\cap L^2(0,T; V_{2\sigma+1})$ for $\xex$: hence, this naturally forces
    $\delta$ to satisfy the constraint $2\delta\leq2\sigma+1$.
    Since by assumption $\delta>\sigma+\frac d4$, 
    dimension $d=2$ appears to be critical under this approach
    as one would obtain both $\sigma+\frac12<\delta\leq\sigma+\frac12$.
    By contrast, dimension $d=1$ allows for admissible ranges of $\delta$ and $\sigma$.
    More precisely, for $\delta\in(\frac14,\frac12]$
    one has $\sigma\in(0, \frac14)$,
    while for $\delta\in(\frac12,1)$
    one has $\sigma\in(0, \frac12)$.
\end{remark}

\begin{remark}
    Note that the assumptions of 
    Proposition~\ref{prop:Xe_inverse} are compatible 
    with the ones of Theorem~\ref{th:wp}.
    Indeed, if $\sigma\geq\frac14$ then all values of $\gamma$
    are admissible in Theorem~\ref{th:wp}
    (with $\alpha>2$ if $\sigma=\frac14$),
    so the requirements on the amplitude of $\gamma$
    in Proposition~\ref{prop:Xe_inverse} cause no issues.
    Otherwise, if $\sigma\in (0,\frac14)$,
    $\gamma\geq\alpha+2$ can be chosen in Proposition~\ref{prop:Xe_inverse} provided that 
    $\alpha+2\leq\frac{\alpha-8\sigma}{1-4\sigma}$, 
    i.e.~$\alpha\geq\frac2{4\sigma}$: hence, any values
    $\alpha>2$ are included for $\sigma>\frac1{2\alpha}$.
    Analogously, $\gamma\geq\alpha+4$ can be chosen in Proposition~\ref{prop:Xe_inverse} provided that 
    $\alpha+4\leq\frac{\alpha-8\sigma}{1-4\sigma}$, 
    i.e.~$\alpha\geq\frac{1-2\sigma}{\sigma}$: hence, any values
    $\alpha>2$ are included for $\sigma>\frac1{\alpha+2}$.
    Similarly, the condition 
    $\frac{\alpha+2}{4\sigma}+2<\frac{\alpha-8\sigma}{1-4\sigma}$
    yields $\alpha>2$ and $\sigma>\frac18+\frac1{4\alpha}$:
    here, the choice $\gamma=\alpha+\beta$ can be done 
    for $\sigma>\frac14\frac{\alpha+2}{\alpha+\beta-3}$,
    while $\gamma=\alpha+2\beta+2$ requires
    $\sigma>\frac14\frac{\alpha+2}{\alpha+2\beta}$.
    Lastly,  
    $\frac{\alpha+4}{4\sigma}+2<\frac{\alpha-8\sigma}{1-4\sigma}$
    yield $\alpha>2$ and $\sigma>\frac18\frac{\alpha+4}{\alpha+1}$: here, $\gamma=\alpha+2\beta+2$ can be taken 
    provided that $\sigma>\frac14\frac{\alpha+4}{\alpha+2\beta}$.
\end{remark}

\subsection{Uniform estimates on the Bismut-Elworthy-Li formula}
We introduce here the transition semigroup associated to 
the approximated problem \eqref{eq_app} and 
show a Bismut-Elworthy-Li formula
to represent its derivative. Then, we exploit Proposition~\ref{prop:Xe_inverse} to prove uniform estimates 
on the derivatives of the approximated semigroup.

For every $\varepsilon \in (0,1)$,
the transition Markov semigroup 
$\pe:=(\pe_t)_{t \ge 0}$ associated to equation \eqref{eq_app}
is defined as 
\begin{equation}
\label{P_t_eps}
\pe_t\varphi(x):= \mathbb{E}[ \varphi(\xex(t))], \quad x\in H,\quad \varphi\in \mathcal{B}_b(H).
\end{equation}
The fact that $P^\eps$ is a well-defined Markov semigroup on
$\mathcal{B}_b(H)$ is a well-known result that follows from 
Proposition~\ref{ex_app_thm} and the fact that $\Psi'_{\beta,\eps}$
and $G_{\alpha,\eps}$ are Lipschitz-continuous.

Let us note that since $m_{\alpha,\eps}\geq\eps^\alpha$
the approximated noise operator 
$G_{\alpha,\eps}:H\to\cL^2(H,H)$ is nondegenerate,
in the sense that 
\begin{align*}
  \norm{G_{\alpha,\eps}(x)}_{\cL^2(H,H)}^2
  &=\sum_{k\in\mathbb N_+}
  \norm{m_{\frac\alpha2,\eps}(x)(\operatorname{I}+A)^{-\delta}e_k}_H^2
  =\sum_{k\in\mathbb N_+}\int_{\Td}
  m_{\alpha,\eps}(x)|(\operatorname{I}+A)^{-\delta}e_k|^2\\
  &\geq\eps^{\alpha}\sum_{k\in\mathbb N_+}
  \norm{(\operatorname{I}+A)^{-\delta}e_k}^2_H
  =\eps^\alpha\norm{(\operatorname{I}+A)^{-\delta}}_{\cL^2(H,H)}^2.
\end{align*}
The nondegeneracy of $G_{\alpha,\eps}$ allows to 
prove, at the $\eps$-approximated level, a
Bismut-Elworthy-Li formula to represent the derivative of
the semigroup $P^\eps$. 
Let us stress that even if the $\eps$-regularised
noise coefficient is non-degenerate,
the proof requires careful examination since 
the image of the inverse is still
$V_{2\delta}$, and not $H$
as in more classical settings.
More precisely, we have the following result.

\begin{prop}[Bismut-Elworthy-Li formula for $P^\eps$]
\label{BEL3}
Under the assumptions of Theorem~\ref{th:wp}
and Proposition~\ref{prop:Xe_inverse},
suppose that $\xi>\max\{\delta-\frac12,
\delta-\sigma-\frac14\}$, and
let 
$\eps\in(0,1)$,
$\varphi\in C^2_b(H)$, and $t>0$. Then, 
$P_t^\eps\varphi$ is G\^ateaux-differentiable in $V_{2\xi}$ along directions of $V_{2\delta}$, and 
\[
D(P^\eps_t\varphi)(x)[h]=\frac 1t \mathbb{E}\left[ \varphi(\xex(t))\int_0^t\left( G_{\alpha,\varepsilon}^{-1}(\xex(s))\yeh(s), \d W(s)\right)_H \right]
\quad\forall\,x\in V_{2\xi},\quad\forall\,h\in V_{2\delta}.
\]
\end{prop}
\begin{proof}
The proof is based on a further approximation on 
on the operators $\Psi_{\beta,\eps}$ and $G_{\alpha,\eps}$
in order to recover smoothness, and that allow to 
exploit the classical Bismut-Elworthy-Li
in the non-degenerate case of noise with uniformly bounded 
inverse. Then, a passage to the limit
concludes the proof. We proceed in three steps.

\noindent\underline{\sc The approximation}.
For every $n\in\enne$, let $P_n:H\to H_n:=\operatorname{span}\{e_1,\ldots,e_n\}$ the orthogonal projection.
We define
\begin{alignat*}{2}
  &\mathfrak F_{\beta,\eps,n}:H_n\to \erre, \qquad
  &&\mathfrak F_{\beta,\eps,n}(x):=\int_{\Td}
  \Psi_{\beta,\eps}
  (x),
  \quad x\in H_n,\\
  &\mathfrak G_{\alpha,\eps,n}:H_n\to \cL^2(H,H_n)\,, \qquad
  &&\mathfrak G_{\alpha,\eps,n}(x)[v]:=
  P_nG_{\alpha,\eps}(x), \quad x\in H_n.
\end{alignat*}
By recalling that $\Psi_{\beta,\eps}\in C^\infty(\erre)$
with bounded derivatives of any order $\geq2$ and
that $m_{\frac\alpha2, \eps}\in C^\infty_b(\erre)$
with bounded derivatives of any order, 
since $H_n\subset L^\infty(\Td)$ it readily follows that 
$\mathfrak F_{\beta,\eps,n}\in C^\infty(H_n)$
with bounded derivatives of order $\geq2$
and $\mathfrak G_{\alpha,\eps,n}\in C^\infty_b(H_n)$
with bounded derivatives of order.
In particular, it holds that 
\[
  D\mathfrak F_{\beta,\eps,n}\in C^\infty(H_n;H_n),
  \qquad
  D\mathfrak F_{\beta,\eps,n}(x)=
  P_n\Psi_{\beta,\eps}'(x) \quad\forall\,x\in H_n,
\]
and 
\[
  D\mathfrak G_{\alpha,\eps,n}\in C^\infty(H_n;\cL^2(H,H_n)),
  \qquad
  D\mathfrak G_{\alpha,\eps,n}(x)=
  P_n DG_{\alpha,\eps}(x) \quad\forall\,x\in H_n.
\]
For such an approximation, it is not guaranteed that 
$\mathfrak G_{\alpha,\eps,n}(x)$ is invertible for every $x\in H_n$.
Nonetheless, we can prove that it always admits a right-inverse.
Indeed, since $H_n\subset V_{2\delta}$
we can define the operator 
\[
  \mathfrak G_{\alpha,\eps,n}^{-1}:H_n\to \cL(H_n,H),
  \qquad
  \mathfrak G_{\alpha,\eps,n}^{-1}(x)[v]:=
  G_{\alpha,\eps}^{-1}(x)[v], \quad v,x\in H_n,
\]
which satisfies, for every $x,v\in H_n$, 
\[
  \mathfrak G_{\alpha,\eps,n}(x)
  \mathfrak G_{\alpha,\eps,n}^{-1}(x)[v]
  =P_n\left[G_{\alpha,\eps}(x)G_{\alpha,\eps}^{-1}(x)[v]
  \right]=P_nv=v.
\]
In other words, we have that $\mathfrak G_{\alpha,\eps,n}(x)
\mathfrak G_{\alpha,\eps,n}^{-1}(x)=\operatorname{I}_{H_n}$
for every $x\in H_n$. Let us stress again that 
$\mathfrak G_{\alpha,\eps,n}^{-1}(x)$ may not be 
also left-inverse of $\mathfrak G_{\alpha,\eps,n}(x)$,
for every $x\in H_n$.

Now, for every $x\in H_n$, 
the approximated equation
\[
  \d X_{\eps,n}^x+AX_{\eps,n}^x\,\d t
  +D\mathfrak F_{\beta,\eps,n}(X_{\eps,n}^x)\,\d t
  =\mathfrak G_{\alpha,\eps,n}(X_{\eps,n}^x)\,\d W,
  \qquad X_{\eps,n}^x(0)=x,
\]
admits a unique solution  
$X^x_{\eps,n}\in L^\ell_\cP(\Omega;
C([0,T]; H_n))$
for every $\ell\geq2$ and $T>0$.
Moreover, a direct application of the It\^o formula and 
classical computations on finite dimensional approximations
yield that, as $n\to\infty$,
\[
  X_{\eps,n}^x\to\xex \quad\text{in } 
  L^\ell_\cP(\Omega; C([0,T]; H)\cap L^2(0,T; V))
  \quad\forall\,\ell\geq2, \quad\forall\,T>0.
\]
For every $\ell\geq2$ and $T>0$,
due to the smoothness of the coefficients, 
the solution map 
\[
S_{\eps,n}:H\to
L^\ell_\cP(\Omega; C([0,T]; H)\cap L^2(0,T; V)),
\qquad
  S_{\eps,n}:x\mapsto X_{\eps,n}^x, \quad x\in H,
\]
is in particular 
twice Fr\'echet-differentiable with continuous derivatives,
and for every $x,h\in H_n$ it holds that $DS_{\eps,n}(x)[h]=Y_{\eps,n}^h$, where 
$Y_{\eps,n}^h\in L^\ell_\cP(\Omega; C([0,T]; H)\cap L^2(0,T; V))$
is the unique solution to the approximated equation
\[
  \d Y_{\eps,n}^h+AY_{\eps,n}^h\,\d t
  +D^2\mathfrak F_{\beta,\eps,n}(X_{\eps,n}^x)[Y_{\eps,n}^h]\,\d t
  =D\mathfrak G_{\alpha,\eps,n}(X_{\eps,n}^x)[Y_{\eps,n}^h]\,\d W,
  \qquad Y_{\eps,n}^h(0)=h.
\]
Again, a direct application of the It\^o formula and 
of the finite dimensional approximations yield that 
\[
  Y_{\eps,n}^h\to\yeh \quad\text{in } 
  L^\ell_\cP(\Omega; C([0,T]; H)\cap L^2(0,T; V))
  \quad\forall\,\ell\geq2, \quad\forall\,T>0.
\]

\noindent\underline{\sc Classical Bismut-Elworthy-Li formula}.
Let $P^{\eps,n}=(P^{\eps,n}_t)_{t\geq0}$ be the Markov transition 
semigroup associated to the approximated problem above on $H_n$,
namely
\[
P_t^{\eps,n}\varphi(x):= 
\mathbb{E}[ \varphi(X^x_{\eps,n}(t))], \quad x\in H_n,\quad \varphi\in \mathcal{B}_b(H_n).
\]
Since the coefficients are smooth, 
by the It\^o formula it can be shown 
(see \cite[Lem.~7.7.2]{Dap-Zab2002}) that 
for every $\varphi\in C^2_b(H_n)$ and $T>0$, the function 
$v_{\eps,n}:[0,T]\times H_n\to \erre$
given by $v_{\eps,n}(t,x):=P^{\eps,n}_t\varphi(x)$,
$(t,x)\in[0,T]\times H_n$, is the unique classical solution to 
the parabolic Kolmogorov equation
\begin{align*}
\frac{\partial}{\partial t}v_{\eps,n}(t,x)
&-\frac 12 \operatorname{Tr}
\left(\mathfrak G_{\alpha,\eps,n}(x)^*
D^2 v_{\eps,n}(t,x)
\mathfrak G_{\alpha, \varepsilon,n}(x) \right)\\
&+(Ax+D\mathfrak F_{\beta,\eps,n}(x), 
Dv_{\eps,n}(t,x))_{H_n}
=0, \qquad (t,x)\in(0,T)\times H_n,
\end{align*}
with initial condition
\[
v_{\eps,n}(0,x)=\varphi(x), \qquad x \in H_n.
\]
Moreover, by applying the It\^o formula to 
$s \mapsto v_{\eps,n}(t-s,X^x_{\eps,n}(s))$, 
one obtains
\begin{equation}
\label{BEL1}
\varphi(X^x_{\eps,n}(t))
=v_{\eps,n}(t,x)+ 
\int_0^t \left( Dv_{\eps,n}(t-s, X^x_{\eps,n}(s)),
\mathfrak G_{\alpha, \varepsilon,n}(X^x_{\eps,n}(s))
\, \d W(s)\right)_{H_n}
\end{equation}
for every $t\geq0$, $x\in H_n$, and $\varphi\in C^2_b(H_n)$.
Now, we consider the stochastic integral
\begin{equation}
\label{BEL2}
\int_0^t \left( \mathfrak G_{\alpha,\varepsilon,n}^{-1}
(X^x_{\eps,n}(s))[Y^h_{\eps,n}(s)], \d W(s)\right)_H, \quad t\geq0,
\end{equation}
which is well-defined since by \eqref{est:inv_eps} one has
\begin{align*}
\E\int_0^t\|\mathfrak G_{\alpha,\varepsilon,n}^{-1}
(X^x_{\eps,n}(s))[Y^h_{\eps,n}(s)\|_H^2\,\d s
&=\E\int_0^t\|G_{\alpha,\varepsilon}^{-1}
(X^x_{\eps,n}(s))[Y^h_{\eps,n}(s)\|_H^2\,\d s\\
&\lesssim\E\int_0^t\|\Psi''_{\frac\alpha2,\eps}(X^x_{\eps,n}(s))\|_{V_{2\delta}}^2
\|Y^h_{\eps,n}(s)\|_{V_{2\delta}}^2\,\d s\\
&\lesssim_{\eps,n}t
\E\left[\sup_{s\in[0,t]}\|X^x_{\eps,n}(s)\|_{H_n}^2
\sup_{s\in[0,t]}\|Y^h_{\eps,n}(s)\|_{H_n}^2\,\d s\right]<+\infty.
\end{align*}
By multiplying both members of \eqref{BEL1} by the term \eqref{BEL2}, by taking expectations we obtain the equality 
\begin{align*}
&\mathbb{E}\left[\varphi(X^x_{\eps,n}(t))
\int_0^t \left( \mathfrak G_{\alpha,\varepsilon,n}^{-1}(X^x_{\eps,n}(s))Y^h_{\eps,n}(s), \d W(s)\right)_H \right]\\
&=\mathbb{E}\left[v_{\eps,n}(t,x) 
\int_0^t \left( \mathfrak G_{\alpha,\varepsilon,n}^{-1}(X^x_{\eps,n}(s))Y^h_{\eps,n}(s), \d W(s)\right)_H \right]
\\
& \qquad +\mathbb{E}\left[ \int_0^t \left( Dv_{\eps,n}(t-s, X^x_{\eps,n}(s)), \mathfrak G_{\alpha, \varepsilon,n}(X^x_{\eps,n}(s))\, \d W(s)\right)_{H_n}\right.\\
&\qquad\qquad\left.\times\int_0^t \left( \mathfrak G_{\alpha,\varepsilon,n}^{-1}(X^x_{\eps,n}(s))Y^h_{\eps,n}(s), \d W(s)\right)_H\right].
\end{align*}
On the one hand, we clearly have 
\[
\mathbb{E}\left[v_{\eps,n}(t,x) 
\int_0^t \left( \mathfrak G_{\alpha,\varepsilon,n}^{-1}(X^x_{\eps,n}(s))Y^h_{\eps,n}(s), \d W(s)\right)_H \right]=0.
\]
On the other hand, recalling the right-inverse property
$\mathfrak G_{\alpha,\eps,n}(x)
\mathfrak G_{\alpha,\eps,n}^{-1}(x)=\operatorname{I}_{H_n}$
for every $x\in H_n$, 
we have
\begin{align*}
&\mathbb{E}\left[ \int_0^t \left( Dv_{\eps,n}(t-s, X^x_{\eps,n}(s)), \mathfrak G_{\alpha, \varepsilon,n}(X^x_{\eps,n}(s))\, \d W(s)\right)_{H_n}\int_0^t \left( \mathfrak G_{\alpha,\varepsilon,n}^{-1}(X^x_{\eps,n}(s))Y^h_{\eps,n}(s), \d W(s)\right)_H\right]
\\
&=\mathbb{E}\int_0^t \left( 
\mathfrak G_{\alpha, \varepsilon,n}(X^x_{\eps,n}(s))^*
Dv_{\eps,n}(t-s, X^x_{\eps,n}(s)), 
\mathfrak G_{\alpha,\varepsilon,n}^{-1}(X^x_{\eps,n}(s))Y^h_{\eps,n}(s)\right)_{H}\, \d s 
\\
&=\mathbb{E}\int_0^t \left( 
Dv_{\eps,n}(t-s, X^x_{\eps,n}(s)), 
\mathfrak G_{\alpha, \varepsilon,n}(X^x_{\eps,n}(s))\mathfrak G_{\alpha,\varepsilon,n}^{-1}(X^x_{\eps,n}(s))Y^h_{\eps,n}(s)\right)_{H_n}\, \d s 
\\
&=\mathbb{E}\int_0^t \left( 
Dv_{\eps,n}(t-s, X^x_{\eps,n}(s)), 
Y^h_{\eps,n}(s)\right)_{H_n}\, \d s 
\\
&=\int_0^t D\left(\mathbb{E}\left[v_{\eps,n}(t-s, X^x_{\eps,n}(s))\right]\right)[h]\, \d s
=\int_0^t D(P^{\eps,n}_{t-s}P^{\eps,n}_s\varphi(x))[h]\, \d s = tD(\ve(t,x))[h].
\end{align*}
Upon rearranging the terms we obtain 
the approximated Bismut-Elworthy-Li formula
\[
D(P^{\eps,n}_t\varphi)(x)[h]=
\frac 1t \mathbb{E}\left[ \varphi(X^x_{\eps,n}(t))
\int_0^t\left( \mathfrak G_{\alpha,\varepsilon,n}^{-1}(X^x_{\eps,n}(s))Y^h_{\eps,n}(s), \d W(s)\right)_H \right]
\quad\forall\,x,h\in H_n,
\]
for every $\varphi\in C^2_b(H_n)$ and $t>0$.

\noindent\underline{\sc Passage to the limit}.
We pass here to the limit as $n\to\infty$. 
Let $\varphi\in C^2_b(H)$, $x\in V_{2\xi}$,
$h\in V_{2\delta}$,
and set $\varphi_n:=\varphi\circ P_n\in C^2_b(H_n)$,
$x_n:=P_nx\in H_n$, and $h_n:=P_n h\in H_n$.
We know from the  Bismut-Elworthy-Li formula proved above that
for every $\theta_0\geq0$,
\begin{align}
\notag
    &P^{\eps,n}_t\varphi_n(x_n+\theta_0 h_n)-P^{\eps,n}_t\varphi_n(x_n)\\
\label{aux:9}
    &=\int_0^{\theta_0}
    \frac{1}t \mathbb{E}\left[ \varphi(X^{x_n+\theta h_n}_{\eps,n}(t))
\int_0^t\left( \mathfrak G_{\alpha,\varepsilon,n}^{-1}(X^{x_n+\theta h_n}_{\eps,n}(s))Y^{h_n}_{\eps,n}(s), \d W(s)\right)_H \right]
\,\d\theta.
\end{align}
Now, since $x_n+\theta h_n\to x+\theta h$ in $V_{2\xi}$
and $h_n\to h$ in $V_{2\delta}$,
for every $\theta\in[0,\theta_0]$, classical computations 
based on finite dimensional approximations ensure that 
\begin{alignat*}{2}
  X^{x_n+\theta h_n}_{\eps,n}&\to X^{x+\theta h}_{\eps}
  \quad &&\text{in } L^\ell_\cP(\Omega; C([0,T]; V_{2\xi})\cap L^2(0,T; V_{2\sigma+1}))
  \quad\forall\,\ell\geq2,\quad\forall\,T>0,\\
  Y^{h_n}_{\eps,n}&\to Y^{h}_{\eps}
  \quad&&\text{in } L^\ell_\cP(\Omega; C([0,T]; V_{2\delta})\cap L^2(0,T; V_{2\sigma+1}))
  \quad\forall\,\ell\geq2,\quad\forall\,T>0.
\end{alignat*}
Hence, for every $\theta\in[0,\theta_0]$,
by the dominated convergence theorem we have 
\begin{align*}
    P^{\eps,n}_t\varphi_n(x_n+\theta h_n)
    &=\E\left[\varphi_n(X^{x_n+\theta h_n}_{\eps,n}(t))\right]\\
    &=\E\left[\varphi(X^{x_n+\theta h_n}_{\eps,n}(t))\right]
    \to \E\left[\varphi(X^{x+\theta h}_{\eps}(t))\right]
    =P^{\eps}_t\varphi(x+\theta h).
\end{align*}
Furthermore, note that, by definition of $\mathfrak G_{\alpha,\eps,n}^{-1}$, $G_{\alpha,\eps}^{-1}$, and by \eqref{est:inv_eps}, we have
\begin{align*}
    &\norm{\mathfrak G_{\alpha,\varepsilon,n}^{-1}(X^{x_n+\theta h_n}_{\eps,n})Y^{h_n}_{\eps,n}
    - G_{\alpha,\varepsilon}^{-1}(X^{x+\theta h}_{\eps})Y^{h}_{\eps}}_{L^2(0,t;H)}^2\\
    &=\norm{G_{\alpha,\varepsilon}^{-1}(X^{x_n+\theta h_n}_{\eps,n})Y^{h_n}_{\eps,n}- 
    G_{\alpha,\varepsilon}^{-1}(X^{x+\theta h}_{\eps})Y^{h}_{\eps}}_{L^2(0,t;H)}^2\\
    &\leq2\norm{G_{\alpha,\varepsilon}^{-1}(X^{x_n+\theta h_n}_{\eps,n})[Y^{h_n}_{\eps,n}- 
    Y^{h}_{\eps}]}_{L^2(0,t;H)}^2
    +2\norm{(G_{\alpha,\varepsilon}^{-1}(X^{x_n+\theta h_n}_{\eps,n})- 
    G_{\alpha,\varepsilon}^{-1}(X^{x+\theta h}_{\eps}))Y^{h}_{\eps}}_{L^2(0,t;H)}^2\\
    &\lesssim
    \int_0^t\norm{\Psi''_{\frac\alpha2,\eps}
    (X^{x_n+\theta h_n}_{\eps,n}(s))}_{V_{2\delta}}^2
    \norm{Y^{h_n}_{\eps,n}(s)- 
    Y^{h}_{\eps}(s)}_{V_{2\delta}}^2\,\d s\\
    &\qquad+\int_0^t\norm{\Psi''_{\frac\alpha2,\eps}
    (X^{x_n+\theta h_n}_{\eps,n}(s))-
    \Psi''_{\frac\alpha2,\eps}
    (X^{x+\theta h}_{\eps}(s))}_{V_{2\delta}}^2
    \norm{Y^{h}_{\eps}(s)}_{V_{2\delta}}^2\,\d s\\
    &\lesssim 
    \norm{\Psi''_{\frac\alpha2,\eps}
    (X^{x_n+\theta h_n}_{\eps,n})}_{L^2(0,t;V_{2\delta})}^2
    \norm{Y^{h_n}_{\eps,n}- 
    Y^{h}_{\eps}}_{C([0,T];V_{2\delta})}^2\\
    &\qquad+
    \norm{\Psi''_{\frac\alpha2,\eps}
    (X^{x_n+\theta h_n}_{\eps,n})-
    \Psi''_{\frac\alpha2,\eps}
    (X^{x+\theta h}_{\eps})}_{L^2(0,t; V_{2\delta})}^2
    \norm{Y^{h}_{\eps}}_{C([0,T];V_{2\delta})}^2.
\end{align*}
If $\delta<\frac12$, since $\Psi_{\frac\alpha2,\eps}''$
is Lipschitz-continuous, one has that 
\begin{align*}
  \norm{\Psi''_{\frac\alpha2,\eps}
    (X^{x_n+\theta h_n}_{\eps,n})}_{L^2(0,t;V_{2\delta})}^2
  &\lesssim_\eps \norm{X^{x_n+\theta h_n}_{\eps,n}}_{L^2(0,t;V_{2\delta})}^2,\\
  \norm{\Psi''_{\frac\alpha2,\eps}
    (X^{x_n+\theta h_n}_{\eps,n})-
    \Psi''_{\frac\alpha2,\eps}
    (X^{x+\theta h}_{\eps})}_{L^2(0,t; V_{2\delta})}^2
    &\lesssim_{\eps}
    \norm{X^{x_n+\theta h_n}_{\eps,n}-
    X^{x+\theta h}_{\eps}}_{L^2(0,t; V_{2\delta})}^2,
\end{align*}
so the convergences above imply that
\[
\mathfrak G_{\alpha,\varepsilon,n}^{-1}(X^{x_n+\theta h_n}_{\eps,n})Y^{h_n}_{\eps,n}
    \to 
    G_{\alpha,\varepsilon}^{-1}(X^{x+\theta h}_{\eps})Y^{h}_{\eps}
\quad\text{in } L^2_\cP(\Omega; L^2(0,t; H)) \quad\forall\,\ell\geq2.
\]
If $\delta\in[\frac12,1)$, we need a more careful examination.
Since $\xi>\max\{\delta-\frac12,
\delta-\sigma-\frac14\}$ by assumption,
we can choose $z\in[\delta-\frac12,\frac12)
\cap (\delta-\sigma-\frac 14,\xi]$,
so that Lemma~\ref{multiplication} yields
\begin{align*}
  &\norm{\Psi''_{\frac\alpha2,\eps}
    (X^{x_n+\theta h_n}_{\eps,n})}_{L^2(0,t;V_{2\delta})}^2
  \lesssim_\eps \norm{X^{x_n+\theta h_n}_{\eps,n}}_{L^2(0,t;V)}^2+
  \norm{\Psi_{\frac\alpha2,\eps}'''(X^{x_n+\theta h_n}_{\eps,n})
  \nabla X^{x_n+\theta h_n}_{\eps,n}}_{L^2(0,t; V_{2(\delta-\frac12)})}^2\\
  &\qquad\lesssim\norm{X^{x_n+\theta h_n}_{\eps,n}}_{L^2(0,t;V)}^2+
  \int_0^t
  \norm{\Psi_{\frac\alpha2,\eps}'''(X^{x_n+\theta h_n}_{\eps,n}(s))}_{V_{2z}}^2
  \norm{\nabla X^{x_n+\theta h_n}_{\eps,n}(s)}_{V_{2\sigma}}^2\,\d s\\
  &\qquad\lesssim_\eps 
  \norm{X^{x_n+\theta h_n}_{\eps,n}}_{L^2(0,t;V)}^2+
  \norm{X^{x_n+\theta h_n}_{\eps,n}}_{C([0,t]; V_{2\xi})}^2
  \norm{X^{x_n+\theta h_n}_{\eps,n}}_{L^2(0,t;V_{2\sigma+1})}^2.
\end{align*}
Analogous computations show that 
\begin{align*}
    &\norm{\Psi''_{\frac\alpha2,\eps}
    (X^{x_n+\theta h_n}_{\eps,n})-
    \Psi''_{\frac\alpha2,\eps}
    (X^{x+\theta h}_{\eps})}_{L^2(0,t; V_{2\delta})}^2\\
    &\lesssim
    \norm{\Psi'''_{\frac\alpha2,\eps}
    (X^{x_n+\theta h_n}_{\eps,n})
    \nabla X^{x_n+\theta h_n}_{\eps,n}-
    \Psi'''_{\frac\alpha2,\eps}
    (X^{x+\theta h}_{\eps})
    \nabla X^{x+\theta h}_{\eps}}_{L^2(0,t; H)}^2\\
    &\qquad+
    \norm{X^{x_n+\theta h_n}_{\eps,n}}_{C([0,t]; V_{2\xi})}^2
  \norm{X^{x_n+\theta h_n}_{\eps,n}-X^{x+\theta h}_{\eps}}_{L^2(0,t;V_{2\sigma+1})}^2\\
  &\qquad+\norm{\Psi'''_{\frac\alpha2,\eps}
    (X^{x_n+\theta h_n}_{\eps,n})
    -\Psi'''_{\frac\alpha2,\eps}
    (X^{x+\theta h}_{\eps})}_{C([0,t]; V_{2z})}^2
    \norm{X^{x+\theta h}_{\eps}}_{L^2(0,t;V_{2\sigma+1})}^2.
\end{align*}
Hence, by exploiting  
the fact that $\Psi_{\frac\alpha2,\eps}'''$ is bounded and Lipschitz-continuous, the dominated convergence theorem,
and the convergences pointed out above, we infer 
also in this case that 
\[
\mathfrak G_{\alpha,\varepsilon,n}^{-1}(X^{x_n+\theta h_n}_{\eps,n})Y^{h_n}_{\eps,n}
    \to 
    G_{\alpha,\varepsilon}^{-1}(X^{x+\theta h}_{\eps})Y^{h}_{\eps}
\quad\text{in } L^2_\cP(\Omega; L^2(0,t; H)) \quad\forall\,\ell\geq2.
\]
Putting everything together, we can let $n\to\infty$ 
in \eqref{aux:9} by using the It\^o isometry 
and the dominated convergence theorem. This concludes the proof.
\end{proof}

We are now ready to prove uniform estimate on the derivatives
$(DP^\eps_t)_\eps$, for every $t>0$.

\begin{prop}
\label{prop:est_BEL}
Under the assumptions of Theorem~\ref{th:sf},
let $\varphi\in C^2_b(H)$ and $T>0$.
Then, there exists a constant
$C>0$, 
only depending on $\alpha,\beta,\delta,\sigma,\xi,T$,
such that, for every $t\in(0,T)$ and
$\eps\in(0,1)$, it holds
\begin{align*}
\norm{D(P^\eps_t\varphi)(x)}_{V_{2\delta}^*}
\leq \frac{C}{t^{\mathfrak a}}
\left(1+\|x\|^{\mathfrak b}_{V_{2\sigma}}+
\|\Psi_{\gamma}(x)\|_{L^1(\Td)}^{\mathfrak b}\right)
\sup_{v\in H}|\varphi(v)|
\quad\forall\,x\in V_{2\xi}\cap\mathcal K_\gamma.
\end{align*}
\end{prop}
\begin{proof}
    First of all, let us note that the assumptions of 
    Theorem~\ref{th:sf} imply that all the assumptions 
    of Theorem~\ref{th:wp} and Propositions~\ref{prop_stime_Y_3},
    \ref{prop_stime_Y_4}, and \ref{prop:Xe_inverse} 
    are satisfied (see Remarks~\ref{rem_max_reg}--\ref{rem_max_reg2}).
    By Proposition~\ref{BEL3}, It\^o isometry,
    and the H\"older inequality,
    for every $x\in V_{2\sigma}$ and $h\in V_{2\delta}$
    we have that 
    \begin{align*}
        |D(P_t^\eps\varphi)(x)[h]|^2
        &\leq
        \frac1{t^2}
        \sup_{v\in H}|\varphi(v)|^2
        \E\int_0^t\norm{G_{\alpha,\eps}^{-1}
        (\xex(s))Y_\eps^h(s)}_H^2\,\d s\\
        &\leq\frac1{t^2}
        \sup_{v\in H}|\varphi(v)|^2
        \E\int_0^t
        \norm{G_{\alpha,\eps}^{-1}
        (\xex(s))}^2_{\cL(V_{2\delta}, H)}
        \norm{Y_\eps^h(s)}_{V_{2\delta}}^2\,\d s\\
        &\leq\frac1{t^2}
        \sup_{v\in H}|\varphi(v)|^2
        \E\left[
        \norm{G_{\alpha,\eps}^{-1}(\xex)}_{L^2(0,t;\cL(V_{2\delta},H))}^2
        \norm{Y_\eps^h}_{C([0,t]; V_{2\delta})}^2
        \right]\\
        &\leq\frac1{t^2}
        \sup_{v\in H}|\varphi(v)|^2
        \norm{G_{\alpha,\eps}^{-1}(\xex)}_{L^4(\Omega; L^2(0,t;\cL(V_{2\delta},H)))}^2
        \norm{Y_\eps^h}_{L^4(\Omega; C([0,T]; V_{2\delta}))}^2,
    \end{align*}
    so that 
    \[
    |D(P_t^\eps\varphi)(x)[h]|
    \leq\frac1t
    \sup_{v\in H}|\varphi(v)|
    \norm{G_{\alpha,\eps}^{-1}(\xex)}_{L^4(\Omega; L^2(0,t;\cL(V_{2\delta,H})))}
        \norm{Y_\eps^h}_{L^4(\Omega; C([0,T]; V_{2\delta}))}.
    \]
    By letting $r$ as in Proposition~\ref{prop:Xe_inverse}, 
    by noting that $\gamma$ in Theorem~\ref{th:sf}
    is admissible in Proposition~\ref{prop:Xe_inverse},
    by the H\"older inequality and 
    Proposition~\ref{prop:Xe_inverse} we have 
    \begin{align*}
    \norm{G_{\alpha,\eps}^{-1}(\xex)}_{L^4(\Omega; L^2(0,t;\cL(V_{2\delta,H})))}
    &\leq t^{\frac12-\frac1r}
    \norm{G_{\alpha,\eps}^{-1}(\xex)}_{L^4(\Omega; L^r(0,T;\cL(V_{2\delta,H})))}\\
    &\leq Ct^{\frac12-\frac1r}\left(1+\|x\|_{V_{2\xi}}^{\mathfrak p}
    +\|\Psi_{\gamma}(x)\|_{L^1(\Td)}^{\mathfrak p}\right).
    \end{align*}
    Moreover, by noting that $\gamma$ in Theorem~\ref{th:sf}
    always dominates the choices $\alpha+\beta$
    and $\alpha+2\beta+2$ appearing in 
    Remark~\ref{rem_max_reg2}, by 
    Propositions~\ref{prop_stime_Y_3}-\ref{prop_stime_Y_4}
    we infer
    \[
    \norm{Y_\eps^h}_{L^4(\Omega; C([0,T]; V_{2\delta}))}
    \leq C\left(1+\|x\|_{V_{2\sigma}}^{\mathfrak q}
    +\|\Psi_{\alpha+\beta}(x)\|_{L^1(\Td)}^{\mathfrak q}\right)
    \|h\|_{V_{2\delta}},
    \]
    where $\mathfrak q:=2\sigma$ if $\delta<\frac12$ and $\sigma<\frac14$, 
    $\mathfrak q:=\frac32+2\sigma$ if $\delta\geq\frac12$
    and $\sigma<\frac14$, and 
    $\mathfrak q:=\frac32+\frac\beta{2(\gamma-2)}$
    if $\delta\geq\frac12$
    and $\sigma\geq\frac14$.
    Putting everything together yields
    \[
    |D(P_t^\eps\varphi)(x)[h]|
    \leq\frac{C}{t^{\frac12+\frac1r}}
    \sup_{v\in H}|\varphi(v)|
    \left(1+\|x\|_{V_{2\sigma}}^{\mathfrak p+\mathfrak q}
    +\|\Psi_{\gamma}(x)\|_{L^1(\Td)}^{\mathfrak p+\mathfrak q}\right)
    \|h\|_{V_{2\delta}}
    \]
    and the thesis follows since $\frac12+\frac1r=\mathfrak a$, $\mathfrak q\leq2$,
    and 
    $\mathfrak p+\mathfrak q\leq\mathfrak b$.
\end{proof}

\subsection{Proof of the strong Feller property}
We are now ready to conclude the proof of the main Theorem~\ref{th:sf}.
Under the assumptions of Theorem~\ref{th:sf},
given $\varphi\in\mathcal B_b(\mathcal X)$, let us define $\tilde\varphi\in\mathcal B_b(H)$
by extending $\varphi$ to $0$ outside $\mathcal X$, namely 
\[
  \tilde\varphi(x):=
  \begin{cases}
      \varphi(x) \quad&\text{if } x\in \mathcal X,\\
      0 \quad&\text{if } x\in H\setminus\mathcal X.
  \end{cases}
\]
Now, by relying on classical 
approximation techniques via
convolutions with Gaussian measures on $H$, it is clear that there exists a sequence $(\varphi_m)_m\subset C^2_b(H)$
such that 
\[
\lim_{m\to\infty}\varphi_m(x)\to\tilde\varphi(x) \quad\forall\,x\in H
\]
and 
\[
  \sup_{v\in H}|\varphi_m(v)|\leq\sup_{v\in H}|\tilde\varphi(v)|
  =\sup_{v\in \mathcal X}|\varphi(v)|
  \quad\forall\,m\in\enne.
\]
Given $x_1,x_2\in V_{2\xi}\cap \mathcal K_{\gamma}$ with $x_1-x_2\in V_{2\delta}$, by convexity one has that 
$x_2+\theta(x_1-x_2)\in V_{2\xi}\cap \mathcal K_{\gamma}$
for every $\theta\in[0,1]$: hence, 
by Proposition~\ref{prop:est_BEL} we have 
\begin{align*}
&|P^\eps_t\varphi_m(x_1)-P_t^\eps\varphi_m(x_2)|
=\left|\int_0^1D(P^\eps_t\varphi_m)(x_2+\theta(x_1-x_2))[x_1-x_2]\,\d\theta\right|\\
&\leq\int_0^1
\norm{D(P^\eps_t\varphi_m)(x_2+\theta(x_1-x_2))}_{V_{2\delta}^*}
\|x_1-x_2\|_{V_{2\delta}}\,\d\theta\\
&\leq\frac{C\|x_1-x_2\|_{V_{2\delta}}}{t^{\mathfrak a}}
\sup_{v\in H}|\varphi_m(v)|
\left(1+\|x_2+\theta(x_1-x_2)\|_{V_{2\sigma}}^{\mathfrak b}+
\|\Psi_{\gamma}(x_2+\theta(x_1-x_2))\|_{L^1(\Td)}^{\mathfrak b}\right)\\
&\leq\frac{C}{t^{\mathfrak a}}
\sup_{v\in \mathcal X}|\varphi(v)|
\sum_{i=1}^2
    \left(1+\|x_i\|_{V_{2\sigma}}^{\mathfrak b}
    +\|\Psi_{\gamma}(x_i)\|_{L^1(\Td)}^{\mathfrak b}\right)
\|x_1-x_2\|_{V_{2\delta}}.
\end{align*}
Now, since $\P(X(t)\in \mathcal X)=1$,
by letting $\eps\to0^+$ and then $m\to\infty$ via the dominated convergence theorem one has, for $i=1,2$, that
\begin{align*}
  \lim_{m\to\infty}\lim_{\eps\to0^+}
  P^\eps_t\varphi_m(x_i)&=
\lim_{m\to\infty}\lim_{\eps\to0^+}
\E[\varphi_m(X^{x_i}_\eps(t))]
  =\lim_{m\to\infty}\E[\varphi_m(X^{x_i}(t))]\\
  &=\E[\tilde\varphi(X^{x_i}(t))]
  =\E[\varphi(X^{x_i}(t))]=P_t\varphi(x_i),
\end{align*}
and this concludes the proof of Theorem~\ref{th:sf}.


\section{Irreducibility and uniqueness of the invariant measure}
\label{sec:irr}
This section is devoted to the proof of Theorems~\ref{th:irr} and \ref{th:uniq_im} on irreducibility and uniqueness of the invariant measure.

\subsection{Irreducibility}
\label{ssec:irr}
Let us work here under 
the assumptions of Theorem~\ref{th:irr}.
Since $\gamma\geq2\beta$, we know from Theorem~\ref{th:wp}
that 
\begin{equation}
    \label{est:X_R}
  X\in L^\ell_\cP(\Omega; C([0,T]; V_{2\delta}))
  \quad\forall\,\ell\geq2,\quad\forall\,T>0.
\end{equation}

Now, given $t,r>0$ and $a\in \mathcal K_{\gamma}\cap V_{2\delta}$
as in Theorem~\ref{th:irr},
let $\tilde a\in \mathcal K_{\gamma}\cap V_{2+2\delta}$
be such that 
$\|a-\tilde a\|_{V_{2\delta}}<r/2$
(note that such $\tilde a$ always exists by convexity of $\Psi_{\gamma}$ and classical density results).
For $\tau\in(0,t)$ and $M>0$ to be chosen later, 
we consider the modified 
equation
\[
  \begin{cases}
  \d Z(s)+AZ(s)\,\d s + \Psi_\beta'(Z(s))\,\d s
  +\frac{M}{t-\tau}(Z-\tilde a)\,\d s
  =G_\alpha(Z(s))\,\d W(s),
  \quad s\in[\tau,t],\\
  Z(\tau)=X^x(\tau).
  \end{cases}
\]
A direct
adaptation of the proof of Theorem~\ref{th:wp}
ensures that, for every $\tau\in(0,t)$ and $M>0$, 
there exists a
unique variational solution
\begin{equation}
\label{est:Z_R}
Z\in L^\ell_\cP(\Omega; C([\tau,t]; V_{2\delta})\cap L^2(\tau,t; V_{2\sigma+1}))
\quad\forall\,\ell\geq2.
\end{equation}
In order to find a right scaling between 
$\tau$ and $M$, we need to prove some estimates on $Z$, by 
keeping track of the dependence on $\tau$ and $M$.
The idea will be to choose $M$ sufficiently large and
$t-\tau$ sufficiently small: hence, it is not 
restrictive to assume $t-\tau<1$.
\begin{lem}
    \label{lem:Z0}
    In the current setting, 
    there exist positive 
    constants $(C_{1,\ell})_{\ell\geq1}$,
    independent of $\tau$ and $M$,
    such that, for every $\ell\geq2$,
    \begin{align}
    \label{est:Z2}
    \|Z\|_{L^\ell(\Omega;L^\infty(\tau,t;H))}
    +\|Z\|_{L^\ell(\Omega;L^2(\tau,t;V))}
    +\|Z\|_{L^\ell(\Omega;L^{\frac1\sigma}(\tau,t;V_{2\sigma}))}
    &\leq C_{1,\ell}(1+\sqrt{M}),\\
    \label{est:Z4}
    \|\Psi_\gamma(Z)\|_{L^{\ell}(\Omega;L^\infty(\tau,t;
    L^1(\Td)))}
    +\|\Psi''_{\frac\gamma2}(Z)\nabla Z\|^2_{
    L^\ell(\Omega;L^2(\tau,t;H))}
    &\leq C_{1,\ell}(1+M).
    \end{align}
    Moreover, 
    \begin{alignat}{2}
    \label{est:Z4ter}
    \|\Psi_\beta'(Z)\|_{L^{\ell}(\Omega;L^2(\tau,t;
    H))}
    &\leq C_{1,\ell}(1+M)\quad&&\text{if } 
    \delta<\frac12,\\
    \label{est:Z4bis}
    \|\Psi_\beta'(Z)\|_{L^{\ell}(\Omega;L^{\frac1\sigma}(\tau,t;
    V_{2\sigma}))}
    &\leq C_{1,\ell}(1+M)\quad&&\text{if } \delta\geq\frac12,
\end{alignat}
and
\begin{equation}
    \label{est:Z5}
  \P\left(\int_\tau^t\|\Psi_{\frac\alpha2}''
  (Z(s))\|^2_{V_{2\delta}}\,\d s<+\infty\right)=1.
\end{equation}
\end{lem}
\begin{proof}
    It\^o formula implies, for every $s\in[\tau,t]$, $\mathbb{P}$-a.s., that
\begin{align*}
    &\frac 12 \|Z(s)\|^2_{H} + 
    \int_\tau^s \| A^{\frac12}Z(r)\|^2_{H} \, \d r
    + \int_\tau^s \left(  
    \Psi_{\beta}'(Z(r)),
    Z(r)\right) \, \d r 
    \\
    &\qquad+\frac{M}{t-\tau}\int_\tau^t
    \left(Z(r)-\tilde a, Z(r)\right)_H\,\d r\\
    &= \frac 12 \|X^x(\tau)\|^2_{H} 
    + \int_\tau^s \left( Z(r), 
    G_{\alpha} (Z(r)) \, \d W(r)\right)_H  
    + \frac 12 \int_\tau^s 
    \|G_{\alpha}(Z(r))
    \|^2_{\cL^2(H,H)}\, \d r.
\end{align*}
On the left-hand side, arguing as in Proposition~\ref{prop1Xeps}
one has by monotonicity of $\Psi_\beta'$ that
\[
\int_\tau^s \left( 
    \Psi_{\beta}'(Z(r)), Z(r)\right) \, \d r\geq0,
\]
while by the Young inequality 
we have 
\begin{align*}
&\frac{M}{t-\tau}\int_\tau^s
    \left(Z(r)-\tilde a,  Z(r)\right)_H\,\d r\\
&\geq\frac{M}{t-\tau}\int_\tau^s
    \|Z(r)\|_H^2\,\d r-
    \frac{M}{t-\tau}\int_\tau^s
    \|\tilde a\|_H\|Z(r)\|_H\,\d r\\
&\geq\frac{M}{2(t-\tau)}\int_\tau^s
    \|Z(r)\|_H^2\,\d r-M\|\tilde a\|_{H}^2.
\end{align*}
By proceeding now as in the 
proof of Proposition~\ref{prop1Xeps}
and using that $|Z|\leq1$,
we deduce that 
\begin{align*}
    &\E\sup_{r\in[\tau,s]}\|Z(r)\|_H^2
    +\E\int_\tau^s\| A^{\frac12}Z(r)\|^2_{H} \, \d r
    +\frac{M}{t-\tau}\int_\tau^s
    \|Z(r)\|_H^2\,\d r\\
    &\lesssim_t 1+\E\|X^x(\tau)\|^2_{H}+
    M\|\tilde a\|_{H}^2 
    \lesssim_t 1 + M,
\end{align*}
and \eqref{est:Z2} follows by interpolation.
Analogously, by formally replicating the estimate of Proposition~\ref{prop3Xeps}, by the Fenchel-Young inequality and the convexity of $\Psi_\gamma$
we have, for every $s\in[\tau,t]$, that
\begin{align*}
&\norm{\Psi_{\gamma}(Z(s))}_{L^1(\Td)}
+\int_\tau^s \norm{\Psi_{\frac\gamma2}''(Z(r)) 
\nabla Z(r)}_H^2\, \d r
+ \int_\tau^s \int_{\Td}\Psi_{\gamma}'(Z(r))
\Psi_{\beta}'(Z(r))\, \d r 
\\
&\qquad+\frac{M}{t-\tau}\int_\tau^s\|
\Psi_\gamma(Z(r))\|_{L^1(\Td)}\,\d r\\
&\leq \norm{\Psi_{\gamma}(X^x(\tau))}_{L^1(\Td)}
+M\norm{\Psi_\gamma(\tilde a)}_{L^1(\Td)}
+ \int_\tau^s \left(\Psi_{\gamma}'(Z(r)), 
G_{\alpha} (Z(r))\, \d W(r)\right)_H \\
&\qquad+
\frac 12 \int_\tau^s \sum_{k \in \mathbb{Z}^d}
\int_{\Td} \Psi_{\gamma}''(Z(r))
m_{\alpha}(Z(r))
|(\operatorname{I}+A)^{-\delta}e_k|^2\, \d r.
\end{align*}
By arguing exactly as in the proof of
Proposition~\ref{prop3Xeps}, we infer that 
\begin{align*}
    &\E\sup_{r\in[\tau,s]}\norm{\Psi_{\gamma}(Z(r))}_{L^1(\Td)}
+\E\int_\tau^s \norm{\Psi_{\frac\gamma2}''(Z(r)) 
\nabla Z(r)}_H^2\, \d r
+\frac{M}{t-\tau}\E\int_\tau^s\|
\Psi_\gamma(Z(r))\|_{L^1(\Td)}\,\d r\\
&\lesssim_t
1+\E\norm{\Psi_{\gamma}(X^x(\tau))}_{L^1(\Td)}+M
+\E\int_\tau^s\norm{\Psi_{\gamma}(Z(r))}_{L^1(\Td)}\,\d r\\
&\lesssim
1+\|\Psi_\gamma(x)\|_{L^1(\Td)}+ M
+\E\int_\tau^s\norm{\Psi_{\gamma}(Z(r))}_{L^1(\Td)}\,\d r,
\end{align*}
so that  
\eqref{est:Z4} follows.
Moreover, \eqref{est:Z4ter}--
\eqref{est:Z4bis} follow from \eqref{est:Z2}--\eqref{est:Z4} 
by arguing as in Proposition~\ref{prop5Xeps}.
Finally, \eqref{est:Z5} can be proved
by arguing exactly as in Proposition~\ref{prop:Xe_inverse},
by using \eqref{est:Z_R} and \eqref{est:Z4}
instead of Proposition~\ref{prop5Xeps} and \ref{prop3Xeps},
respectively.
\end{proof}

\begin{lem}
    \label{lem:Z1}
    In the current setting, there exist constants $C_2,C_3, C_4>0$, independent of $\tau$ and $M$,
    such that, for every $s\in[\tau,t]$,
    \begin{align}
    \label{est:Za}
    \E\|Z(s)-\tilde a\|_H^2
    &\leq 
    C_2
    e^{-\frac{M}{t-\tau}(s-\tau)},\\
    \label{est:Za2}
    \E\|Z(s)-\tilde a\|_{V_{2\delta}}^2
    &\leq C_3\left[e^{-\frac{M}{t-\tau}(s-\tau)}
    +(1+M)(t-\tau)^{C_4}
    \right].
\end{align}
\end{lem}
\begin{proof}
First of all,
the It\^o formula implies, for every $s\in[\tau,t]$, that 
\begin{align*}
    &\frac12\|Z(s)-\tilde a\|_H^2
+\int_\tau^s\left(A^{\frac12}Z(r),
A^{\frac12}(Z(r)-\tilde a)\right)_H\,\d r
+\int_\tau^s\int_{\Td}
\Psi_\beta'(Z(r))(Z(r)-\tilde a)\,\d r\\
&\qquad+\frac M{t-\tau}\int_\tau^s\|Z(r)-\tilde a\|^2_H\,\d r\\
&=\frac12\|X^x(\tau)-\tilde a\|_H^2
+\int_\tau^s\left(Z(r)-\tilde a, G_\alpha(Z(r))\,\d W(r)\right)_H+\frac12\int_\tau^s
\|G_\alpha(Z(r))\|^2_{\cL^2(H,H)}\,\d r
\end{align*}
and classical computations based on 
the Young inequality, the convexity of $\Psi_\beta$,
Lemma~\ref{G_reg}, the fact that $|X^x|\leq1$,
and the Burkholder-Davis-Gundy 
inequality yield
\begin{align*}
    &\E\sup_{r\in[\tau,s]}\|Z(r)-\tilde a\|_H^2
+\E\int_\tau^s\|
A^{\frac12}(Z(r)-\tilde a)\|_H^2\,\d r
+\E\int_\tau^s
\|\Psi_\beta(Z(r))\|_{L^1(\Td)}\,\d r\\
&\qquad+\frac M{t-\tau}
\E\int_\tau^s\|Z(r)-\tilde a\|^2_H\,\d r
\lesssim 1+(t-\tau),
\end{align*}
and \eqref{est:Za} follows by the Gronwall lemma.  
Let us estimate now the $V_{2\delta}$-norm. To this end, 
note that 
by adding and subtracting suitable terms we have 
\[
  \d (Z-\tilde a) + \left(A+
  \frac{M}{t-\tau}\operatorname{I}\right)
  (Z-\tilde a)\,\d s
  =-A\tilde a\,\d t
  -\Psi_\beta'(Z)\,\d s + G_\alpha(Z)\,\d W,
\]
yielding
\begin{align*}
    Z(s)-\tilde a
    &=e^{-\frac{M}{t-\tau}(s-\tau)}S(s-\tau)(X^x(\tau)-\tilde a)
    -\int_\tau^se^{-\frac{M(s-r)}{t-\tau}}S(s-r)
    \left[A\tilde a+\Psi_\beta'(Z(r))\right]\,\d r\\
    &+\int_\tau^s
    e^{-\frac{M(s-r)}{t-\tau}}S(s-r)
    G_\alpha(Z(r))\,\d W(r)
    \qquad\forall\,s\in[\tau,t],\quad\P\text{-a.s.}
\end{align*}
By recalling \eqref{est:X_R}, we infer that 
there exists $R>0$, independent of $\tau$ and $M$, such that
\begin{align*}
    &\E\|Z(s)-\tilde a\|_{V_{2\delta}}^2
    \lesssim e^{-\frac{M}{t-\tau}(s-\tau)}
    \left(R^2+
    \|\tilde a\|_{V_{2\delta}}^2\right)
    +\|\tilde a\|_{V_{2+2\delta}}^2(t-\tau)\\
    &\quad+\E\norm{\int_\tau^se^{-\frac{M(s-r)}{t-\tau}}
    S(s-r)\Psi_\beta'(Z(r))\,\d r}_{V_{2\delta}}^2
    +\E\norm{\int_\tau^se^{-\frac{M(s-r)}{t-\tau}}
    S(s-r)G_\alpha(Z(r))\,\d W(r)}_{V_{2\delta}}^2.
\end{align*}
If $\delta<\frac12$ by interpolation and 
deterministic maximal regularity 
one has
\begin{align*}
    \E\norm{\int_\tau^se^{-\frac{M(s-r)}{t-\tau}}
    S(s-r)\Psi_\beta'(Z(r))\,\d r}_{V_{2\delta}}^2
    &\lesssim (t-\tau)^{1-2\delta}\E\int_\tau^t
    \norm{\Psi_\beta'(Z(r))}^2_{H}\,\d r,
\end{align*}
while if $\delta\in[\frac12,\frac12+\sigma)$
by deterministic and 
stochastic maximal regularity and Lemma~\ref{G_reg}
we have
\begin{align*}
    \E\norm{\int_\tau^se^{-\frac{M(s-r)}{t-\tau}}
    S(s-r)\Psi_\beta'(Z(r))\,\d r}_{V_{2\delta}}^2
    &\lesssim \E\int_\tau^t
    \norm{\Psi_\beta'(Z(r))}^2_{V_{2\sigma}}\,\d r\\
    &\leq (t-\tau)^{1-2\sigma}
    \E\norm{\Psi_\beta'(Z)}^2_{L^{\frac1\sigma}(\tau,t; V_{2\sigma})}
\end{align*}
and 
\begin{align*}
    \E\norm{\int_\tau^se^{-\frac{M(s-r)}{t-\tau}}
    S(s-r)G_\alpha(Z(r))\,\d W(r)}_{V_{2\delta}}^2
    &\lesssim\E\int_\tau^t
    \norm{G_\alpha(Z_\tau(r))}^2_{
    \cL^2(H,V_{2\sigma})}\,\d r\\
    &\lesssim\E\int_\tau^t
    \norm{Z_\tau(r)}^2_{
    V_{2\sigma}}\,\d r\\
    &\leq(t-\tau)^{1-2\sigma}
    \E\norm{Z}^2_{L^{\frac1\sigma}(\tau,t; V_{2\sigma})},
\end{align*}
so that \eqref{est:Z4ter}--\eqref{est:Z4bis} yield \eqref{est:Za2} with $C_4=1-2\delta$ if $\delta<\frac12$
and $C_4=1-2\sigma$ if $\delta\geq\frac12$.
\end{proof}

We are now ready to prove the irreducibility result of Theorem~\ref{th:irr}.
For $\tau\in(0,t)$ and $M>0$,
with $M>2(t-\tau)$,
we define the process
\[
\widetilde X(s):=X(s)\ind_{[0,\tau)}(s)
+Z(s)\ind_{[\tau,t]}(s), \qquad s\in[0,t],
\]
so that 
\[
  \widetilde X\in L^\ell_\cP(\Omega; C([0,t]; V_{2\delta}))
  \quad\forall\,\ell\geq2
\]
is the unique solution to 
\[
  \begin{cases}
  \d \widetilde X(s)+
  A\widetilde X(s)\,\d s + 
  \Psi_\beta'(\widetilde X(s))\,\d s
  +\frac{M}{t-\tau}(\widetilde X-\tilde a)
  \ind_{[\tau,t]}(s)\,\d s
  =G_\alpha(\widetilde X(s))\,\d W(s),
  \quad s\in[0,t],\\
  \widetilde X(0)=x.
  \end{cases}
\]
By noting that 
\begin{align*}
\int_\tau^t\norm{G_\alpha^{-1}(\widetilde X(s))
\left[\frac{M}{t-\tau}(\widetilde X(s)-\tilde a)\right]}_{H}^2\,\d s
&\lesssim_{M,\tau}
\int_\tau^t\norm{\Psi''_{\frac\alpha2}(Z(s))}_{V_{2\delta}}^2
\left(\norm{Z(s)}_{V_{2\delta}}^2
+\|\tilde{a}\|_{V_{2\delta}}^2\right)\,\d s\\
&\lesssim \left(1+\|\widetilde X\|^2_{C([0,t];V_{2\delta})}
\right)\norm{\Psi''_{\frac\alpha2}(Z)}^2_{L^2(\tau,t;V_{2\delta})},
\end{align*}
it follows from \eqref{est:Z5} that 
\[
  \P\left(
  \int_\tau^t\norm{G^{-1}(\widetilde X(s))
\left[\frac{M}{t-\tau}(\widetilde X(s)-\tilde a)\right]}_{H}^2\,\d s<+\infty\right)=1.
\]
Hence, the relaxed version of the Girsanov theorem 
contained in \cite{Ferr} (see also \cite[Thm.~7.4]{LS})
ensures that, 
for every choice of $\tau\in(0,t)$ and $M>0$,
the law of $\widetilde X$ is absolutely continuous 
with respect to the law of $X$.
Consequently, to show that 
$\P(\|X(t)-a\|_{V_{2\delta}}<r)>0$, it is sufficient
to find $\bar\tau\in(0,t)$ and $\bar M>0$
such that $\P(\|\widetilde X(t)-a\|_{V_{2\delta}}<r)>0$.
To this end, recalling that $\|\tilde a-a\|_{V_{2\delta}}<r/2$,
by the Markov inequality we have that 
\begin{align*}
  \P\left(\|\widetilde X(t)-a\|_{V_{2\delta}}<r\right)
  &\geq\P\left(\|\widetilde X(t)-\tilde a\|_{V_{2\delta}}<\frac r2\right)
  =1-\P\left(\|Z(t)-\tilde a\|_{V_{2\delta}}\geq\frac r2\right)\\
  &\geq1-\frac4{r^2}\E\|Z(t)-\tilde a\|_{V_{2\delta}}^2,
\end{align*}
so that \eqref{est:Za2} yields
\[
\P\left(\|\widetilde X(t)-a\|_{V_{2\delta}}<r\right)\geq
1-
\frac{4C_3}{r^2}\left[e^{-M}
    +(1+M)(t-\tau)^{C_4}
    \right].
\]
Hence, we first fix $\bar M>0$ large enough such that 
\[
\frac{4C_3}{r^2}e^{-\bar M}<\frac12.
\]
Secondly, we fix $\bar\tau$
(possibly depending on $\bar M$) 
close enough to $t$ such that 
\[
\frac{4C_3}{r^2}
(1+\bar M)(t-\bar \tau)^{C_4}<\frac12.
\]
and this concludes the proof of Theorem~\ref{th:irr}.

\subsection{Uniqueness of the invariant measure}
We prove here Theorem~\ref{th:uniq_im} on existence and uniqueness of invariant measures. To this end, since the existence part is rather classical while the uniqueness is the main novelty of the work, we shall focus in 
particular on the uniqueness part and just give a sketch of the 
existence, by providing suitable references.

First of all, from the energy estimates contained in Propositions~\ref{prop1Xeps}, \ref{prop3Xeps}, and \ref{prop5Xeps},
and from the classical Krylov-Bogoliubov theorem,
by directly adapting the arguments of \cite{SZ} one can readily show that there exists an invariant measure for the semigroup $P$, and that every invariant measure is concentrated on 
$V_{2\sigma+1}$. 
Moreover, from the estimate of Proposition~\ref{prop3Xeps}
one obtains, for every $\gamma$ as in Theorem~\ref{th:wp}
a control on the term
\begin{align*}
\E\int_0^t\int_{\Td}\Psi_{\beta-1}'(X^x(s))
\Psi_{\gamma-1}'(X^x(s))\,\d s
&\gtrsim
\E\int_0^t\int_{\Td}\Psi''_{\gamma+\beta-2}(X^x(s))\,\d s\\
&\gtrsim
\E\int_0^t\int_{\Td}\Psi_{\gamma+\beta}(X^x(s))\,\d s.
\end{align*}
Hence, classical computations concerning the support of invariant measures (see again \cite{SZ}) show that every invariant 
measure is supported on $\mathcal K_{\gamma+\beta}\cap V_{2\sigma+1}$
for every $\gamma$ as in Theorem~\ref{th:wp}.
In particular, recalling that 
$\mathcal K_{\gamma+\beta}\cap V_{2\sigma+1}\subseteq 
\mathcal K_\infty$ if $\gamma+\beta>\frac{2}{4\sigma+1}+2$
(see Corollary~\ref{cor:sep}) and that 
by assumption of Theorem~\ref{th:sf}
one can choose $\gamma$ with $\gamma+\beta>4$, 
since $\frac{2}{4\sigma+1}+2<4$ one gets exactly
that every invariant measure is concentrated on 
$\mathcal K_\infty\cap V_{2\sigma+1}$.

Let us focus on the uniqueness part.
We start by showing that the semigroup $P$ is regular,
i.e.~that for every $t>0$, the transition 
probabilities $(\mu_t(x,\cdot))_{x\in\mathcal K_{\gamma}\cap V_{2\delta}}$
are all equivalent, where $\gamma$ is 
as in Theorem~\ref{th:sf}.
To this end, let $x\in 
\mathcal K_{\gamma}\cap V_{2\delta}$,
$A\in\cB(\mathcal X)$, and suppose that $\mu_t(x,A)>0$: then, by the semigroup property 
and the fact that $\mu_{\frac t2}(x,\mathcal K_{\gamma}\cap V_{2\delta})=1$,
it holds that 
\[
0<\mu_t(x,A)=\int_{\mathcal X}\mu_{\frac t2}(z,A)\,
\mu_{\frac t2}(x,\d z)=
\int_{\mathcal K_{\gamma}\cap V_{2\delta}}
(P_{\frac t2}\ind_A)(z)
\,\mu_{\frac t2}(x,\d z).
\]
It follows that there exists 
$z_0\in\mathcal K_{\gamma}\cap V_{2\delta}$ such that 
$(P_{\frac t2}\ind_A)(z_0)>0$. By the strong Feller property of Theorem~\ref{th:sf}
we infer that there exists $r_0>0$
and $\eps_0>0$
such that, setting $B^{2\delta}_{z_0}(r):=
\left\{z\in V_{2\delta}:\;\|z-z_0\|_{V_{2\delta}}<r_0\right\}$,
it holds that 
$(P_{\frac t2}\ind_A)(z)\geq\eps_0$
for all $z\in \mathcal K_{\gamma}\cap B^{2\delta}_{z_0}(r_0)$.
Hence, for every $y\in\mathcal K_{\gamma}\cap V_{2\delta}$ one has that 
\[
\mu_t(y,A)=\int_{\mathcal X}\mu_{\frac t2}(z,A)\,
\mu_{\frac t2}(y,\d z)=
\int_{\mathcal K_{\gamma}\cap V_{2\delta}}
(P_{\frac t2}\ind_A)(z)
\,\mu_{\frac t2}(y,\d z)
\geq\eps_0\mu_{\frac t2}(y,\mathcal K_{\gamma}\cap B^{2\delta}_{z_0}(r)).
\]
Since 
$\gamma$ satisfies the assumptions of Theorem~\ref{th:irr},
by the irreducibility Theorem~\ref{th:irr} and
the fact that  $\mu_{\frac t2}(y,\mathcal K_{\gamma})=1$ 
it holds that 
\[
\mu_{\frac t2}(y,\mathcal K_{\gamma}\cap B^{2\delta}_{z_0}(r))
=\P\left(\|X^y(t)-z_0\|_{V_{2\delta}}\right)>0.
\]
By the arbitrariness of $x$ and $y$, this shows that $\mu_t(x,\cdot)$ and $\mu_t(y,\cdot)$
are equivalent for all $x,y\in\mathcal K_{\gamma}\cap V_{2\delta}$, as required.

We are now ready to conclude. 
Let $\mu$ be an invariant measure for $P$. 
We first show that $\mu$ is ergodic.
To this end, let $A\in\cB(\mathcal X)$ 
with $P_t\ind_A=\ind_A$
for all $t\geq0$ and $\mu(A)>0$:
since $\mu(\mathcal K_{\gamma}\cap V_{2\delta})=1$,
there exists $x_0\in A\cap \mathcal K_{\gamma}
\cap V_{2\delta}$
and $t_0>0$
such that $\mu_{t_0}(x_0,A)=1$.
It follows by the equivalence of transition probabilities proved above that $\mu_{t_0}(x,A)=1$
for all $x\in \mathcal K_{\gamma}\cap V_{2\delta}$, hence
\[
\mu(A)=\int_{\mathcal X}
\mu_{t_0}(x,A)\,\mu(\d x)
=\int_{\mathcal K_{\gamma}\cap V_{2\delta}}
\mu_{t_0}(x,A)\,\mu(\d x)=1.
\]
By virtue of the 
characterisation result 
\cite[Thm.~3.2.4]{DPZ-erg}, this shows that 
$\mu$ is ergodic.

Given now two invariant measures $\mu$ and $\nu$, 
by the result just proved one has that $\mu$ and $\nu$ are
ergodic. Moreover,  by invariance,
for every $A\in\cB(\mathcal X)$ and $t>0$ it holds that 
\[
\mu(A)
=\int_{\mathcal K_{\gamma}\cap V_{2\delta}}
\mu_{t}(x,A)\,\mu(\d x), \qquad
\nu(A)
=\int_{\mathcal K_{\gamma}\cap V_{2\delta}}
\mu_{t}(x,A)\,\nu(\d x)
\]
It follows that $\mu$ and $\nu$
are both equivalent to every transition probability.
Hence, $\mu$ and $\nu$ are equivalent and ergodic, 
which implies in turn that $\mu=\nu$ by \cite[Prop.~3.2.5]{DPZ-erg}. 
Eventually, the strong mixing property follows as in 
\cite[Thm.~4.2.1]{DPZ-erg}, and this concludes the proof 
of Theorem~\ref{th:uniq_im}.

\section*{Acknowledgements}
The authors are members of Gruppo Nazionale per l'Analisi Matematica, la Probabilit\'a e le loro Applicazioni (GNAMPA), Istituto Nazionale di Alta Matematica (INdAM).
The authors gratefully acknowledge the financial support of the project  ``Prot. P2022TX4FE\_02 -  Stochastic particle-based anomalous reaction-diffusion models with heterogeneous interaction for radiation therapy'' financed by the European Union - Next Generation EU, Missione 4-Componente 1-CUP: D53D23018970001.
The authors also gratefully acknowledge the financial support of the project  ``Equazioni differenziali stocastiche: sviluppi teorici e applicazioni a modelli per fenomeni fisici'' financed by Indam-Gnampa, CUP: E53C25002010001.

\section*{Data availability statement}
No new data were created or analysed in this study. 
Data sharing is not applicable.

\section*{Conflict of interest statement}
The authors have no conflicts of interest to declare.


\bibliographystyle{abbrv}\bibliography{ref}


\end{document}